\documentclass{amsart}[10pt]
\usepackage{amssymb,amsmath,amscd,xy,graphicx}

\newtheorem{theorem}{Theorem}[section]
\newtheorem{lemma}[theorem]{Lemma}
\newtheorem{corollary}[theorem]{Corollary}
\newtheorem{definition}[theorem]{Definition}

\newtheorem{remark}[theorem]{\it Remark}

\newtheorem{proposition}[theorem]{Proposition}

\numberwithin{equation}{subsection}

\xyoption{arrow}
\xyoption{matrix}

\def\tree{\mathcal{T}}

\def\C{\mathbb{C}}
\def\R{\mathbb{R}}
\def\Z{\mathbb{Z}}
\def\N{\mathbb{N}}
\def\T{\mathbb{T}}

\def\br{\mathbf{r}}

\def\bg{\mathbf{g}}

\def\q{/\!/}
\def\ql{\backslash \! \backslash}

\def\cq{/\!/}

\def \calE {\mathcal{E}}

\begin{document}

\title{The toric geometry of triangulated polygons in Euclidean space}

\author{Benjamin Howard, Christopher Manon and John Millson}
\thanks{B. Howard was supported by NSF grant DMS-0405606 and NSF fellowship DMS-0703674.
C. Manon was supported by NSF FRG grant DMS-0554254, and J. Millson  
was supported by NSF grant DMS-0405606, NSF FRG grant DMS-0554254, and The Simons Foundation.}
\date{\today}

\begin{abstract}
Speyer and Sturmfels \cite{SpeyerSturmfels} associated Gr\"obner toric
degenerations $\mathrm{Gr}_2(\C^n)^{\tree}$ of $\mathrm{Gr}_2(\C^n)$
to each
trivalent tree $\tree$ with $n$ leaves. These degenerations
induce toric
degenerations  $M_{\br}^{\tree}$ of $M_{\br}$, the
space of $n$ ordered, weighted (by $\br$) points on the projective line. 
Our goal in this paper is to give a
geometric (Euclidean polygon) description of the toric fibers as
stratified symplectic spaces and describe the action of the 
compact part of the torus
as ``bendings of polygons.''
We prove the conjecture of Foth and Hu \cite{FothHu} that
the toric fibers are homeomorphic
to the spaces defined by Kamiyama and Yoshida \cite{KamiyamaYoshida}.

\end{abstract}

\maketitle

\tableofcontents

\

\section{Introduction}\label{introduction}
In \cite{SpeyerSturmfels} the authors associated Gr\"obner toric
degenerations $\mathrm{Gr}_2(\C^n)^{\tree}$ of $\mathrm{Gr}_2(\C^n)$
to each
trivalent tree $\tree$ with $n$ leaves. These degenerations
induce toric
degenerations  $M_{\br}^{\tree}$ of $M_{\br}$, the
space of $n$ ordered, weighted (by $\br$) points on the projective line.  
We denote the  corresponding toric fibers by
$\mathrm{Gr}_2(\C^n)_0^{\tree}$ and $(M_{\br})_0^{\tree}$ respectively.
Our goal in this paper is to give a
geometric (Euclidean polygon) description of the toric fibers as
stratified symplectic spaces (see \cite{SjamaarLerman} for this
notion) and describe the action of the compact part of the torus
as ``bendings of polygons.''


\subsection{The Grassmannian and imploded spin-framed polygons}
We start by identifying the Grassmannian $\mathrm{Gr}_2(\C^n)$
with the moduli space of
``imploded spin-framed'' $n$-gons in $\R^3$.
We define the space of \emph{imploded framed vectors}, which is
topologically the cone $C \mathrm{SO}(3,\R)$ of $\mathrm{SO}(3,\R)$, as the space
$$\Big\{(F,e) \in \mathrm{SO}(3,\R) \times \R^3 \mid
\text{ $e = tF(\epsilon_1)$ for some $t \in \R_{\geq 0}$}\Big\}$$
modulo the equivalence relation
$(F_1,0) \sim (F_2,0)$ for all $F_1,F_2 \in \mathrm{SO}(3,\R)$,
where $\epsilon_1 = (1,0,0)$ is the first standard basis vector of
$\R^3$.
We will see later that this equivalence relation is ``implosion'' in
the sense of \cite{GuilleminJeffreySjamaar}.
We call the equivalence class of $(F,e)$ an ``imploded framed vector''.  Note that the isotropy $T_{\mathrm{SO}(3,\R)}$ of $\epsilon_1$ in $\mathrm{SO}(3,\R)$ has a natural right action on the space of imploded framed vectors. 

Now fix a covering homomorphism $\pi : \mathrm{SU}(2) \to \mathrm{SO}(3,\R)$ such that the Cartan subgroup of diagonal matrices $T_{\mathrm{SU}(2)}$ in $\mathrm{SU}(2)$
maps onto $T_{\mathrm{SO}(3,\R)}$.
We define the space of
imploded \emph{spin}-framed vectors, which is topologically
the cone $C \mathrm{SU}(2) \cong \C^2$, as the space
$$\Big\{(F,e) \in \mathrm{SU}(2) \times \R^3 \mid
\text{ $e = t \pi(F)(\epsilon_1)$ for some $t \in \R_{\geq 0}$}\Big\}$$
modulo the equivalence relation
$(F_1,0) \sim (F_2,0)$ for all
$F_1,F_2 \in \mathrm{SU}(2)$.
We call the equivalence class of $(F,e)$ an ``imploded spin-framed vector''.
An \emph{imploded spin-framed $n$-gon} is an $n$-tuple
 $((F_1,e_1),\ldots,(F_n,e_n))$ of imploded spin-framed vectors
 such that $e_1 + e_2 + \cdots + e_n = 0$.
Let $\widetilde{P}_n(\mathrm{SU}(2))$
denote the space of imploded spin-framed $n$-gons.
There is an
action of $\mathrm{SU}(2)$ on $\widetilde{P}_n(\mathrm{SU}(2))$ given by
$$F \cdot ((F_1,e_1),\ldots,(F_n,e_n)) =
((F F_1, \pi(F)(e_1)),\ldots,(F F_n, \pi(F)(e_n))).$$
We let $P_n(\mathrm{SU}(2))$ denote the quotient
space. Note that since we may scale the edges of an $n$-gon the
space $P_n(\mathrm{SU}(2))$ is a cone with vertex the zero $n$-gon (all the
edges are the zero vector in $\R^3$).  Finally, note that 
there is a natural right action of $T_{\mathrm{SU}(2)^n}$ on 
$\widetilde{P}_n(\mathrm{SU}(2))$, which rotates frames but fixes 
the vectors:
$$(t_1,\ldots,t_n) \cdot ((F_1,e_1),\ldots,(F_n,e_n)) = 
((F_1 t_1, e_1),\ldots,(F_n t_n, e_n)).$$

\begin{remark}
We will see later that $P_n(\mathrm{SU}(2))$
occurs naturally in equivariant symplectic
geometry, \cite{GuilleminJeffreySjamaar}: it is the symplectic
quotient by the left diagonal action of $\mathrm{SU}(2)$ on the right
imploded cotangent bundle of $\mathrm{SU}(2)^n$. The space
$\widetilde{P}_n(\mathrm{SU}(2))$ is the zero level set of the momentum map
for the left diagonal action on the right imploded product of
cotangent bundles. Hence we find that the space $P_n(\mathrm{SU}(2))$ has a
(residual) right action of an $n$--torus $T_{\mathrm{SU}(2)^n}$, the
maximal torus in $\mathrm{SU}(2)^n$ which rotates the imploded spin-frames.
\end{remark}

Let $Q_n(\mathrm{SU}(2))$ be the quotient of the subspace of imploded spin-framed $n$-gons
of perimeter $1$ by the action of the diagonal embedded circle in
$T_{\mathrm{SU}(2)^n}$. In what follows $\mathrm{Aff} \mathrm{Gr}_2(\C^n)$ denotes the affine cone over
the Grassmannian $\mathrm{Gr}_2(\C^n)$ for the Pl\"ucker embedding. Thus $\mathrm{Aff} \mathrm{Gr}_2(\C^n)$ is the
subcone of $\bigwedge^2(\C^n)$ consisting of the decomposable
bivectors (the zero locus of the Pl\"ucker equations). The reason
why we consider $Q_n(\mathrm{SU}(2))$ here is the following theorem (proved in \S \ref{HJconstruction})
which gives a polygonal interpretation of $\mathrm{Gr}_2(\C^n)$. It is the
starting point of our work.

\begin{theorem}\label{polygonGrass}\hfill
\begin{enumerate}
\item $P_n(\mathrm{SU}(2))$ and
$\mathrm{Aff} \mathrm{Gr}_2(\C^n)$ are homeomorphic.
\item This homeomorphism induces a homeomorphism between  $Q_n(\mathrm{SU}(2))$ and $\mathrm{Gr}_2(\C^n).$
\end{enumerate}
\end{theorem}

\subsection{Triangulations, trivalent trees, and the construction of Kamiyama-Yoshida}
Let $P$ denote a fixed convex
planar $n$-gon.
Throughout the paper we will use the
symbol $\tree$ to denote either a triangulation of $P$ or its dual
trivalent tree. Accordingly we fix a triangulation $\tree$ of $P$.
The points of
$\mathrm{Gr}_2(\C^n)_0^{\tree}$ are ``imploded spin-framed $n$--gons'' with a
fixed perimeter in $\R^3$
modulo an equivalence relation  called $\tree$-congruence and denoted
$\sim_{\tree}$ that depends on the triangulation $\tree$.

\begin{figure}
\centering
\includegraphics[scale = 0.2]{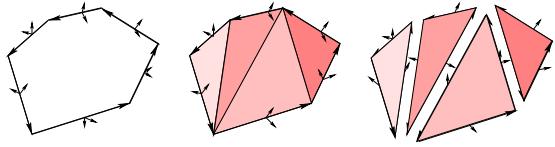}
\caption{A framed spatial polygon with chosen triangulation.}\label{fig:TriangulatedPolygon}
\end{figure}

Now that we have a polygonal interpretation of $\mathrm{Gr}_2(\C^n)$  we
will impose the equivalence relation of $\tree$-congruence (to be
described below) on our space of framed polygons and obtain a polygonal
interpretation of $\mathrm{Gr}_2(\C^n)_0^{\tree}$
corresponding to the triangulation (tree) $\tree$. We now describe
the equivalence relation of $\tree$-congruence.
Here we will discuss only the case
of the {\it the standard triangulation} $\tree_0$, that is the
triangulation of $P$ given by drawing the diagonals from the first
vertex to the remaining nonadjacent vertices. The dual tree to the
standard triangulation will be called the caterpillar or fan
(see Figure \ref{fig:CaterpillarTree7Leaves}). However we will state 
our theorems in the generality in which they are proved in the paper 
namely for all trivalent trees $\tree$ with $n$ leaves.


\begin{figure}
\centering
\includegraphics[scale = 0.2]{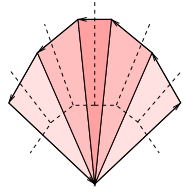}
\caption{The standard triangulation of the model $7$-gon with dual
caterpillar tree.}\label{fig:CaterpillarTree7Leaves}
\end{figure}

The reader is urged to refer to the pictures below to
understand the following description. The  equivalence relation
for the standard triangulation (in the case of $n$-gon linkages)
described below was first introduced in \cite{KamiyamaYoshida}. We
have extended their definition to all triangulations and to
$n$-gons equipped with imploded spin-frames. Label the diagonals
of the triangulation counterclockwise by $1$ through $n-3$. For each $S
\subset \{1,2,\cdots,n-3\}$ let $\widetilde{P}_n(\mathrm{SU}(2))^{[S]}$ denote the
subspace of $\widetilde{P}_n(\mathrm{SU}(2))$ where the diagonals
corresponding to the elements in $S$ are zero and all other diagonals are nonzero.
Let $\mathbf{F}=
((F_1,e_1),(F_2,e_2),\cdots,(F_n,e_n))$ be a point
in $\widetilde{P}_n(\mathrm{SU}(2))^{[S]}$.  Suppose that $|S| = i$.
Since $i$ diagonals are zero there will be $i+1$ sums of the form
$e_j + e_{j+1}+ \cdots + e_{j+k_j}$ that are zero, and the
$n$-gon underlying $\mathbf{F}$ will be the wedge of $i+1$ closed
subpolygons corresponding to the $i+1$ closed subpolygons in the
collapsed reference polygon. Thus we can divide $\mathbf{F}$ into
$i+1$ imploded spin-framed closed subpolygons ( in terms of
formulas we can break up the above $n$-tuple into $i+1$ sub
$k_j$-tuples of edges for $1 \leq j \leq i+1$). We may act on
each imploded spin-framed subpolygon ($k_j$-tuple) by a copy of
$\mathrm{SU}(2)$. We pass to the quotient $\widetilde{P}_n(\mathrm{SU}(2))/\sim_{\tree _0}$
by dividing out $\widetilde{P}_n(\mathrm{SU}(2))^{[S]}$ by the resulting action
of $\mathrm{SU}(2)^{i+1}$. By definition $\tree$-congruence is
the equivalence relation induced by the above quotient operation,
see Figure \ref{fig:TreeEquivalence}.

\begin{figure}[htp]
\centering
\includegraphics[scale = 0.2]{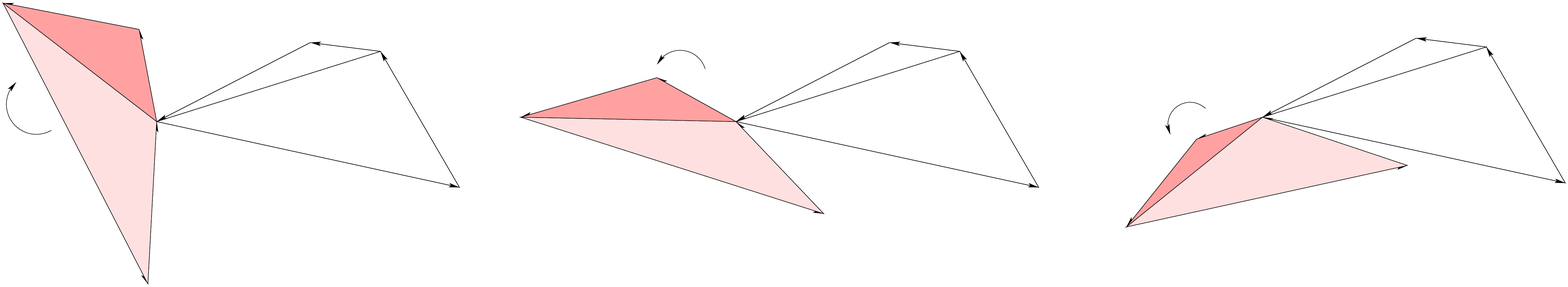}
\caption{A picture of equivalent polygons in $V_n^\tree$ which
are not equivalent in $P_n(\mathrm{SU}(2))$.}\label{fig:TreeEquivalence}
\end{figure}

We let $V_n^{\tree_0}$ denote the quotient space
$\widetilde{P}_n(\mathrm{SU}(2))/\sim_{\tree_0}$. Thus $V_n^{\tree_0}$ is
decomposed into the pieces $(V_n^{\tree_0})^{[S]} =
\widetilde{P}_n(\mathrm{SU}(2))^{[S]}/\sim_{\tree_0}$. We will call the
resulting decomposition in the special case of $\tree_0$
the Kamiyama-Yoshida decomposition (or
KY-decomposition).  In this paper we use a
weakened definition of the term {\it decomposition}; we will refer
to decompositions of spaces where the pieces are products
of spaces with isolated singularities.

\begin{remark}
The equivalence relation of $\tree$-congruence can be defined 
analogously for
any triangulation of $P$ (equivalently any trivalent tree with
$n$-leaves) and induces an equivalence relation on $Q_n(\mathrm{SU}(2))$
(and many other spaces associated to spaces of $n$-gons in $\R^3$,
for example, the space of $n$-gons itself or the space of
$n$-gon linkages). We will use the symbol $\sim_{\tree}$ to denote
all such equivalence relations.
\end{remark}

We will use $W_n^{\tree}$ to denote the quotient of
$Q_n(\mathrm{SU}(2))$ by
the equivalence relation $\sim_{\tree}$. We can now state our
first main result (this is proved in \S \ref{homeomorphismVtoP}).

\begin{theorem} \label{firstmaintheorem}\hfill
\begin{enumerate}
\item The toric fiber of the toric degeneration of
$\mathrm{Aff} \mathrm{Gr}_2(\C^n)$ corresponding to
the trivalent tree $\tree$ is homeomorphic to  $V_n^{\tree}$.
\item The toric fiber of the toric degeneration of
$\mathrm{Gr}_2(\C^n)$
corresponding to the trivalent tree $\tree$ is homeomorphic to
$W_n^{\tree}$.
\end{enumerate}
\end{theorem}

\begin{remark}
The quotient map from
$Q_n(\mathrm{SU}(2))$ to
$W_n^{\tree}$ given by passing to $\tree$-congruence classes
maps the generic fiber of the toric degeneration onto the special
(toric) fiber.
\end{remark}

We next describe the space $V_{\br}^{\tree_0}$ that will be proved
later to be homeomorphic to the toric fiber $(M_{\br})_0^{\tree_0}$.
Again we will restrict ourselves to the standard triangulation
in our description. Starting with the space $\widetilde{M}_{\br} =
\{(e_1,\ldots,e_n) \in (\R^3)^n \mid \sum_i e_i = 0, \; \|e_i\| = r_i\}$
of closed $n$-gon linkages with side-lengths $\br$, we define
$$V_{\br}^{\tree_0} = \widetilde{M}_{\br}/\sim_{\tree_0}.$$
This is the construction of \cite{KamiyamaYoshida} - see below
for some pictures.

\begin{theorem} \label{secondmaintheorem}
The toric fiber $(M_{\br})^{\tree}_0$ of the toric degeneration
of $M_{\br}$ corresponding to
the trivalent tree $\tree$ is homeomorphic to
$V_{\br}^{\tree}$.
\end{theorem}

\begin{remark} The quotient map from $M_{\br}$ to
$V_{\br}^{\tree}$ maps the generic fiber of the toric degeneration onto the special
(toric) fiber.
\end{remark}

The result in the previous theorem was
conjectured by Philip Foth and Yi Hu in \cite{FothHu}.

\subsection{Bending flows, edge rotations, and the toric structure of $W_{n}^{\tree_0}$}
The motivation for the above construction becomes more clear once we introduce the bending flows of
\cite{KapovichMillson} and \cite{Klyachko}. The lengths of the $n-3$ diagonals created are continuous
functions on $M_{\br}$ and are smooth where they are not zero.
They give rise to Hamiltonian flows which were called {\em bending
flows} in \cite{KapovichMillson}. The bending flow associated to a
given diagonal has the following description. The part of the
$n$-gon to one side of the diagonal does not move, while the other
part rotates around the diagonal at constant speed. The lengths of
the diagonals are action variables which generate the bending
flows, the conjugate angle variables are the dihedral angles
between the fixed and moving parts. However {\em the bending flow
along the $i$-th diagonal is not defined at those $n$-gons where
the $i$-th diagonal is zero}. If the bending flows are everywhere
defined (for example if one of the side-lengths is much larger
than the rest) then we may apply the theorem of Delzant,
\cite{Delzant} to conclude that $M_{\br}$ is toric. However for
many $\br$ (including the case of regular $n$
-gons) the bending
flows are not everywhere defined. The point of
\cite{KamiyamaYoshida} was to make the bending flows well-defined
by dividing out the subspaces of $\widetilde{M}_{\br}$ where a
collection of $i$ diagonals vanish by $\mathrm{SO}(3,\R)^{i+1}$. We
illustrate their construction with two examples.\vspace {1 mm}

First, let $\br = (1,1,1,1,1,1)$ so $M_{\br}$ is the space of
regular hexagons with side-lengths all equal to $1$. Let
$M_{\br}^{(2)}$
be the subspace of $M_{\br}$
where the middle (second) diagonal vanishes. Thus $M_{\br}^{(2)}$
is the space of ``bowties'' (see Figure \ref{fig:Bowties}) modulo the diagonal action of $\mathrm{SO}(3,\R)$
on the two equilateral triangles. We can no longer bend on the
second diagonal because we have no axis of rotation. Passing to
$\tree$-congruence classes
collapses the
space $M_{\br}^{(2)}$ to a point by dividing by the action
of $\mathrm{SO}(3,\R) \times \mathrm{SO}(3,\R)$. Bending along the second diagonal fixes 
this point by definition.


\begin{figure}[htbp]
\centering
\includegraphics[scale = 0.2]{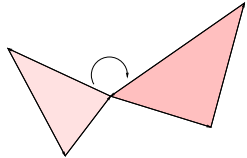}
\caption{The bowties are all $\tree_0$-congruent
and so define a single point in $V_\br^{\tree_0}$.
However in $M_\br$ the subspace of bowties is homeomorphic to $\mathrm{SO}(3,\R)$.}\label{fig:Bowties}
\end{figure}

For our second example,  we consider the space of regular octagons
with all side-lengths equal to $1$ and  the subspace $M_{\br}^{(3)}$ where
the middle diagonal vanishes. Thus $M_{\br}^{(3)}$ is the space of
wedges of rhombi modulo the diagonal action of $\mathrm{SO}(3,\R)$ on the two
rhombi, see Figure \ref{fig:TwoRhombi}.  Passing to $\tree$-congruence classes
amounts to dividing by the action of $\mathrm{SO}(3,\R) \times \mathrm{SO}(3,\R)$ on the
two rhombi.


\begin{figure}[htbp]
\centering
\includegraphics[scale = 0.2]{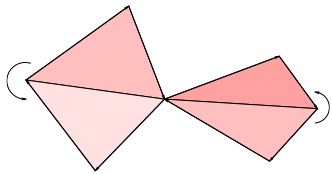}
\caption{Here
$n=8$ and $S= \{3 \}$.
The middle three components
of $T^5$ act trivially, and the
quotient $2$-torus acts by
bending along the first and fifth diagonals.}\label{fig:TwoRhombi}
\end{figure}

 We obtain the action of the bending flows on $V_n^{\tree_0}$
as follows. In this case we will be given a lift of the
one-parameter group bending along a diagonal to $\mathrm{SU}(2)$. The one
parameter group acts through its quotient in $\mathrm{SO}(3,\R)$ by bending
along the diagonal. If an edge moves under this bending then the
imploded spin-frame is moved by the one-parameter group in $\mathrm{SU}(2)$
in the same way. Hence, one part of the imploded spin-framed
polygon is fixed and the other moves by a ``rigid motion'' --i.e. all
the spin-framed edges of the second part are moved by the same
one-parameter group in $\mathrm{SU}(2)$. The bendings give rise to an
action of an $n-3$ torus $T_{bend}$ on $V_n^{\tree}$. There are also ''edge-rotations''
that apply a one-parameter group to the imploded spin-frame but do
not move the edge. This action on frames is the action of the torus
$T_{\mathrm{SU}(2)^n}$ coming from the theory of the imploded cotangent bundle
of $\mathrm{SU}(2)^n$. This action is not faithful, the diagonal subtorus acts
trivially. The bendings together with the edge rotations give rise to
an action of a compact $2n-4$ torus $T = T_{bend} \times T_{\mathrm{SU}(2)^n}$.

\subsection{A sketch of the proofs.}

The main step in proving Theorems \ref{firstmaintheorem} and \ref{secondmaintheorem} is to
produce a space $P_n^{\tree_0}(\mathrm{SU}(2))$ that ``interpolates''
between $V_n^{\tree_0}$ and $\mathrm{Gr}_2(\C^n)^{\tree_0}_0$.
Our construction  of $P_n^{\tree_0}(\mathrm{SU}(2))$
was motivated by the construction of the toric
degeneration of $\mathrm{SU}(2)$-character varieties of fundamental groups
of surfaces given by Hurtubise and Jeffrey in
\cite{HurtubiseJeffrey}. The connection is that the  space
$M_{\br}$ can be interpreted as the (relative) character variety
of the fundamental group of the $n$-punctured two-sphere with
values in the translation subgroup of the Euclidean group
$\mathbb{E}_3$ - a small loop around the $i$-th puncture maps to
translations by the $i$-th edge of the polygon (considered as a
vector in $\R^3$).

Take the triangulated model (convex planar) $n$-gon $P$ and break
it apart into $n-2$ triangles $T_1,\cdots,T_{n-2}$.
Equivalently we break apart the dual tree $\tree$ into a forest
$\tree^D$ consisting of $n-2$ tripods.

\begin{figure}[htbp]
\centering
\includegraphics[scale = 0.2]{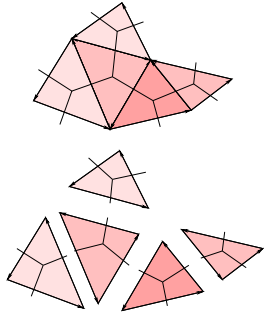}
\caption{The decomposed polygon and the decomposed dual tree $\tree^D$.}\label{fig:decomposed_ngon}
\end{figure}

Attach to each of the $3(n-2)$ edges  of the $n-2$ triangles
(or each edge of the forest $\tree^D$) a copy
of $T^*(\mathrm{SU}(2))$. Now right-implode each copy of $T^*(\mathrm{SU}(2))$ so that  a
copy of $\mathcal{E} T^\ast(\mathrm{SU}(2))) \cong \C^2$, the imploded cotangent bundle of
$\mathrm{SU}(2)$, is attached to each edge. Since
$\mathcal{E} T^\ast(\mathrm{SU}(2))$ admits
an action of the circle (from the right) and $\mathrm{SU}(2)$ from the left
the resulting space admits an action of a torus $\mathbb{T}$ of
dimension $3(n-2)$ and a
commuting action of $\mathrm{SU}(2)^{3(n-2)}$. The torus has a product
decomposition $\mathbb{T} = \mathbb{T}_e \times\mathbb{T}_d$
where $\mathbb {T}_e$ is the product of factors corresponding to
the $n$ edges of the polygon $P$ and $\mathbb{T}_d$ is the
product of factors corresponding to the $n-3$ diagonals of $P$.
Note that each diagonal of $P$ occurs in two triangles so corresponds to
two edges of $\tree^D$. Hence each diagonal gives
rise to a two-torus $S^1 \times S^1$ in $
\mathbb{T}$ which we will
refer to as the two-torus corresponding to that diagonal. Define a
subtorus $\mathbb{T}_d^-$ of $\mathbb{T}_d$ of dimension $n-3$ by
taking an antidiagonal embedding of $S^1$ in each two-torus
corresponding to a diagonal.

Now take the symplectic quotient (at level zero) of (the product
of) the three copies of $\C^2$ associated to the three sides of
each triangle by $\mathrm{SU}(2)$ acting diagonally. For each triangle
(or each tripod) we
obtain a resulting copy of $\bigwedge^2(\C^3)$.  The resulting
product $(\bigwedge^2(\C^3))^{n-2}$ has an induced action of the
torus $\mathbb{T}$. Glue the copies $\bigwedge^2(\C ^3)$
associated to the triangles together along the edges of the
triangles associated to diagonals by taking the symplectic quotient at level zero by the
torus $\mathbb{T}_d^- \subset \mathbb{T}_d$ described above. Each of
the two previous symplectic quotients has a corresponding GIT quotient.
 Taking both GIT quotients we obtain
a space which is an affine torus quotient of affine space. Hence the combined
symplectic quotient is the space underlying the
affine toric variety
$$P_n^{\tree_0}(\mathrm{SU}(2)) =(\mathrm{\bigwedge}
\space^2(\C^3))^{n-2}\cq
_0 \mathbb{T}_d^-.$$
Let $\mathbf{t}_e(\lambda)$ be the element in complexified torus $\underline{\mathbb{T}} \cong (\C^\ast)^{3n-6}$ of $\mathbb{T}$ 
such that all the edge
components coincide with $\lambda \in \C^{\ast}$ and
all the components corresponding to
diagonals are $1$. Then Theorem \ref{firstmaintheorem} follows by
putting together items 2 and 4 in the next theorem, which is proved in
\S \ref{homeomorphismVtoP}.

\begin{theorem} \label{firstauxiliarytheorem}
\hfill
\begin{enumerate}
\item The toric varieties $\mathrm{Aff} \mathrm{Gr}_2(\C^n)_0^{\tree}$ 
and $P_n^{\tree}(\mathrm{SU}(2))$ are isomorphic as
affine toric varieties.
\item The grading action of $\lambda \in
\C^*$ on $\mathrm{Aff} \mathrm{Gr}_2(\C^n)_0^{\tree}$ corresponds to
the action of $\mathbf{t}_e(\sqrt{\lambda}^{-1})$ on
$P_n^{\tree }(\mathrm{SU}(2))$ (this is well-defined).
Consequently $ \mathrm{Gr}_2(\C^n)_0^{\tree}$
is projectively isomorphic to the quotient of
$P_n^{\tree}(\mathrm{SU}(2))$ by this $\C^{\ast }$ action - we will
denote this quotient by $Q_n^{\tree}(\mathrm{SU}(2))$.
\item There is a
homeomorphism (that creates imploded spin-frames along the
diagonals of the triangulation)
$$\Psi_n^{\tree}:V_n^{\tree} \to P_n^{\tree}(\mathrm{SU}(2)).$$
\item The homeomorphism
$\Psi_n^{\tree}$ induces a
homeomorphism from $W_n^{\tree}$ to the projective toric variety
$Q_n^{\tree}(\mathrm{SU}(2))$.
\end{enumerate}
\end{theorem}
The quotient
$P_n^{\tree_0}(\mathrm{SU}(2))=\bigwedge^2(\C^3)^{n-2} \q \mathbb{T}_d^-$
admits a residual
action by the quotient torus $\mathbb{T}/ \mathbb{T}_d^-$. This
quotient torus contains a factor that can be identified with
$\mathbb{T}_e$. The toric fiber $(M_{\br})_0$ is obtained from
$P_n^{\tree_0}(\mathrm{SU}(2))$ by taking the symplectic quotient by
$\mathbb{T}_e$ at level $\br$. Theorem \ref{secondmaintheorem}
follows from the two statements of the following theorem.

\begin{theorem}\label{secondauxiliarytheorem}\hfill
\begin{enumerate}
\item The toric variety $P_n^{\tree}(\mathrm{SU}(2))\cq_{\br} \mathbb{T}_e $
is isomorphic to the toric variety
$(M_{\br})_0^{\tree}$.
\item For each $\br$ the homeomorphism $\Psi_n^{\tree}$ induces a homeomorphism
$\Psi_{\br}:V^{\tree}_{\br} \to P_n^{\tree}(\mathrm{SU}(2))\cq_{\br} \mathbb{T}_e .$
\end{enumerate}\end{theorem}

Note that the toric varieties
$(\bigwedge^2(\C^3))^{n-2} \q \mathbb{T}_d^-$
resp. $\bigwedge^2(\C^3)^{n-2}
\cq_{\br,0} (\mathbb{T}_e \times\mathbb{T}_d^-)$ mediate between the
Kamiyama-Yoshida spaces $V_n^{\tree}$ resp. $V_{
\br }^{\tree}$ and the toric
varieties $\mathrm{Gr}_2(\C^n)_0^{\tree}$ resp. $(M_{ \br})_0^{\tree}$.

\subsection{The Hamiltonian nature of the edge rotations and the bending flows}
It remains to place the edge rotations and bending flows in their proper context (in
terms of  symplectic and algebraic geometry). We do this with the following theorems. 
We note that since the spaces
$P_n^{\tree}(\mathrm{SU}(2))$ and $P_{\br}^{\tree}(\mathrm{SU}(2))$ are quotients of
the affine space $(\bigwedge^2(\C^3))^{n-2}$ by tori they inherit symplectic
stratifications from the orbit type stratification of
$(\bigwedge^2(\C^3))^{n-2}$ according to \cite{SjamaarLerman}.  
We identify the edge flows and bending flows with the action of
the maximal compact subgroups (compact torus) of the complex tori
that act holomorphically with open orbits on the toric
varieties $\mathrm{Gr}_2(\C^n)_0^{\tree}$ and
$(M_{ \br})_0^{\tree}$. This is accomplished
by the following theorem, which is proved
by Theorems \ref{noncompact} and \ref{noncompact2}.

\begin{theorem}
\hfill
\begin{enumerate}
\item The action of the edge rotations corresponds under
$\Psi_n^{\tree}$ to the residual action of $\T / \T_d^-$.
\item The action of the bending flows corresponds under
$\Psi_n^{\tree}$ (resp. $\Psi_{\br}^{\tree}$) to the residual action
of $\T / \T_e \times \T_d^-$. 
\end{enumerate}

\end{theorem}
We may now give the edge rotations and the 
bending flows on $V_n^{\tree}$ and $V_{\br}^{\tree}$ 
a natural Hamiltonian interpretation.  This is
proved by the last theorem along with Proposition \ref{pullback}
and Theorem \ref{diagonalpullback}.

\begin{theorem}\label{stratification}
\hfill
\begin{enumerate}
\item The edge rotations on $V_n^{\tree}$ are the stratified
symplectic Hamiltonian flows (in the sense of \cite{SjamaarLerman})
associated to the lengths of edges in $\tree$-congruence
classes of spin-framed polygons. 
\item The bending flows on $V_n^{\tree}$ (resp. $V_{\br}^{\tree}$) are the
stratified Hamiltonian flows associated to the lengths of the diagonals 
of $\tree$-congruence classes of spin-framed polygons (resp. polygonal linkages).  
\end{enumerate}

\end{theorem}

\medskip

{\bf Acknowledgements.} 
We thank Bill Goldman, Henry King and Reyer Sjamaar for useful
conversations. The  authors would also like to thank
Andrew Snowden and Ravi Vakil, a number of ideas from the collaboration
 \cite{HowardMillsonSnowdenVakil} have reappeared in this paper.
 We thank Philip Foth and Yi Hu for posing the
problem solved in this paper. They first observed that the toric
degenerations of flag varieties constructed by \cite{AlexeevBrion}
could be descended to the associated weight varieties. We
especially thank Philip Foth for pointing out that the
construction of \cite{AlexeevBrion} gives one toric degeneration
of $M_{\br}$ for each triangulation of the model convex
$n$-gon.
This led us to consider triangulations other than the standard
one. We thank Bernd Sturmfels for telling us
about \cite{SpeyerSturmfels} and \cite{BuczynskaWisniewski} which
led us to understand from the point of view of tree metrics why
there was one toric degeneration for each triangulation.
We thank Allen Knutson for pointing out that  toric degenerations
of Grassmannians were first constructed by Sturmfels in
\cite{Sturmfels}. Finally we should emphasize that
the notion of $\tree$-congruence
is based on the work of Kamiyama and Yoshida \cite{KamiyamaYoshida}
and that the notion of bending flows is based on \cite{KapovichMillson}
and \cite{Klyachko}.

\section{The moduli spaces  of $n$-gons and $n$-gon linkages in $\R^3$}
Throughout this paper the term $n$-gon will mean a closed $n$-gon
in $\R^3$ modulo translations. More precisely an $n$-gon $\mathbf{e}$
will be an $n$-tuple $\mathbf{e} = (e_1,e_2,\cdots,e_n)$ of vectors
in $\R^3$ satisfying the closing condition
$$e_1 + e_2 + \cdots + e_n =0.$$
We will say the $e_i$ is the $i$-th edge of $\mathbf{e}$.
We will say two $n$-gons $\mathbf{e}$ and $\mathbf{e}^{\prime}$ are congruent
if there exists a rotation $ g \in \mathrm{SO}(3,\R)$ such that
$$ e_i^{\prime} = g e_i , 1 \leq i \leq n.$$
We will let ${\mathrm{Pol}_ n}$ denote the space of closed $n$-gons in $\R^3$
and $\overline{\mathrm{Pol}}_ n$ denote the quotient space of $n$-gons modulo congruence.

Now let $\br = (r_1,r_2,\cdots,r_n)$ be an $n$-tuple of nonegative
real numbers. We will say an $n$-gon $\mathbf{e}$ is an $n$-gon { \em linkage}
with side-lengths $\br$ if the $i$-th edge of $\mathbf{e}$ has length $r_i,
1 \leq i \leq n$. We will say an $n$-gon or $n$-gon linkage is
{\it degenerate} if it is contained in a line.

We define the configuration space $\widetilde{M_{\br}}$ to be
the set of $n$--gon linkages with side-lengths $\br$.  We will define the moduli space $M_{\br}$ of $n$-gon linkages to be the quotient of
the configuration space by $\mathrm{SO}(3,\R)$. The space $M_{\br}$ is a complex analytic space
(see \cite{KapovichMillson} with isolated singularities at the degenerate $n$-gon linkages.

Recall that we have defined a reference convex planar $n$-gon $P$.
Let $u_i, u_j$ be an ordered pair of nonconsecutive vertices of $P$. For any $n$-gon $\mathbf{e}
\in \R^3$ we have corresponding vertices $v_i$ and $v_j$ defined up to
simultaneous translation.  The vector in $\R^3$ pointing from $v_i$ to $v_j$
will be called a {\em diagonal} of $\mathbf{e}$.  We let $d_{ij}(\mathbf{e})$
be the length of this diagonal.
In  \cite{KapovichMillson} and \cite{Klyachko} the authors described
the Hamiltonian flow corresponding to $d_{ij}(\mathbf{e})$ - see the Introduction.
In \cite{KapovichMillson} these flows were called {\em bending flows}.
Furthermore it was proved in \cite{KapovichMillson} and \cite{Klyachko} if two such diagonals do not
intersect then the corresponding bending flows commute.
Since each triangulation $\tree$ of $P$ contains $n-3 = \frac{1}{2} \dim (M_{\br})$
nonintersecting diagonals it follows that each one has an integrable
system on $M_{\br}$ for each such $\tree$.  Unfortunately these flows
are not everywhere defined. The bending flow
corresponding to $d_{ij}$ is not well-defined for those $\mathbf{e}$ where $d_{ij}(\mathbf{e})$ is zero and the Hamiltonian $d_{ij}$ is not differentiable
at such $\mathbf{e}$.

\section{The space of imploded spin-framed Euclidean $n$--gons and the
Grassmannian of two planes in complex $n$ space} \label{HJconstruction}

In this section we will construct the space $P_n(\mathrm{SU}(2))$ of imploded spin-framed $n$--gons in $\R^3$ modulo $\mathrm{SU}(2)$ and prove that this space is
isomorphic
to $\mathrm{Aff} \mathrm{Gr}_2(\C^n)$ as a symplectic manifold
and as a complex projective variety. We will first
construct $P_n(\mathrm{SU}(2))$ using the extension and implosion technique
of \cite{HurtubiseJeffrey} without reference to $n$--gons in $\R^3$,
then
relate the
result to $\mathrm{Gr}_2(\C^n)$. Then we will show that a point in $P_n(\mathrm{SU}(2))$
can be interpreted as a Euclidean $n$-gon equipped with an
imploded spin-frame.

\subsection{The imploded extended moduli spaces $P_n(G)$ of $n$-gons
and $n$-gon linkages.}
In this subsection we will define the imploded extended moduli
space $P_n(G)$ of $n$-gons in $\mathfrak{g}^*$ for a
general semisimple Lie group $G$.
We have included this subsection to make the connection
with \cite{HurtubiseJeffrey}.
Throughout we assume that $G$ is semisimple.\subsubsection{The moduli
spaces of $n$-gons and $n$-gon linkages}
To motivate the definition of the next subsection we briefly recall
two definitions.\begin{definition}
An $n$-gon in $\mathfrak{g}^*$ is an $n$-tuple of vectors $e_i, 1 \leq
i \leq n,  \in\mathfrak{g}^*$ satisfying the closing condition$$e_1 +
\cdots + e_n = 0.$$ We define the moduli space of $n$-gons to
be the set of all $n$-gons modulo the diagonal coadjoint action of
$G$.
\end{definition}
Now choose $n$ coadjoint orbits $\mathcal{O}_i, 1 \leq i \leq
n$.

\begin{definition}
We define an $n$-gon {\em linkage} to be an $n$-gon such that $e_i \in
\mathcal{O}_i, 1 \leq i \leq n$. We define the moduli space
of $n$-gon linkages to be the set of all $n$-gon linkages modulo
the diagonal coadjoint action of $G$.
\end{definition}
We leave the proof of the following lemma to the reader.

\begin{lemma}
The moduli space of $n$-gon linkages is the symplectic quotient$$G \ql
(\prod_{i=1}^n \mathcal{O}_i).$$
\end{lemma}

\begin{remark}\label{moduliofconnections}
The moduli space of $n$-gons is in fact
a moduli space of flat connections modulo gauge
transformations (equivalently
a character variety).  The moduli space
of $n$-gon linkages is a moduli space of
flat connections with the conjugacy classes
of holonomies fixed in advance (a relative character
variety). In this case the
flat connections are over an $n$-fold punctured $2$-sphere and the
structure group is the cotangent bundle $T^*(G) = G \ltimes \mathfrak{g}^*$.
The holonomy around each puncture is a ``translation'', i.e. a group
element of the form $(1,v), v \in \mathfrak{g}^*$.
To see this connection in more detail the reader is referred to
\S 5 of
\cite{KapovichMillson}. We will not need this connection in
what follows.
\end{remark}

\subsubsection{The extended moduli spaces of $n$-gons and $n$-gon
linkages}
The following definition is motivated by the definition of the extended
moduli spaces of flat connections of \cite{Jeffrey}.

\begin{definition}
We define the extended moduli space $M_n(G)$ to be
the symplectic quotient of  $T^*(G)^n$ by the left diagonal action
of $G$: $$M_n(G) = G \ql T^*(G)^n.$$
\end{definition}
The space $M_n(G)$ has a $G^n$ action coming from right multiplication
on $T^*(G)$.
We take each $T^*(G)$ to be identified with $G\times\mathfrak{g^*}$ using
the
left-invariant trivialization.

\begin{lemma}$M_n(G) = G\backslash\{((g_1,\alpha_1),
\ldots,(g_n,\alpha_n)): \sum_{i=1}^n Ad_{g_i}(\alpha_i) = 0
\}.$
\end{lemma}

\begin{proof}
The momentum mapping associated to the diagonal left action
on $T^*(G)^n$ (identified with $(G \times \mathfrak{g}^*)^n$ using the
left-invariant trivialization) is
$$\mu_L((g_1,\alpha_1), \ldots, (g_n,\alpha_n)) =- \sum_{i=1}^n
Ad_{g_i}(\alpha_i).$$
The expression on the right above
clearly corresponds to the 0-momentum level of $\mu_L$ which
is $M_n(G)$.
\end{proof}

\begin{remark}
The  reason for the term {\em extended} moduli space of $n$-gons is that
this space is obtained
from the moduli space of $n$-gons  by adding the frames $g_1,g_2,\cdots,g_n$.
Note that the moduli space of $n$-gons is embedded in the extended
moduli space as the subspace corresponding to $g_1 = g_2 = \cdots = g_n = e$.
If we wish to fix the conjugacy classes of the second components we
will call the above the extended moduli space of $n$-gon {\em
linkages}.
\end{remark}

\subsubsection{The imploded extended moduli spaces of $n$-gons and $n$-gon
linkages}
Choose a maximal torus $T_G\subset G$ and a (closed) Weyl chamber $\Delta$
contained
in the
Lie algebra $\mathfrak{t}$ of $T_G$. Note that the action of $G^n$ 
on $M_n(G)$ by right multiplication induces an action
by the torus $T_G^n$ on $M_n(G)$.
We now obtain the imploded extended moduli space $P_n(G)$ by {\it imploding},
following \cite{GuilleminJeffreySjamaar}, the extended moduli space $M_n(G)$.

\begin{definition}\label{Definitionofimplodedextended moduli space $P_n(G)$} $$P_n(G) = M_n(G)_{impl}.$$
\end{definition}
Following \cite{HurtubiseJeffrey} we will use $\mathcal{E}T^*(G)$
to denote the imploded cotangent bundle
$$\mathcal{E}T^*(G) = T^*(G)_{impl}.$$
For the benefit of the reader we will recall the definition
of $\mathcal{E}T^*(G)$. We have
$$\mathcal{E}T^*(G) = \mu_G^{-1}(\Delta)/\sim $$
Here $\mu_G$ is the momentum map for the action of $G$
by right translation and we have identified $\mathfrak{t}$
(resp. $\Delta$) with $\mathfrak{t}^*$ (resp. a dual Weyl chamber) using the
Killing form.
The equivalence relation $\sim$ is described as
follows.
Let $F$ be a open face of $\Delta$ and let $h_F$ be a generic element of $F$.
Let $G_F$ be the subgroup of $G$ which is the derived subgroup of the
stabilizer of $h_F$ under the
adjoint representation. We define $x$ and $y$ in $\mu_T^{-1}(\Delta)$
to be equivalent if $\mu_T(x)$ and $\mu(y)$ lie in the same face $F$
of $\Delta$ and $x$ and $y$ are in the same orbit under $G_F$. Thus we
divide
out the inverse images of  faces by different subgroups of $G$.
In what follows we will use the symbol $\sim$ to denote
this equivalence relation (assuming the group $G$, the torus $T$
and the chamber $\Delta$ are understood).

We define the space $E_n(G)$ by
$$E_n(G) = \mathcal{E}T^*(G)^n.$$
Noting that right implosion commutes
with left symplectic quotient we have
$$P_n(G) = G \backslash E_n(G) 
=  G\backslash \Big\{ ([g_1,\alpha_1], \ldots, [g_n,\alpha_n]) \mid \alpha_1,
\ldots, \alpha_n \in \Delta,\sum_{i=1}^n Ad_\mathfrak{g_i}(\alpha_i) =
0\Big\}.$$
In the above $[g_i, \alpha_i]$ denotes the equivalence class
in $T^*(G)$ relative to the equivalence relation $\sim$ above.
We will sometimes use $E_n(G)$ to denote the product
$\mathcal{E}T^*(G)^n = \mathcal{E}T^*(G^n)$. In what follows we let
$\mathbf{t}(\lambda)$ denote the element of the complexification maximal torus
$\underline{T}_{\mathrm{SU}(2)^n}$.

\subsection{The isomorphism of  $Q_n(\mathrm{SU}(2))$ and $\mathrm{Gr}_2(\C^n)$}
In this subsection we will prove Theorem \ref{polygonGrass}
by calculating $P_n(\mathrm{SU}(2))$, the imploded extended
moduli space for the group $\mathrm{SU}(2)$ and its quotient $Q_n(\mathrm{SU}(2))$.
We will in fact need a slightly more precise version than that
stated in the Introduction.

\begin{theorem}\label{polygonGrasspreciseversion}
\hfill
\begin{enumerate}
\item There exists a homeomorphism
$\psi:\mathrm{Aff} \mathrm{Gr}_2(\C^n) \to P_n(\mathrm{SU}(2))$.
\item The homeomorphism $\psi$ intertwines the natural action of the maximal torus $T_{\mathrm{U}(n)}$ of \ $\mathrm{U}(n)$ with  (the inverse of) that of $T_{\mathrm{SU}(2)}^n$ acting on imploded frames.
\item The homeomorphism $\psi$ intertwines the grading circle (resp. $\C^{\ast}$)
actions on $\mathrm{Aff} \mathrm{Gr}_2(\C^n)$ with the actions of $\mathbf{t}((\exp{i\theta})^{-1/2})$ (resp. $\mathbf{t}((\lambda)^{-1/2})$)
\item The homeomorphism $\psi$ induces a homeomorphism between $\mathrm{Gr}_2(\C^n)$
and $Q_n(\mathrm{SU}(2)).$
\end{enumerate}

\end{theorem}
In order to prove the theorem we first
need to compute $\mathcal{E}T^*(\mathrm{SU}(2))$.

\subsubsection{The imploded cotangent bundle of $\mathrm{SU}(2)$ }

In this section we will review the formula
of \cite{GuilleminJeffreySjamaar} for the (right) imploded
cotangent bundle $\mathcal{E}T^*(\mathrm{SU}(2))$. Let $T_{\mathrm{SU}(2)}$ be the maximal
torus of $\mathrm{SU}(2)$ consisting of the diagonal matrices.
Let $\mathfrak{t}$ be the Lie algebra of $T_{\mathrm{SU}(2)}$ and let $\Delta$ be
the
positive Weyl chamber in $\mathfrak{t}$ (so $\Delta$ is a ray in
the one-dimensional vector space $\mathfrak{t}$).

We will take as basis for $\mathfrak{t}$ the coroot $\alpha^{\vee}$
(multiplied by $i$), that is
$$\alpha^{\vee}= \begin{pmatrix}  i & 0 \\  0 & -i \\  \end{pmatrix}$$
Then $\alpha^{\vee}$ may be
  identified  with a
  basis vector over $\Z$ for the cocharacter lattice $X_*(T_{\mathrm{SU}(2)})$.
It is tangent at the identity to a unique cocharacter.
   Let $\varpi_1$ be the fundamental weight
of $\mathrm{SL}(2,\C)$ thus
$$ \varpi_1(\alpha^{\vee}) = 1.$$
Then $\varpi_1$ may be identifed with a character of $T$ (it
is the derivative at the identity of a unique character) which is a basis
for
the character lattice of $T_{\mathrm{SU}(2)}$.

\smallskip
Let $\mu_R$ be the momentum map for the action of $\mathrm{SU}(2)$ on $T^*\mathrm{SU}(2)$
induced by right multiplication. Specializing the
definition of the imploded cotangent bundle to $G = \mathrm{SU}(2)$ we find that
the (right)
implosion of $T^*\mathrm{SU}(2)$ is  the set of equivalence
classes
$$\mathcal{E} T^\ast (\mathrm{SU}(2)) = \mu_R^{-1}(\Delta)/\sim$$
where two points $x,y \in \mu_R^{-1}(\Delta)$ are
equivalent if
$$\mu_R(x) = \mu_R(y) = 0 \ \text{and} \ x = y \cdot g \ \text{for
some}\ g \in \mathrm{SU}(2).$$
It follows from the general theory of  \cite{GuilleminJeffreySjamaar}
that $\mathcal{E}T^*(\mathrm{SU}(2))$ has an induced
structure of a stratified symplectic space in the sense of
\cite{SjamaarLerman} with induced isometric actions
of $\mathrm{SU}(2)$  induced by the left action of $\mathrm{SU}(2)$ on $T^*\mathrm{SU}(2)$
and $T_{\mathrm{SU}(2)}$ induced by the right action of $T_{\mathrm{SU}(2)}$ on $T^*(\mathrm{SU}(2))$.
However the exceptional
feature of this special case is that $\mathcal{E}T^*(\mathrm{SU}(2))$ is a
{\em manifold}. The multiplicative group $\R_+$ acts on $\mathcal{E}T^*(\mathrm{SU}(2))$
by the formula
$$\mu \cdot [g,\lambda \varpi_1] = [g, \mu \lambda \varpi_1]$$
whereas the actions of $g_0 \in \mathrm{SU}(2)$ and $t \in T_{\mathrm{SU}(2)}$ described
above are given by
$$g_0 \cdot [g, \lambda \varpi_1] = [g_0 g, \lambda
\varpi_1] \ \text{ and } t \cdot [g, \lambda \varpi_1] = [gt, \lambda
\varpi_1].$$

The product group $\mathrm{SU}(2) \times \R_+$ fixes the point $[e,0]$
and acts transitively on the complement of $[e,0]$ in $\mathcal{E}T^*(\mathrm{SU}(2))$.
The reader will verify, see Example 4.7 of \cite{GuilleminJeffreySjamaar},
that the symplectic form $\omega$ on $\mathcal{E}T^*(\mathrm{SU}(2))$
is homogeneous of degree one under $\R_+$.

The following lemma is extracted from Example 4.7
of \cite{GuilleminJeffreySjamaar}. Since this lemma is central to what
follows
we will go into more detail for its proof than is
in \cite{GuilleminJeffreySjamaar}.

\begin{lemma}\label{implodedcotangentbundle}
The map $\phi: \mathcal{E} T^\ast (\mathrm{SU}(2)) \to \C^2$
given by
$$\phi([g,\lambda \varpi_1]) = \sqrt{2\lambda} g
\begin{pmatrix}1\\0\end{pmatrix} \ \text{and} \ \phi([e,0]) =
\begin{pmatrix}0\\0\end{pmatrix}$$
induces a symplectomorphism of stratified symplectic
spaces onto $\C^2$
where $\C^2$ is given its standard symplectic structure.
Under this isomorphism the action of
$\mathrm{SU}(2)$ on $\mathcal{E}T^*(\mathrm{SU}(2))$
goes to the standard action on $\C^2$ and
the action of $T$ on $\mathcal{E}T^*(\mathrm{SU}(2))$ goes to the action of $T$
given by $t\cdot (z,w) = (t^{-1}z, t^{-1}w)$.
\end{lemma}

\begin{proof}
The inverse to $\phi$ is given (on nonzero elements of $\C^2$)
by $$\psi((z,w)) = \Big[g, \frac{|z|^2 + |w|^2}{2} \varpi_1\Big],$$
where $$g= \frac{1}{\sqrt{|z|^2 + |w|^2}} \begin{pmatrix}z &
-\overline{w} \\w & \overline{z}\end{pmatrix}.$$
It follows that $\phi$ is a homeomorphism. It remains to check
that it is a symplectomorphism.
Noting that $\phi$ is homogeneous under $\R_+$ of degree one-half
and the standard symplectic form $\omega_{\C^2}$ is homogeneous of
degree two it
follows (from the transitivity of the action of $\mathrm{SU}(2) \times \R_+$) that
the  symplectic form $\phi^{\ast} \omega_{\C^2}$ is a constant
multiple of $\omega$.
Thus is suffices to prove that the two above forms both take value $1$ on
the ordered
pair of tangent vectors $(\partial/\partial \lambda, \alpha^{\vee})$. We
leave
this to the reader.
\end{proof}

\begin{remark}
The fact that we have inverted the usual action of the circle on the
complex plane will have
the effect
of changing the signs of the torus Hamiltonians that occur in the rest
of this paper (from {\it minus} one-half a sum of squares of norms of complex
numbers to {\it plus} one-half of the corresponding sum).
\end{remark}
Taking into account the above remark we obtain a formula for the
momentum map for the action of $T_{\mathrm{SU}(2)}$ on $\mathcal{E}T^*(\mathrm{SU}(2))$
which will be
critical in what follows.
Now we have
\begin{lemma}\label{momentummap}
  The momentum map $\mu_T$ for the action of $T_{\mathrm{SU}(2)}$ on
$\mathcal{E} T^\ast (\mathrm{SU}(2))=\C^2$
  is given by
  $$\mu_T((z,w)) = (1/2)(|z|^2 + |w|^2) \varpi_1.$$
  Consequently, the Hamiltonian $f_{\alpha^{\vee}}(z,w)$ for the
fundamental
  vector field on $\mathcal{E}T^\ast (\mathrm{SU}(2))$ induced by $\alpha^{\vee}$
  is given by
  $$f_{\alpha^{\vee}}(z,w) = (1/2)(|z|^2 + |w|^2).$$
\end{lemma}

\subsubsection{The computation of $P_n(\mathrm{SU}(2))$ and $Q_n(\mathrm{SU}(2))$}
We can now prove Theorem \ref{polygonGrasspreciseversion}. Recall we have introduced
the abbreviation $E_n(\mathrm{SU}(2))$ for the product $\mathcal{E} T^\ast (\mathrm{SU}(2))^n$.
We will use $\mathbf{E}$ to denote an element of $E_n(\mathrm{SU}(2))$.
It will be convenient to regard  $\mathbf{E}$ as a function from
the edges of the reference polygon $P$ into $\mathcal{E} T^\ast (\mathrm{SU}(2))$.
We note that there is a map from $E_n(\mathrm{SU}(2))$ to not necessarily closed
$n$-gons in $\R^3$,
the map that scales the first element of the frame by $\lambda$ and forgets
the other two elements of the frame this sends
each equivalence class $[g_i, \lambda \varpi_1]$ to $Ad^{\ast} g_i (\lambda_i \varpi_1)$ for  $1 \leq i \leq n$.

By the previous lemmas we may represent an element  of $E_n(\mathrm{SU}(2))$
by the $2$ by $n$ matrix
\begin{center}$A = \begin{pmatrix}z_1 & z_2 & \ldots & z_{n-1} & z_n
\\w_1 & w_2 & \ldots & w_{n-1} & w_n\end{pmatrix}$.\end{center}
The left diagonal action of the group $\mathrm{SU}(2)$ on $E_n(\mathrm{SU}(2))$
is then represented
by the diagonal action of $\mathrm{SU}(2)$ on the columns of $A$.
We let $[A]$ denote the orbit equivalence class of $A$ for this action.
  We let $T_{\mathrm{U}(n)}$ be the compact $n$ torus
of diagonal elements in
in $\mathrm{U}(n))$. Then $T_{\mathrm{U}(n)}$ acts by scaling the columns of the matrices $A$.
This action coincides with the inverse of the action of the $n$-torus $T_{\mathrm{SU}(2)^n}$
coming from the theory of the imploded cotangent bundle. By this we mean
that if an element in $T_{\mathrm{U}(n)}$ acts by scaling a column by $\lambda$
then the corresponding element acts by scaling the same column by
$\lambda^{-1}$.

The first statement of
Theorem \ref{polygonGrasspreciseversion} will be a consequence of the following two lemmas.
Let $M_{2,n}(\mathbb{C})$ be the space of $2$ by $n$ complex
matrices $A$ as above. In what follows we will let $Z$ denote the first
row
and $W$ the second row of the above matrix $A$. We give $M_{2,n}(\mathbb{C})$
the Hermitian structure given by
$$( A,B) = Tr (AB^*).$$
We will use $\mu_G$ to denote the momentum map
for the action of $\mathrm{SU}(2)$ on the left of $A$.

\begin{lemma}\label{levelzero}$$\mu_G(A) = (1/2)\begin{pmatrix}
(\|Z\|^2 - \|W\|^2)/2 & Z \cdot W \\W \cdot Z & (\|W\|^2 -
\|Z\|^2)/2\end{pmatrix}$$
Hence $\mu_G(A) = 0$ $\iff$ the two rows of $A$ have the same length
and are orthogonal.
\end{lemma}

\begin{proof}
First note that the momentum map $\mu$ for the action of $\mathrm{SU}(2)$
on $\C^2$ is given by$$\mu(z,w) = (1/2)\begin{pmatrix} (|z|^2 -
|w|^2)/2 & z\overline{w} \\w\overline{z}  & (|w|^2 -
|z|^2)/2\end{pmatrix}.$$
Now add the momentum maps for each of the columns to get the lemma.
\end{proof}
We find that $$ \mathrm{SU}(2) \ql M_{2,n}(\C) = \mathrm{SU}(2) \backslash \{ (Z,W):
\|Z\|^2 = \|W\|^2, Z \cdot W = 0 \}.$$

\begin{lemma}\label{affinegrass}
 $$\mathrm{SU}(2) \ql M_{n,2}(\C) \cong \mathrm{Aff} \mathrm{Gr}_2(\C^n).$$
\end{lemma}

\begin{proof}
The map $\varphi:\mu^{-1}(0) \to \bigwedge^2(\C^n)$ given by $F(A) = Z\wedge
W$  maps onto the decomposable vectors, descends to the quotient
by $\mathrm{SU}(2)$, and induces the required isomorphism.
\end{proof}

In what follows it will be important to understand the above result
in terms of the GIT quotient of $M_{2,n}(\mathbb{C})$ by $\mathrm{SL}(2,\C)$
acting on the
left. Since we are taking a quotient of affine space
by a reductive group  the Geometric Invariant Theory
quotient of $M_{2,n}(\mathbb{C})$ by $\mathrm{SL}(2, \mathbb{C})$
coincides with the symplectic quotient by $\mathrm{SU}(2)$ and consequently is
$\mathrm{Aff} \mathrm{Gr}_2(\C^n)$, \cite{KempfNess}, see also \cite{Schwarz}, Theorem
4.2. Note that since $\mathrm{SU}(2)$ is simple the question of normalizing the
momentum map does not arise.  It will be
important to
understand this in terms of the
subring $\mathbb{C}[M_{2,n}(\mathbb{C})]^{\mathrm{SL}(2, \mathbb{C})}$ of
invariant
polynomials
on $M_{2,n}(\mathbb{C})$. In what follows let $Z_{ij} = Z_{ij}(A)$ or $[i,j]$ be the
determinant
of the $2$ by $2$ submatrix of $A$ given by taking columns $i$ and $j$ of $A$
( the $ij$-th Pl\"ucker coordinate or bracket). The following result on one of
the basic results in invariant theory, see \cite{Dolgachev},Chapter 2.

\begin{proposition}\label{PnGGIT}
$\mathbb{C}[M_{2,n}(\mathbb{C})]^{\mathrm{SL}(2,
\mathbb{C})}$ is generated by
the $Z_{ij}$ subject to the Pl\"ucker relations.
\end{proposition}
But the above ring is the homogeneous coordinate ring of $\mathrm{Gr}_2(\C^n)$.
This gives the required invariant-theoretic proof of Theorem \ref{polygonGrass}.

\begin{remark}
In fact we could do {\em all} of the previous analysis in terms
of invariant theory by using Example 6.12 of
\cite{GuilleminJeffreySjamaar} to replace the
imploded cotangent bundle $\mathcal{E}T^{\ast}\mathrm{SU}(2)\cong \C^2$ by the
quotient $\mathrm{SL}(2,\C)\q N \cong \C^2$ where $N$ is the subgroup of $\mathrm{SL}(2,\C)$ of
strictly
upper-triangular matrices. We leave the details to the reader.
\end{remark}
We have now proved the first statement in Theorem \ref{polygonGrasspreciseversion}
for both symplectic and GIT quotients. We note the consequence
(since every element of $\bigwedge^2(\C^3)$ is decomposable)

\begin{equation}\label{P3}
P_3(\mathrm{SU}(2)) \cong \mathrm{\bigwedge} \space^2(\C^3).
\end{equation}
It is clear that $\psi$ is equivariant as claimed and the second statement
follows.

Since the grading action of $\lambda \in \C^{\ast}$ on 
$\mathrm{Aff}\mathrm{Gr}_2(\C^n)$ is the action that
scales each Pl\"ucker coordinate by $\lambda$ the third statement is also clear.
It remains to prove the fourth statement of Theorem \ref{polygonGrasspreciseversion}.
First recall  the standard $\mathrm{U}(n)$ invariant positive definite
Hermitian form $( \ , \ )$ on $\bigwedge^2(\C^n)$ is given on the
bivector $a \wedge b$ by
$$ (a\wedge b, a \wedge b ) = det \begin{pmatrix} (a,a) & (a,b) \\
(b,a) & (b,b)
\end{pmatrix}.$$
We can obtain $\mathrm{Gr}_2(\C^n)$ as a symplectic manifold  the cone $\mathrm{Aff} \mathrm{Gr}_2(\C^n)$ of decomposable
bivectors by taking the subset of decomposable bivectors of norm squared $R$
then dividing by the action of the circle by scalar multiplication
(symplectic quotient by $S^1$ of level $\frac{R}{2}$). Thus we can find
a manifold diffeomorphic to $\mathrm{Gr}_2(\C^n)$ by computing the pull-backs of
the function $f(a\wedge b) = (a \wedge b, a \wedge b)$ and scalar multiplication
of bivectors by the circle under the map $\varphi$. We represent elements of $P_n(\mathrm{SU}(2))$
by left $\mathrm{SU}(2)$ orbit-equivalence classes $[A]$ of $2$ by $n$ matrices $A$ such that the rows
$Z$ and $W$ of $A$ are orthogonal and of the same length.

\begin{lemma}
$$\varphi^{\ast} f (A) = \Bigg(\frac{||Z||^2 + ||W||^2}{2}\Bigg)^2.$$
\end{lemma}

\begin{proof}
We have
$$\varphi^{\ast} f (A) = det \begin{pmatrix} (Z,Z) & (Z,W) \\
(W,Z) & (W,W)
\end{pmatrix}.$$
But $(Z,Z) = (W,W)$ and $(Z,W) = (W,Z) = 0$.
Hence
$$\varphi^{\ast} f (A) = ||Z||^2 ||W||^2 = \Bigg(\frac{||Z||^2 + ||W||^2}{2}\Bigg)\Bigg(\frac{||Z||^2 + ||W||^2}{2}\Bigg).$$
\end{proof}
Now we ask the reader to look ahead and see that in subsection \ref{frommatricestopolygons}
we define a map $F_n$ from $P_n(\mathrm{SU}(2))$ to the space of $n$-gons in $\R^3$.
Furthermore $F_n([A])$ is the $n$-gon underlying the imploded spin-framed $n$-gon
represented by $[A]$. Accordingly we obtain the following corollary of the
above lemma.

\begin{corollary}
$\varphi^{\ast} f (A)$ is four times the square of the perimeter of the $n$-gon
in $\R^3$ underlying the imploded spin-framed $n$-gon represented by $[A]$.
\end{corollary}
Now  observe that scaling the bivector $Z \wedge W$ by $\exp{i\theta}$
corresponds to scaling the class of the matrix $[A]$ by the square-root of $\exp{i\theta}$,
Note that the choice square-root is irrelevant since
$$[-A] = [A].$$
But scaling the columns of $A$ corresponds to rotating the corresponding imploded
spin-frames.
Thus fixing the perimeter of the underlying $n$-gon to be one corresponds to
taking the symplectic quotient of the cone of decomposable bivectors by $S^1$ at level $2$.
and we have now completed the proof of the second part of Theorem \ref{polygonGrass}.

\subsection{The momentum polytope for the $T_{\mathrm{U}(n)}=T_{\mathrm{SU}(2)^n}$ action on $P_n(\mathrm{SU}(2))$}
We have identified the momentum map for the action of $\mathrm{SU}(2)$
on $M_{2,n}(\C)$. The torus $T_{\mathrm{U}(n)}$ acts on $M_{2,n}(\C)$ by scaling
the
columns.
We leave the proof of the following lemma to the reader.
Let $\mu_{T_n}$ be the momentum map for the above action of $T_{\mathrm{U}(n)}$.

\begin{lemma}$$\mu_{T_n}(A) = \Bigg(\frac{||C_1||^2}{2},\cdots ,
\frac{||C_n||^2}{2}\Bigg).$$
\end{lemma}
The action of $T_{\mathrm{U}(n)}$ descends to the quotient $P_n(\mathrm{SU}(2))$ with the
same momentum map (only now $A$ must satisfy $\mu_G(A) =0$). Thus we have
a formula for the momentum map for the action of $T_{\mathrm{U}(n)}$ on the
Grassmannian
in terms of special representative matrices $A$ (of $\mathrm{SU}(2)$-momentum
level zero). It will be important to extend this formula to
all $A \in M_{2,n}(\C)$ of rank two. We do this by relating the
above  formula to
the usual formula for momentum map of the action of $T_{\mathrm{U}(n)}$ on a general
two
plane in $\C^n$ represented by a general rank 2 matrix $A \in M_{2,n}(\C)$.
This momentum map
is described in Proposition 2.1 of \cite{GGMS}.
Let  $A$ be a $2$ by $n$ matrix of rank two and $[A]$
denote the two plane in $\C^n$ spanned by its columns.
Then the $i$-th component of the momentum map of $T_{\mathrm{U}(n)}$ is given
by$$\mu_i([A]) = \frac{1}{\|A\|^2}\sum_{j,j \neq i} \|Z_{ij}(A)
\|^2.$$
Here $\|A\|^2$ denotes the sum of the squares of the norms of the
Pl\"ucker coordinates of $[A]$.
The required extension formula is them implied by

\begin{lemma}\label{equalityofHamiltonians}
Suppose that $\mu_G(A) = 0$. Let $C_i$ be the $i$-th column of $A$.
Then

$$\frac{\|C_i\|^2 }{2}= \frac{1}{\|A\|^2}\sum_{j,j\neq i}
\|Z_{ij}(A) \|^2.$$

\end{lemma}

\begin{proof}
It suffices to prove the formula in the special case that $i=1$.
If the first row of $A$ is zero then both sides of the equation are zero.
So assume the first row is not zero. Apply $g$ to $A$ so that the first
row of $Ag$ is of the form $(r,0)$ where $r = \|R_1\|$. Let $z_i', w_i'$
be the $i$-th row of $Ag$. Note
that $\mu_G(Ag)$ is still zero, hence the length of the second column of $Ag$
is equal to $(1/2)\|Ag\|^2 = (1/2)\|A\|^2$.  Also $Z_{ij}(Ag) =
Z_{ij}(A)$. Now compute
the left-hand side of the equation.
We have $ Z_{1j}(A) = Z_{1j}(Ag) = r w_j', 2 \leq j \leq n$ and hence (noting
that $w_1' =0$) we have$$\sum_{j,j >1} \|Z_{1j}(A) \|^2 = r^2
\sum_{j=1}^n |w_j'|^2.$$
Note that the sum on the right-hand side is the length squared of the
second column of $Ag$. The lemma follows.
\end{proof}
Let $D_n \subset \Delta^n$ be the cone defined by$$D_n = \{ \mathbf{r}
\in \Delta^n : 2 r_i \leq \sum_{j=1}^n r_j, 1 \leq i \leq n\}.$$
Then $D_n$ is the cone on the hypersimplex $\Delta^2_{n-2}$ of
\cite{GGMS} \S 2.2.
Note that $D_3$ is the set of nonnegative real numbers satisfying the
usual
triangle inequalities.
We can now apply Lemma \ref{equalityofHamiltonians} and the results of
\S 2.2 of \cite{GGMS} to
deduce

\begin{proposition}\label{triangleinequalities}$$\mu_{T_n}(P_n(\mathrm{SU}(2))
= D_n.$$
\end{proposition}
We give another proof of this theorem in terms of the side-lengths
inequalities
for Euclidean $n$-gons in the next subsection.

\subsection{The relation between points of $P_n(\mathrm{SU}(2))$ and
framed Euclidean $n$--gons}\label{frommatricestopolygons}

The point of this subsection is to show that a point of $P_n(\mathrm{SU}(2))$
may be interpreted as a Euclidean $n$--gon equipped with an imploded
spin-frame. The critical role in this interpretation is played by
a map $F:\C^2 \to \R^3$.  The map $F$ is  the Hopf map.

\subsubsection{The equivariant map $F$ from $\calE T^\ast (\mathrm{SU}(2))$ to
$\R^3$}\label{maptoR3}

In this section we define a map $F$ from
the imploded cotangent bundle $\calE T^\ast (\mathrm{SU}(2)) \cong \C^2$
to $\R^3$. This map will give the critical connection between
the previous constructions and polygonal linkages in $\R^3$.

We first define $\pi: \calE T^\ast (\mathrm{SU}(2)) \to \mathfrak{su}^\ast(2)$
by$$\pi([g, \lambda \varpi_1]) = Ad^\ast(g)(\lambda \varpi_1),$$
where $\varpi_1$
denotes the fundamental weight of $\mathrm{SU}(2)$.
We leave the following lemma to the reader.

\begin{lemma}\label{circlebundle}
$\pi$ factors through the action of
$T_{\mathrm{SU}(2)}$. The map $\pi$ restricted
to
the complement of the point $[e,0]$ is a principal $T_{\mathrm{SU}(2)}$ bundle
(the Hopf bundle).
\end{lemma}
Define the normal frame bundle of $\R^3-{0}$ to be the set
of triples $\mathbf{F} = (v,u_1,u_2)$ where $v$ is a nonzero vector in $\R^3$
and the pair $u_1,u_2$
is a properly oriented frame for the normal plane to $v$. The projection
to $v$ is a principal circle bundle. The normal frame
bundle is homotopy equivalent to $\mathrm{SO}(3,\R)$ and admits a unique nontrivial
two-fold cover which again is a principal circle bundle over $\R^3 -0$
which we call the bundle of normal spin-frames for $\R^3 -0$.

In order to prove Proposition \ref{framedngons} below we will need the
following lemma. We leave its proof to the reader. In what follows we
identify the coadjoint action of $\mathrm{SU}(2)$ with its action (through its
quotient $\mathrm{SO}(3,\R)$) on $\R^3$.

\begin{lemma}\label{framedvectors}\hfill
\begin{enumerate}\item  The map
$\pi$ restricted to the subset of $\mathcal{E}T^\ast (\mathrm{SU}(2))$ with
$\lambda \neq 0$ is the spin-frame bundle of $\R^3 -0$. \item The
action of $T_{\mathrm{SU}(2)}$ on $\mathcal{E}T^\ast (\mathrm{SU}(2))$ preserves the
set where $\lambda \neq 0$ and induces the circle
action on the normal spin-frames.
\end{enumerate}

\end{lemma}
We have that $\mathcal{E}T^\ast (\mathrm{SU}(2))$ is symplectomorphic to $\C^2$
and $\mathfrak{su}^\ast(2)$ is isomorphic as a Lie algebra to $\R^3$,
and so
the map $\pi$ induces a map from $\C^2$ to $\R^3$.  The induced map
is the map $F$. We now give details.

Identify $\mathfrak{su}^\ast(2)$ with the traceless Hermitian matrices
$\mathcal{H}_2^0$
via the pairing $\mathcal{H}_2^0 \times \mathfrak{su}(2) \to \R$ given
by
$$\langle X, Y \rangle = \Im (\mathrm{tr}(XY)).$$
In this
identification
we have
$$\varpi_1 = \begin{pmatrix}\frac{1}{2} & 0 \\0 & -\frac{1}{2}
\\\end{pmatrix}.$$
We recall there is a Lie algebra isomomorphism $g$ from $\R^3$ to
the traceless skew-Hermitian matrices given by  given
by
$$g(x_1,x_2,x_3) = \frac{1}{2}\begin{pmatrix}ix_1 & ix_2 - x_3
\\ix_2 + x_3 & -ix_1\end{pmatrix}.$$
The isomorphism $f$ dual to $g$ from
the traceless Hermitian matrices $\mathcal{H}_2^0$ to $\R^3$ is given
by
$$ f\begin{pmatrix}x_1 & x_2 + \sqrt{-1} x_3 \\x_2 - \sqrt{-1}x_3 &
-x_1 \\\end{pmatrix} = (x_1,x_2,x_3).$$
Note that in particular, $f(\varpi_1) = (1/2,0,0) $.

We  now define the map $F$.
Define $h:\C^2 \to \mathcal{H}_2$ as follows. Recall
$$\mu_{\mathrm{U}(2)}(z,w)
= 1/2 \begin{pmatrix}z \bar{z} & w \bar{z} \\z \bar{w} & w \bar{w}
\\\end{pmatrix}.$$
Then define $h(z,w) = \mu_{\mathrm{U}(2)}(z,w)^0$ where the superscript zero
denotes traceless
projection. Hence$$ h(z,w) = (1/4)\begin{pmatrix}z\bar{z} - w \bar{w}
& 2w\bar{z} \\  2z \bar{w} & w\bar{w} - z \bar{z} \\\end{pmatrix}.$$
We define $F$ by$$F = f \circ h.$$

\begin{remark}
The map $h$ is the momentum map for the action of $\mathrm{SU}(2)$ on $\C^2$
after we identify $\mathcal{H}_2^0$ with the dual of the Lie algebra
  of $\mathrm{SU}(2)$ using the imaginary part of the trace. Accordingly
  we will use the notations $h$ and $\mu_{\mathrm{SU}(2)}$ interchangeably.
\end{remark}
In the next lemma we note an important equivariance property of $F$.
Let $\rho:\mathrm{SU}(2) \to \mathrm{SO}(3,\R)$ be the double cover. We leave its
proof to the reader.

\begin{lemma}\label{equivariance}
Let $g \in \mathrm{SU}(2)$. Then$$F \circ g = \rho(g) \circ F.$$
\end{lemma}
We next have

\begin{lemma} We have a commutative diagram
\[ \begin{CD}\calE T^\ast (\mathrm{SU}(2)) @>\pi>> \mathfrak{su}(2)^{\ast}
\\@V\phi VV  @VVfV \\\C^2 @>F>> \R^3\end{CD}\]
\end{lemma}

\begin{proof}
First observe that $f\circ \pi$ and $F\circ \phi$ are homogeneous of
degree one with respect to the $\R_+$ actions on their domain and range
(note that $\phi$ is homogeneous of degree $1/2$ and $F$ is homogeneous
of degree two). Also both intertwine the $\mathrm{SU}(2)$ action on $\calE
T^\ast (\mathrm{SU}(2))$
with its action through the double cover on $\R^3$. Since the action
of $\mathrm{SU}(2) \times \R_+$ on $\calE T^\ast (\mathrm{SU}(2))$ is transitive away
from $[e,0]$ it suffices to prove the two above maps coincide at the
point $[e,\varpi_1]$. But it is immediate that both maps take the
value $(1/2,0,0)$ at this point.
\end{proof}

\begin{remark}\label{SpinF}
We will need the following calculation.

$$F \circ \phi([g, \lambda\varpi_1]) = \lambda f( Ad_g^*(\varpi_1))$$
In particular this implies that $\|F \circ \phi([g, \lambda\varpi_1])\| = \frac{1}{2}\lambda$.
\end{remark}
The following lemma is a direct calculation.

\begin{lemma}\label{factoroftwo}
The formula for $F: \C^2 \to \R^3$ in the usual coordinates is$$F(z,
w) = \frac{1}{4}\Big(z \bar{z} - w \bar{w}, \; 2\Re(w \bar{z}), \;
2\Im(w\bar{z})\Big).$$
Consequently, the Euclidean length of the vector $F(z,w) \in \R^3$ is
given by$$ \|F(z,w)\| = \frac{1}{4}(|z|^2 + |w|^2) .$$
\end{lemma}

\begin{corollary}\label{SpinLength}
Note that the length of $F$ is related to the Hamiltonians
$f_{\alpha^{\vee}}(z,w)$
for the infinitesimal action of $\mathfrak{t}$ by$$||F(z,w)|| = (1/2)
f_{\alpha^{\vee}}(z,w) .$$
Also, by the above remark if $\phi([g, \lambda\varpi_1]) = (z, w)$ then
$$\lambda = f_{\alpha^{\vee}}(z,w).$$
\end{corollary}
Later we will need the following determination of the fibers of $F$.

\begin{lemma}\label{phasefactor}
$$F(z_1,w_1) = F(z_2,w_2) \iff z_1 = c
z_2 \ \text{and} \ w_1 = c w_2 \ \text{with} \ |c|
=1.$$
\end{lemma}

\begin{proof}
The implication $\Leftarrow$ is immediate. We prove the reverse
implication.
Thus we are assuming the equations
\begin{align*}
|z_1|^2 - |w_1|^2 &= |z_2|^2 - |w_2|^2  \\
z_1 \bar{w_1} &= z_2 \bar{w_2}\end{align*}
Square each side of the first equation.
Take four times the norm squared of each side of the second equation and
adding the
resulting equation to the new first equation to obtain$$ (|z_1|^2 +
|w_1|^2)^2 = (|z_2|^2 + |w_2|^2)^2.$$
Hence $|z_1|^2 + |w_1|^2=|z_2|^2 + |w_2|^2$ and consequently$$|z_1|^2
= |z_2|^2 \ \text{and} \ |w_1|^2 = |w_2|^2.$$
Now it is an elementary version of the first fundamental theorem
of invariant theory that if we are given two ordered pairs of vectors
in the plane so that the lengths of corresponding vectors
are equal and the symplectic inner products between the two
vectors in each pair coincide then there is an element in $\mathrm{SO}(2)$
that carries one ordered pair to the other (this is one definition
of the oriented angle between two vectors).
\end{proof}

We now construct the required map to Euclidean $n$-gons.
We define a map $F_n:P_n(\mathrm{SU}(2)) \to (\R^3)^n/\mathrm{SO}(3,\R)$ by defining
$F_n(A)$ to be the orbit of  $(F(C_1),\ldots,F(C_n))$ under the diagonal
action of $\mathrm{SO}(3,\R)$. Here $C_i$ is the $i$-th column of $A$.
Let $\mathrm{Pol}_ n(\R^3)$ denote the space of closed $n$-gons in
$\R^3$.

\begin{theorem}\label{framedngons}\hfill
\begin{enumerate}
\item $F_n$ induces a homeomorphism from $P_n(\mathrm{SU}(2))/T_{\mathrm{SU}(2)^n}$ onto
$\mathrm{Pol}_ n(\R^3)/\mathrm{SO}(3,\R)$.
\item The fiber of $F_n$ over a Euclidean $n$-gon
 is
naturally homeomorphic to the set
of imploded spin framings of the edges of that
$n$-gon.
\item The side-lengths of $F_n(A)$ are related to the norm squared of the columns
of $A$ by
$$||e_i(F_n(A))|| = \frac{||C_i||^2}{4} = \frac{|z_i|^2 + |w_i|^2}{4}.$$
Here $e_i(F_n(A)) = F_n(C_i)$ is the $i$-th edge of any $n$-gon in the
congruence class represented by $F_n(A)$.

\end{enumerate}
\end{theorem}

\begin{proof}
We first prove that if $A$ is in the zero level set of $\mu_G$
then $F_n(A)$ is a {\it closed} $n$-gon.
Recall $C_i,1 \leq i \leq n$ be the $i$-th column of $A$. Then we
have$$F_n(A) = (f((C_1C_1^*)^0),\cdots,f((C_nC_n^*)^0)).$$
Hence the sum $s$ of the edges $F_n(A)$ is given by$$s = f((C_1C_1^* +
\cdots + C_nC_n^*)^0).$$
Now it is a formula in elementary matrix multiplication that we
have$$AA^* = C_1C_1^* + \cdots + C_nC_n^*.$$
But by Lemma \ref{levelzero} we find that $AA^*$ is scalar
whence$(AA^*)^0 =0$ and$$s = f((AA^*)^0) = f(0) = 0.$$
We next prove that $F_n$ is onto.
It is clear that $F_n$ maps $M_{n,2}(\C)$ onto $(\R^3)^n$ (this may
be proved one column at a time). Let $\mathbf{e}=(e_1,\cdots,e_n)$
be the edges of a closed $n$-gon in $\R^3$. Choose $A \in M_{2,n}(\C)$
such that $F_n(A) = \mathbf{e}$. But by the above$$f((AA^*)^0) = e_1 +
\cdots + e_n = 0.$$
Hence $(AA^*)^0 =0$ whence $AA^*$ is scalar and $\mu_G(A) = 0$.
We now prove that $F_n$ is injective.
Suppose there exists $g \in \mathrm{SU}(2)$ such that $F_n(A) = \rho(g) F_n(A')$.
Then $F(A) = F(gA')$. Hence by Lemma \ref{phasefactor} we have $A = g A' t$
for some $t \in T_{\mathrm{U}(n)}$.

The second statement follows because the action of $T_{\mathrm{U}(n)}$ corresponds to
a transitive action on the set
of imploded spin framings.  Indeed using the identification above between
between $\R^3$ and $\mathfrak{su}(2)^{\ast}$ we may replace $F_n$ by the
map $\pi_n :\mathcal{E}T^\ast(\mathrm{SU}(2))^n \to (\mathfrak{su}(2)^{\ast})^n$ given
by$$\pi_n([g_1,\lambda_1 \varpi_1],\cdots,[g_n,\lambda_n \varpi_i])=
(Ad^{\ast}g_1(\lambda_1 \varpi_1), \cdots, Ad^{\ast}g_n(\lambda_n
\varpi_1)).$$ If no $\lambda_i$ is zero  by Lemma \ref{framedvectors}
the (classes of ) the $n$-tuple $[g_1, \lambda_1 \varpi_1],\cdots,
[g_n,\lambda_n \varpi_1]$ represent the imploded
spin-frames over the $n$ vectors in the image.
Then by Lemma \ref{circlebundle} and Lemma \ref{framedvectors} two
such $n$--tuples correspond to
frames over the same image $n$-gon if and only the two $n$--tuples are
related
by right multiplication by $t_1,\cdots, t_n$ with $t_i$
fixing $\varpi_1$ for all $i$. Hence  if no $\lambda_i$ is zero then
the two $n$--tuples are related as above if
and only if they are in the same $T_{\mathrm{U}(n)}$ orbit by definition of
the $T_{\mathrm{U}(n)}$ action. Since $T_{\mathrm{U}(n)}$ acts transitively
on the spin-frames over a given $n$-gon we see that in this case the
fiber of $\pi_n$ is the set of spin-frames over the image $n$-gon and the second statement
is proved.  If some subset  of the $\lambda_i$'s is zero then
we replace the corresponding $g_i$'s by the identity (this does not
change the imploded frame) and procede as above with the remaining
components.

The last statement in the theorem is Lemma \ref{factoroftwo}.

\end{proof}

Let $\sigma:\mathrm{Pol}_ n(\R^3) \to \R_+^n$ be the map that assigns
to a closed $n$-gon the lengths of its sides. It is standard (see for
example the Introduction of \cite{KapovichMillson}) that
the image of $\sigma$ is the polyhedral cone $D_n$.
We can now give another proof of Proposition \ref{triangleinequalities}
based on Euclidean geometry.
We restate it for the convenience of the reader.

\begin{proposition}$$\mu_{T_n}(P_n(\mathrm{SU}(2))) = D_n.$$\end{proposition}\begin{proof}
The proposition follows from the commutative diagram\[
\begin{CD}P_n(\mathrm{SU}(2)) @>\mu_{T_n}>> \R_+^n \\@VF_n VV  @VVIV
\\ \mathrm{Pol}_n(\R^3) @>\sigma>>\R_+^n\end{CD}\]\end{proof}

\subsection{The space $P_n(\mathrm{SO}(3,\R)$}

We now compute $P_n(\mathrm{SO}(3,\R))$ using the fact that $\mathrm{SO}(3,\R)$ is covered by $\mathrm{SU}(2)$.
Let the semi simple Lie Group $\bar{G}$ be a quotient of the semi simple Lie
Group $G$
by a finite group $\Gamma$

$$\Gamma \rightarrow G \xrightarrow{\pi} \bar{G}$$

In Example 4.7 of \cite{GuilleminJeffreySjamaar}, Guillemin-Jeffrey-Sjamaar, show that
this quotient gives a description nomeomorphism $\phi:\mathcal{E}(T^*G)/\Gamma \to  \mathcal{E}(T^* \overline{G}))$. The following lemma is left to
the reader to prove.

\begin{lemma}\label{quotientimplosion}
The homeomorphism $\phi$ induces a homeomorphism from $\Gamma \ql P_n(G)$
to $P_n(\overline{G})$.
\end{lemma}

\begin{theorem}\label{compare} The double cover $\phi:\mathrm{SU}(2) \to \mathrm{SO}(3,\R)$
induces a homeomorphism from $H \ql P_n(\mathrm{SU}(2))$ to $P_n(\mathrm{SO}(3,\R))$.
Here  H is the finite 2-group $Z(\mathrm{SU}(2))^n$ and $Z(\mathrm{SU}(2)) \cong \mathbb{Z}/2$
is the center of $\mathrm{SU}(2)$.
\end{theorem}

We  have the following analogue of Theorem \ref{framedngons} above.
We leave its proof to the reader. We now have orthogonal (imploded) framings
instead of imploded spin framings.  We note that the map
$F_n$ above descends to a map $\overline{F}_n: P_n(\mathrm{SO}(3,\R)) \to
\mathrm{Pol}_ n(\R^3)/\mathrm{SO}(3,\R)$.

\begin{theorem}\label{orthogonalframedngons}\hfill
\begin{enumerate}
\item $\overline{F}_n$ induces a homeomorphism from $P_n(\mathrm{SO}(3,\R))/T_{\mathrm{SO}(3,\R)^n}$ onto the polygon space
$\mathrm{Pol}_ n(\R^3)/\mathrm{SO}(3,\R)$.
\item The fiber of $\overline{F}_n$ over a Euclidean $n$-gon
of the induced map from $P_n(\mathrm{SU}(2))$ onto $\mathrm{Pol}_ n(\R^3)/\mathrm{SO}(3,\R)$ is
naturally homeomorphic to the set
of imploded orthogonal framings of the edges of that
$n$-gon.
\end{enumerate}
\end{theorem}

\section{Toric degenerations associated to trivalent trees}

Suppose one has a planar regular $n$-gon subdivided into
triangles. We call this a triangulation of the $n$-gon.  The dual
graph is a tree with $n$ leaves and $n-2$ internal trivalent nodes
(see Figure \ref{fig:triangulated_hexagon}).  We will see below
how the tree $\tree$ determines a Gr\"obner degeneration of the
Grassmannian $\mathrm{Gr}_2(\C^n)$ to a toric variety. These toric
degenerations first appeared in \cite{SpeyerSturmfels}.

\begin{figure}[htbp]
\centering
\includegraphics[scale = 0.2]{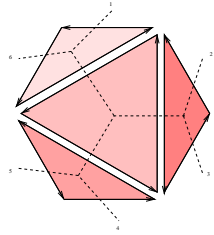}
\caption{ A triangulated hexagon with associated dual tree.}\label{fig:triangulated_hexagon}
\end{figure}

\subsection{Toric degenerations of $\mathrm{Gr}_2(\C^n)$}

Recall that the standard coordinate ring $R = \C[\mathrm{Gr}_2(\C^n)]$ (for
the Pl\"ucker embedding) of the Grassmannian is generated by
$Z_{i,j}$, for $1 \leq i < j \leq n$, subject to the
quadric relations, $Z_{i,j}Z_{k,l} - Z_{i,k}Z_{j,l} + Z_{i,l}Z_{j,k} = 0$ for
$1 \leq i < j < k < l \leq n$. These relations generate the
Pl\"ucker ideal $I_{2,n}$.  For each pair of indices $i,j$ let
$w_{i,j}^\tree$ denote the length of the unique path in $\tree$
joining leaf $i$ to leaf $j$. For example, in Figure
\ref{fig:triangulated_hexagon} we have $w_{1,4}^\tree = 5$. To any
monomial $m = \prod_k Z_{i_k, j_k}$ we assign a weight $w^\tree(m) =
\sum_k w_{i_k,j_k}^\tree.$  Let $I_{2,n}^\tree$ denote the initial
ideal with respect to the weighting $w^\tree$.

It is a standard result in the theory of Gr\"obner degenerations
that one has a flat degeneration of $\C[\mathrm{Gr}_2(\C^n)] =
\C[\{Z_{i,j}\}_{i < j}] / I_{2,n}$ to $\C[\{Z_{i,j}\}_{i<j}] /
I_{2,n}^\tree$. Below we outline how this works using the Rees
algebra.

The weight $w^\tree$ induces an (increasing) filtration on the
ring of the Grassmannian.  Let $F_m^\tree$ be the vector subspace
of $R$ spanned by monomials of weight at most $m$. Then, for any
elements $x \in F_{m_1}^\tree$ and $y \in F_{m_2}^\tree$, the
product $xy$ belongs to $F_{m_1 + m_2}^\tree$. Let $R^\tree$
denote the associated graded ring $R^\tree = \bigoplus_{m=0}^\infty F_m^\tree / F_{m-1}^\tree$,
where $F_{-1}^\tree := 0$.
The Rees algebra $\mathcal{R}^\tree$ is given by,
$\mathcal{R}^\tree = \bigoplus_{m=0}^\infty t^m F_m^\tree$,
where $t$ is an indeterminant. Then (cf. \cite{AlexeevBrion}) 
$\mathcal{R}^\tree$ is flat over $\C[t]$, 
$\mathcal{R}^\tree \otimes_{\C[t]} \C[t,t^{-1}] \cong
R[t,t^{-1}]$, and $\mathcal{R}^\tree \otimes_{\C[t]} (\C[t]/(t))
\cong R^\tree$.

Since the ring $R$ is already graded, there is a natural grading
of the Rees algebra $\mathcal{R}^\tree$, where the degree of the
indeterminant variable $t$ is defined to be zero, and in general,
the degree of a product $t^m x$, where $x \in R^{(k)} \cap
F_m^\tree$, is equal to $\deg(x) = k$. Now, the specializations
$\mathcal{R}^\tree \otimes_{\C[t]} \C[t]/(t-a)$ (at $t = a$) are
all graded algebras, where $\deg(y \otimes \overline{p(t)}) =
\deg(y)$ for $y \in \mathcal{R}^\tree$. In particular the
associated graded algebra $R^\tree$ is bigraded; one grading comes
from the grading of $R$, and the other from the filtration of $R$.
We have that $\mathrm{Proj} (\mathcal{R}^\tree)$ has a flat
morphism to the affine line $\mathbb{A}^1 = \mathrm{Spec} (\C[t])$
with all fibers  projective varieties, with general fiber
isomorphic to $\mathrm{Gr}_2(\C^n)$ and special fiber $\mathrm{Gr}_2(\C^n)_0^\tree$ at
$t=0$.

\begin{definition}
We shall reserve a special name for the case that the
triangulation is such that all diagonals contain the initial
vertex, forming a fan.  The associated tree has the appearance of
a caterpillar and was called such in \cite{SpeyerSturmfels}. We
shall call this triangulation the ``fan'' and the associated
degeneration the LG--degeneration, since the special fiber in this
case agrees with the special fiber of the degeneration of
$\mathrm{Gr}_2(\C^n)$ given by Lakshmibai and Gonciulea in
\cite{LakshmibaiGonciulea}.
\end{definition}

Now we will see why the special fiber $\mathcal{R} \otimes_{\C[t]}
\C[t]/(t)$ is toric.  This was also proved in
\cite{SpeyerSturmfels} where these degenerations were first
discovered. First we establish a basis for $\C[\mathrm{Gr}_2(\C^n)]$ as a
complex vector space. Suppose that the vertices $1,2,\ldots,n$ of
a multigraph (multiple edges allowed between two vertices) are
drawn on the unit circle, in cyclic clockwise order, and all edges
are drawn as straight line segments (a chord of the circle). Then,
if no two edges cross (it is okay for an edge to have multiplicity
greater than one) we call the graph a Kempe graph, in honor of the
work of A. Kempe in 1894 (see \cite{Kempe}). For an example see
Figure \ref{fig:Kempe_graph}.  In fact two edges $ij$ and $kl$
cross exactly when $i < k < j < l$, assuming that $i < k$.

\begin{figure}[htbp]
\centering
\includegraphics[scale = 0.3]{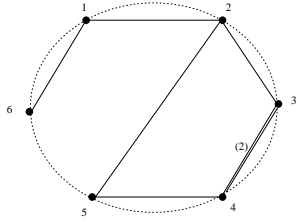}
\caption{ A Kempe graph for $n=6$.}\label{fig:Kempe_graph}
\end{figure}

We assign a monomial $m_G$ to each graph $G$, given by
$$m_G = \prod_{ij \in E(G)} Z_{i,j},$$
where $E(G)$ is the multi-set of edges of $G$. For a graph $G$,
let $\deg(G)$ be the number of edges in $G$, which is one half of
the total sum of valencies of the all the vertices.  Note that
$\deg(G) = \deg(m_G)$. The proof of the proposition and theorem
below appears in \cite{HowardMillsonSnowdenVakil}.
\begin{proposition}
The monomials $m_G$ as $G$ runs over the set of all Kempe graphs
with vertex set $V(G) = \{1,2,\ldots,n\}$ form a $\C$-basis for
$\C[\mathrm{Gr}_2(\C^n)]$. Furthermore, the $m_G$ for which $\deg(G) = k$
and $G$ is a Kempe graph form a basis for the $k$-th graded piece
of $R$.
\end{proposition}

\begin{theorem}\label{leadingterm}
Suppose that $G_1$ and $G_2$ are Kempe graphs.  Suppose that
$$m_{G_1} m_{G_2} = \sum_{\text{$G$ is Kempe}} c_G m_G.$$
Then, there is a unique Kempe graph $G^\ast$ such that
\begin{itemize}
\item $c_{G^\ast} = 1$, \item $w^\tree(m_{G^\ast}) =
w^\tree(m_{G_1}) + w^\tree(m_{G_2})$, \item for all $G$, if $c_G
\neq 0$ and $G \neq G^\ast$ then $w^\tree(m_G) <
w^\tree(m_{G^\ast})$.
\end{itemize}
\end{theorem}

\begin{proof}

Suppose that $i_1 j_1$ is an edge of $G_1$ and $i_2 j_2$ is an
edge of $G_2$ such that $i_1 j_1$ and $i_2 j_2$ cross.  Without
loss of generality we can assume that $i_1 < i_2 < j_1 < j_2$. We
have the Pl\"ucker relation $Z_{i_1,j_1}Z_{i_2,j_2} =
Z_{i_1,i_2}Z_{j_1,j_2} + Z_{i_1,j_2}Z_{i_2,j_1}$, and the latter two terms
(viewed as graphs having two edges) are Kempe graphs since neither
pair of edges cross. Denote $G_1 \cdot G_2$ as the graph with edge
set $E(G_1) \coprod E(G_2)$. Let $G_0 = (G_1 \cdot G_2)$ with the
two edges $i_1 j_1$ and $i_2 j_2$ removed. We have,
$$m_{G_1} m_{G_2} = m_{G_0 \cdot i_1 j_1 \cdot i_2 j_2} =
m_{G_0 \cdot i_1 i_2 \cdot j_1 j_2} + m_{G_0 \cdot i_1j_2 \cdot
i_2j_1}.$$ We will show that one of the two monomials on the right
hand side of the above equation has the same $\tree$--weight
($w^\tree$) as does $m_{G_1} m_{G_2}$ and the other term has
strictly smaller weight.  After sufficiently many Pl\"ucker
relations as above are applied we get $m_{G_1} m_{G_2}$ expressed
as a unique integral combination of Kempe graphs with a unique
term of maximal weight equal to the weight of $m_{G_1} m_{G_2}$.

Let $\gamma(i,j)$ denote the shortest path in $\tree$ joining $i$
to $j$.  Note that the paths $\gamma(i_1,j_1)$ and
$\gamma(i_2,j_2)$ must intersect one another since they cross when
drawn as straight line segments in the graph $G_1 \cdot G_2$. Now
consider the two pairs of paths in $\tree$:

\begin{itemize}
\item $\gamma(i_1,i_2)$ and $\gamma(j_1,j_2)$. \item
$\gamma(i_1,j_2)$ and $\gamma(i_2,j_1)$.
\end{itemize}
Exactly one of the above two above pairs of paths cover precisely
the same set of edges in $\tree$ as does the crossing pair
($\gamma(i_1,j_1)$,$\gamma(i_2,j_2)$.   Say for example it is the
pair ($\gamma(i_1,i_2)$, $\gamma(j_1,j_2)$). Then
$w^\tree(Z_{i_1,i_2}Z_{j_1,j_2}) = w^\tree(Z_{i_1,j_1}Z_{i_2,j_2})$.
Furthermore the two paths $\gamma(i_1,i_2)$ and $\gamma(j_1,j_2)$
must meet one another within the tree $\tree$ along some internal
edges $e_1,\ldots,e_k$ (although they are non-crossing when drawn
as straight line segments).  The edges $e_1,\ldots,e_k$ are also
the edges common to the two paths $\gamma(i_1,j_1)$ and
$\gamma(i_2,j_2)$. The third pair of paths ($\gamma(i_1,j_2)$,
$\gamma(i_2,j_1)$) covers the same set of edges as do the first
two pairs of paths, excepting the edges $e_1,\ldots,e_k$ above. In
particular we have that
$$w^\tree(Z_{i_1,i_2}Z_{j_1,j_2}) = w^\tree(Z_{i_1,j_1}Z_{i_2,j_2}) =
w^\tree(Z_{i_1,j_2}Z_{i_2,j_1}) + 2k.$$
\end{proof}

\begin{definition}
Let $G_1 \ast_\tree G_2$ denote the Kempe graph $G^\ast$ from the
above theorem.   Let $\mathcal{S}_n^\tree$ denote the graded commutative
semigroup of Kempe graphs $G$ with binary operation $(G_1,G_2)
\mapsto G_1 \ast_\tree G_2$.  The grading is given by $\deg(G) =
\deg(m_G)$.
\end{definition}

\begin{corollary} (to Theorem \ref{leadingterm})
The special fiber at $t=0$ of the degeneration, $\mathcal{R}
\otimes_{\C[t]} \C[t]/(t)$, is isomorphic to the semigroup algebra
$\C[\mathcal{S}_n^\tree]$.
\end{corollary}
Given any graph $G$, there is an associated weighting $w_G$ of the
tree $\tree$ where the weight assigned to an edge $e$ of $\tree$
is equal to the number of edges $ij$ in $G$ such that the path
$\gamma(i,j)$ in $\tree$ joining $i$ to $j$ passes through $e$. We
denote this weight by $w_G(e)$.  We now determine the image of the
map $G \mapsto w_G$.

\begin{proposition}\label{admissibleweighting}
Given a weighting $w$ of nonnegative integers to the edges of
$\tree$, then $w = w_G$ for some graph $G$ iff for each triple
$e_1, e_2, e_3$ of edges meeting at a common (internal) vertex of
$\tree$ the sum $w(e_1) + w(e_2) + w(e_3)$ is even, and
$w(e_1),w(e_2),w(e_3)$ satisfy the triangle inequalities.
\end{proposition}

\begin{proof}
Let a $G$--path mean a path in the tree $\tree$ joining vertices
$i$ to $j$ where $ij$ is an edge of $G$.  Thus $G$-paths are in
bijection with edges of $G$, and are meant to be counted with
multiplicity.  Now fix an internal vertex $v_0$ of $\tree$ with
neighboring vertices $v_1, v_2, v_3$ with connecting respective
edges $e_1, e_2, e_3$.  For each $i,j$ with $1 \leq i < j \leq 3$,
let $x_{ij}$ be the number of $G$-paths passing through each of
$v_i$ and $v_j$.  Thus \begin{itemize} \item $w_G(e_1) = x_{12} +
x_{13}$. \item $w_G(e_2) = x_{12} + x_{23}$. \item $w_G(e_3) =
x_{13} + x_{23}$.
\end{itemize}  Now solving for the $x_{ij}$ we have:
\begin{itemize} \item $x_{12} = \frac{1}{2}(w_G(e_1) + w_G(e_2) - w_G(e_3))$.
\item $x_{13} = \frac{1}{2}(w_G(e_1) + w_G(e_3) - w_G(e_2))$.
\item $x_{23} = \frac{1}{2}(w_G(e_2) + w_G(e_3) - w_G(e_1))$.
\end{itemize} Since the $x_{ij}$ are non-negative we have that
$w_G(e_1)$, $w_G(e_2)$, and $w_G(e_3)$ satisfy the triangle
inequalities. Since the $x_{ij}$ are integers we have that
$w_G(e_1) + w_G(e_2) + w_G(e_3)$ is even.

Conversely, given a weighting of $\tree$ satisfying the above
conditions, we will get non-negative integers $x_{12}(v),
x_{13}(v), x_{23}(v)$ at each internal vertex $v$.  We may thus
define a graph $G_v$ on the three neighboring vertices
$v_1,v_2,v_3$ by setting the multiplicity of edge $ij$ to be
$x_{ij}(v)$. If $v'$ is a neighboring internal vertex, say for
example sharing edge $e_1$ with $v$, then we have $x_{12}(v) +
x_{13}(v) = x_{12}(v') + x_{13}(v') = w(e_1)$.  Hence we may glue
together $G_{v}$ and $G_{v'}$ to form a Kempe graph $G_{v,v'}$ on
the neighboring vertices of $v$ and $v'$ by adjoining the edges
whose respective paths go through the edge $vv'$ of $\tree$.  Note
that the weighting of the edges of the tree pertaining to
$G_{v,v'}$ is equal to $w$. Continuing in this way, we may form a
Kempe graph $G$ on the leaves of $\tree$ such that $w = w_G$, by
gluing together all the $G_v$ for $v$ and internal vertex of
$\tree$.
\end{proof}

\begin{definition}
We shall call a weighting $w$ or $\tree$ an admissible weighting
if $w = w_G$ for some graph $G$.
\end{definition}

\begin{proposition}
For each admissible weighting $w$ of $\tree$, there exists a
unique Kempe graph $G$ such that $w = w_G$.  Furthermore, if $G_1$
and $G_2$ are Kempe graphs, then $w_{G_1} + w_{G_2} = w_{G_1
\ast_\tree G_2}$.
\end{proposition}

\begin{proof}
Certainly the paths in the tripods of the proof of Proposition
\ref{admissibleweighting} may be drawn so that they are
non-crossing.  Such a non-crossing tripod graph is unique.
In the process of gluing all these tripod graphs together
to form a graph $G$ such that $w_G = w$, it is clear there is only
one way to join the paths so that no two paths are non-crossing.
Now by straightening these non-crossing paths into line segments,
we see that they remain non-crossing and so form a Kempe graph.

A quick inspection of the proof of Theorem \ref{leadingterm}
reveals that the graph $G_1 \ast_\tree G_2$ has weight equal to
the sum $w_{G_1} + w_{G_2}$.
\end{proof}
Note that $w^\tree(G) = \sum_{ij \in E(G)} w_{i,j}^\tree = \sum_{e
\in E(\tree)} w_G(e)$. We define the degree of an admissible weighting
$w^{\tree}$ to be one -half the sum of the weights of the leaf edges (a leaf edge
is an edge incident to a leaf).  Since each edge of the graph $G$ contributes
a weight of one to two leaf edges we have
$$deg(m_G) = deg(G) = deg(w^{\tree}).$$

\begin{definition}
Define $\mathcal{W}_{n}^{\tree}$ to be the graded semigroup of
admissible weightings of the edges of $\tree$.
\end{definition}

\begin{proposition}\label{gradedsemigroupiso}
The map that associates to an element of $\mathcal{S}^\tree_n$ the induced
weighting of the edges of $\tree$ induces an isomorphism
$\Omega_n:\mathcal{S}_{n}^{\tree} \to \mathcal{W}_{n}^{\tree}$.
\end{proposition}

\subsection{Induced toric degeneration of the space $M_\br$ of polygonal linkages}

Here we briefly review the construction of the GIT quotients of
the Grassmannian by the maximal torus $T$ in $\mathrm{SL}(n,\C)$. The GIT quotient 
(or technically, the subspace of closed points in the GIT quotient) 
is homeomorphic to a 
configuration space $M_\br$ of polygonal
linkages in $\R^3$ with prescribed integral side lengths $\br =
(r_1,\ldots,r_n)$. See \cite{KapovichMillson}, Theorem 2.3, or 
\cite{Klyachko} for more
details.

Let $T$ be the torus of diagonal matrices in $\mathrm{SL}(n,\C)$.  There
is a natural action of $T$ on the Grassmannian $\mathrm{Gr}_2(\C^n)$.  The
GIT quotient $\mathrm{Gr}_2(\C^n) \q T$ depends upon the choice of a
$T$-linearized line bundle of $\mathrm{Gr}_2(\C^n)$.  Let $L$ be the line
bundle associated to the Pl\"ucker embedding of $\mathrm{Gr}_2(\C^n)$.
Let $P \subset \mathrm{SL}(n,\C)$ be the parabolic subgroup fixing the point
$e_1 \wedge e_2 \in \mathrm{Gr}_2(\C^n)$.  It is the subgroup of matrices
$[a_{ij}]_{1 \leq i,j \leq n}$ of determinant one, such that
$a_{ij} = 0$ for $i=1,2$ and $j > 2$. We identify $\mathrm{Gr}_2(\C^n)$
with the homogeneous space $\mathrm{SL}(n,\C)/P$,
by identifying $gP \in \mathrm{SL}(n,\C)/P$ with
$g \cdot (e_1 \wedge e_2) \in \mathrm{Gr}_2(\C^n)$.  Let $\varpi_2 =
(1,1,0,\ldots,0) \in \Z^n$ be the second fundamental weight of
$\mathrm{SL}(n,C)$.  Associated to $\varpi_2$ is a character
$\chi : P \to \C^\ast$ given by $\chi([a_{ij}]_{1 \leq i,j \leq n})
= \det
\begin{pmatrix}
a_{11} & a_{12} \\
a_{21} & a_{22}
\end{pmatrix}$.
Following the Borel-Weil construction, we may take the total space
of $L$ to be the product $\mathrm{SL}(n,\C) \times \C$ modulo
the equivalence relation $(g,z) \sim (gp, \chi(p)z)$, for all
$g \in \mathrm{SL}(n,\C)$, $p \in P$, and $z \in \C$. We denote the equivalence
class of $(g,z)$ as above by $[g,z]$.
The bundle map $\pi : L \to \mathrm{Gr}_2(\C^n)$ is given by
$[g,z] \mapsto gP$.  The Plucker coordinate ring of $\mathrm{Gr}_2(\C^n)$
is now $\bigoplus_{N=0}^\infty \Gamma(\mathrm{Gr}_2(\C^n),L^{\otimes N})$, and it is
generated in degree one by the brackets $[i,j]$
which are associated to global sections $s_{ij}$ of $L$, by
$s_{ij}(gP) := [g, g_{i1}g_{j2} - g_{i2}g_{j1}]$, where
$g = [g_{kl}]_{1 \leq k,l \leq n}$.

We suppose that $|\br| = r_1 + \cdots + r_n$ is an even integer.
The line bundle $L^{\otimes |\br|/2}$ of $L$ may be identified
with the product $\mathrm{SL}(n,\C) \times \C$ modulo the relation
$(g,z) \sim (gp, \chi(p)^{\otimes |\br|/2} z)$.
We define an action of $T$ on $L^{\otimes |\br|/2}$ via the character
$\chi_\br$ of $T$, given by
$\chi_\br(t_1,\ldots,t_n) := \prod_{i=1}^n t_i^{r_i}$.
We define the action of $T$ on $L^{\otimes |\br|/2}$ by
$$(t_1,\ldots,t_n) \cdot [g,z] = [\mathbf{t} \cdot g,
\chi_\br(\mathbf{t})z],$$
where $\mathbf{t} = (t_1,\ldots,t_n) \in T$.  We call this the
$\br$--linearization of $L^{\otimes |\br|/2}$.

The space $M_\br$ is homeomorphic to 
$$\mathrm{Gr}_2(\C^n) \q_\br T = \mathrm{Projm} \Big(
\bigoplus_{N=0}^\infty \Gamma(\mathrm{Gr}_2(\C^n), L^{\otimes N |\br|/2})^T
\Big),$$
where $\Gamma(\mathrm{Gr}_2(\C^n), L^{\otimes N |\br|/2})^T$ indicates
the $T$-invariant global sections of $L^{\otimes N |\br|/2}$.
The ring 
$$R_\br = 
\bigoplus_{N=0}^\infty \Gamma(\mathrm{Gr}_2(\C^n), L^{\otimes N |\br|/2})^T$$
is naturally graded by $N$.
The invariant sections of $L^{\otimes N|\br|/2}$ are spanned by
monomials $m = m_G$ over all graphs $G$ having multi-degree $N\br$,
i.e. the valency of vertex $i$ is $r_i$ for each $i$.
The degree of $m_G$ is then $N$, if $G$ has multidegree $N \br$.

Restricting to torus invariants is an exact functor so the flat
degenerations of the Grassmannian described above restrict to flat
degenerations of $\mathrm{Gr}_2(\C^n) \q_\br T$.  Furthermore, the
special fiber of this restricted degeneration is toric, since it
is the $T$-quotient of a toric variety.
Therefore we obtain a flat degeneration of $\mathrm{Gr}_2(\C^n) \q_\br T$ to a
toric variety $(\mathrm{Gr}_2(\C^n) \q_\br T)^\tree$ for each triangulation $\tree$ of the
model $n$-gon.
The associated semigroup
$\mathcal{S}^\tree_\br$ is the set of Kempe graphs having valency a multiple
of $\br$; it is a sub-semigroup of $\mathcal{S}_n^\tree$, however the grading
of $\mathcal{S}^\tree_\br$ is not the same as that of $\mathcal{S}_n^\tree$.  Instead
the degree of a Kempe graph $G \in \mathcal{S}^\tree_\br$ of multidegree 
$N \br$ is $N$, rather than $N |\br|/2$ as it would have in $\mathcal{S}_n^\tree$.
We have that
$$(\mathrm{Gr}_2(\C^n) \q_\br T)^\tree = \mathrm{Projm}(\C[S^\tree_\br]),$$
with $\mathcal{S}^\tree_\br$ graded as mentioned above.

\begin{definition}
Define the  graded semigroup $\mathcal{W}_{\br}^{\tree}$ to be the graded subsemigroup of
$\mathcal{W}^{\tree}_n$ with leaf-edge weights that are integral multiples of $\br$.
\end{definition}
The admissible weightings of the tree $\tree$ relating
to the sub-semigroup $\mathcal{S}^\tree_\br$ must satisfy that the weighting of
the outer edges $e_1,\ldots,e_n$ (the edge
$e_i$ is adjacent to leaf $i$) is some multiple $N$ of $\br$,
meaning that $w(e_i) = Nr_i$ for each $i$,
and the multiple $N$ is the degree of the weighting.
Thus we have
\begin{proposition}\label{gradedsubsemigroupiso}
The isomorphism $\Omega_n$  of Proposition \ref{gradedsemigroupiso} induces an isomorphism
$\Omega_{\br}:\mathcal{S}^{\tree}_{\br} \to \mathcal{W}_{\br}^{\tree}$.
\end{proposition}

\section{$\tree$-congruence of polygons and polygonal linkages}\label{KYspace}

In this section we will define an equivalence relation on polygons and polygonal linkages 
which depends on the choice of a trivalent tree $\tree$.  First we will collect some results 
about trivalent trees which will be useful in what follows.

\subsection{Trivalent trees and their decompositions into forests}
Let $\tree$ be a trivalent tree with $n$ leaves which we assume is dual
to a triangulation of $P$.  We will
say a vertex is {\it internal} if it is not a leaf. The triangles in the triangulation of $P$ correspond to the internal vertices of $\tree$.
We will say an edge is a {\it leaf edge} or an {\it outer edge} if it is incident to a leaf. Thus the leaf edges are dual to the edges of $P$.
An edge of $\tree$ which does not border a leaf will be called an
{\it inner edge}. Thus the inner edges of $\tree$ are dual to the
diagonals of the triangulation of $P$.

\begin{definition}
We say two leaves of $\tree$ are a matched pair of leaves if they have a common neighbor.
\end{definition}
The following technical lemma will be very useful for giving inductive proofs.

\begin{lemma}\label{trivalentinduction}
For any trivalent tree $\tree$ it is possible
to find a sequence of subtrees

$$\tree_0 \subset \ldots \subset \tree_{n-3} = \tree$$
such that
\begin{itemize}
\item The tree $\tree_0$ is a tripod.
\item The tree $\tree_i$ can be identified with $\tree_{i-1}$ joined with a tripod along some $e_i$.
\item Each internal edge of $\tree$ appears as some $e_i$.
\end{itemize}
\end{lemma}

\begin{proof}
Let $\mathcal{T}$ be a trivalent tree.  let $\mathcal{T'}$ be the (not necessarily trivalent) subtree of $\mathcal{T}$ with vertex set the internal vertices of $\mathcal{T}$ and edge set equal to all edges of $\mathcal{T}$ not connected to leaves.  It is easy to see that $\mathcal{T'}$ is also a tree and therefore has a leaf.  Let $n \in \mathcal{T}$ be the vertex corresponding to this leaf, then $n$ is trivalent by definition and therefore must be connected to two leaves in $\tree$.  This shows that $\tree$ must have a vertex $v$ connected to two leaf edges, $v$ is connected to a third edge, $e$,
by splitting $\tree$ along $e$ we obtain a tripod and a new trivalent tree.  The three items above now follow by induction.
\end{proof}
For each tree $\tree$ we choose once and for all an ordering on the
trivalent vertices of $\tree$ as follows.  Pick any matched pair
of leaves and label their trivalent vertex $\tau_{n-2}$
with $n$ the number of leaves of $\tree$. Then,
repeat this proceedure for the tree given by removing $\tau_{n-2}$
and its matched pair. Note that this ordering is not unique.

We now discuss the decompositions  of $\tree$ obtained
by splitting certain internal edges.
We first describe  the decompositions of $P$ which will induce the required
decompositions of $\tree$. For a diagonal $d$ in the triangulation $\tree$, 
we define the operation of splitting $P$ along $d$ by removing a small
tubular neighborhood of $d$ in $P$.  This results in a disjoint union
of two triangulated polytongs $P' \cup P'' = P^{\{d\}}$. 
We may generalize this operation by performing the same operation for any subset $S$ 
of the set of diagonals of $P$ to obtain a union of triangulated polygons $P^S$. 
The dual graph to $P^S$ is a forest $\tree^S$ that may be obtained by removing
a small open interval from each edge corresponding to a diagonal in $S$. 
If we choose $S$ to  be the entire set of diagonals we obtain the decomposition 
$P^D$ of $P$ into triangles and $\tree$ into the forest $\tree^D$ of trinodes.

It will be important in what follows that we have the quotient map
$$\pi^{S}: \tree^S \to \tree$$
that glues together the pairs of vertices that are the boundaries of
the open intervals removed from $\tree$.
We will say two vertices of $\tree^S$ are equivalent if they have the same
image under $\pi^{S}$ and two edges of $\tree^S$ are equivalent
if they contain equivalent vertices.

There are two types of leaf edges of $\tree^S$.
The first type correspond to edges of $P$ (leaf edges of $\tree$) and the second correspond to (split)
diagonals of $P$.  The $n$ edges of $\tree^S$ corresponding to leaf edges of $\tree$ will from now on be referred to as distinguished
edges, we denote this set $dist(\tree^S)$.  For a connected component $C$ of $\tree^S$ we let $dist(C) = dist(\tree^S)\cap C$.
Note that these components uniquely partition the set $dist(\tree^S) = \bigcup_{C_i \subset \tree^S} dist(C_i)$.

\subsection{$\tree$-congruence of polygonal linkages}

Fix a triangulation of the model n-gon with dual tree $\tree$.
Recall $\mathrm{Pol}_ n(\R^3)$, the space of n-gons  in $\R^3$.
A polygon $\mathbf{e} \in \mathrm{Pol}_ n(\R^3)$ comes with a fixed ordering on its edges,
these ordered edges are in bijection with the leaves of the tree $\tree$.
A set of diagonals $S$ corresponds to a set of internal edges of
$\tree$. We know that such a set defines a
unique partition of the edges of $\mathbf{e} $, $E(\mathbf{e} ) = E_1(\mathbf{e} ) \cup \ldots \cup E_m(\mathbf{e} )$
given by grouping together the distinguished edges of $\tree^S$ that lie in a component tree of the  forest $\tree^S$.

\begin{proposition}
For $\mathbf{e}  \in \mathrm{Pol}_ n(\R^3)$, if all diagonals in the set $S$ are zero then
the edge sets $E_i(\mathbf{e} )$ define closed polygons.
\end{proposition}

\begin{proof}
Split the polygon $P$ along $S$ to obtain a union of polygons.
Choose a  polygon $P_i$ in the union. Then $P_i$ is dual to a component $C_i$ of the forest $\tree^S$.
The edges of the  polygon $P_i$  are either edges of the
orginal original polygon (so distinguished leaves of the tree $C_i$)
and hence edges of the Euclidean polygon $\mathbf{e}$
or diagonals of $S$ and hence zero diagonals of the Euclidean
polygon $\mathbf{e}$. We assume that  we have chosen $i$ consistently
with the division of the distinguished edges above whence
the set of distinguished edges of $C_i$ equals
the set $E_i(\mathbf{e})$. Now since $P_i$ closes up the sum of
all the vectors in $\R^3$ associated to the edges of $P_i$
is zero. But this sum is the sum of the vectors in $\R^3$ associated
to the distinguished edges of $C_i$ (that is the elements
in $E_i(\mathbf{e})$) and a set of vectors all of which are zero.
\end{proof}
The following groups will be useful in defining structures on
spaces of polygons.

\begin{definition}
Let $G^{dist(\tree^S)}$ be the group of maps from
the set $dist(\tree^S)$ into a group $G$.  Define
$G^{D(S)}$ to be the subgroup of $G^{dist(\tree^S)}$
of maps which are constant along the distinguished edges
of each component $C$ of $\tree^S$
\end{definition}
Notice that $G^{D(S)}$ naturally splits as a
product of $G$ over the components of $\tree^S$.
Let $G^{D(C)}$ be the component corresponding
to the component $C$. The sets $S$ of zero diagonals define a filtration of the space $\mathrm{Pol}_ n(\R^3)$
where the subspace $\mathrm{Pol}_ n(\R^3)^S$ in the filtration is defined to be the
set of all points $\mathbf{e} $ such that the diagonals in $S$ have zero length.
This in turn defines a decomposition of $\mathrm{Pol}_ n(\R^3)$ into subspaces.
$$\mathrm{Pol}_ n(\R^3)^{[S]} = \mathrm{Pol}_ n(\R^3)^S \setminus \cup_{S \subset J} \mathrm{Pol}_ n(\R^3)^J$$
The subspace $\mathrm{Pol}_ n(\R^3)^{[S]}$ is the collection of all points in $\mathrm{Pol}_ n(\R^3)$
such that exactly the diagonals in $S$ have zero length.  This also
induces a filtration on $M_{\br}$.

\begin{definition}
Define an action of $\mathrm{SO}(3,\R)^{D(S)}$ on $\mathrm{Pol}_ n(\R^3)^{[S]}$ by
letting $\mathrm{SO}(3, \R)^{D(C_i)}$ act diagonally on the edges in $E_i(\mathbf{e} )$.
The equivalence relation given by $\mathrm{SO}(3,\R)^{D(S)}$ orbit type on $\mathrm{Pol}_ n(\R^3)^{[S]}$
fit together to give an equivalence relation $\sim_{\mathcal{T}}$,
which we call $\tree$-congruence, on $\mathrm{Pol}_ n$.
\end{definition}
We may describe the above equivalence relation geometrically as follows.
A polygon $\mathbf{e} \in \mathrm{Pol}_ n(\R^3)^{[S]}$ is a wedge of $|S| +1$ closed polygons $\mathbf{e}_i$
wedged together at certain vertices of the polygon $\mathbf{e}$. Although each $\mathbf{e}_i$
may contain several wedge points, since the vertices of $\mathbf{e}$ are ordered there
will be a first wedge point $v_i$. Then we apply
a rotation $g_i$ about $v_i$ to each $\mathbf{e}_i$ for $1 \leq i \leq |S|+1$ and identify
points in the resulting orbit of $\mathrm{SO}(3,\R)^{D(S)}$.

$\tree$-congruence for $\mathrm{Pol}_ n(\R^3)$ induces an equivalence relation on the space of $n$-gon linkages  $\widetilde{M}_{\br} \subset \mathrm{Pol}_ n(\R^3)$,
which we also call $\tree$-congruence.

\begin{definition}
Define $$V_{\br}^{\tree} = \widetilde{M}_{\br} / \sim_{\tree}$$
\end{definition}
Kamiyama and Yoshida studied the space $V_{\br}^{\tree}$ for the special case
when $\tree$ was the {\it fan} tree.
Note that $V_{\br}^{\tree}$ inherits a filtration by the subspaces $(V_{\br}^{\tree})^{S}$.
Let $(V_{\br}^{\tree})^{[S]} = (V_{\br}^{\tree})^{S} \setminus \cup_{S \subset J} (V_{\br}^{\tree})^J$.
The spaces  $(V_{\br}^{\tree})^{[S]}$ define a decomposition of $V_{\br}^{\tree}$.
We can say more about the pieces of this decomposition.

\begin{theorem}
$(V_{\br}^{\tree})^{[S]}$ is canonically homeomorphic to $\prod_{C_i \subset \tree^S} M_{\br_i}^o$
where $\br_i$ is the subvector of linkage lengths corresponding to the elements in $dist(C_i)$,
and $M_{\br_i}^o$ is the dense open subset of $M_{\br_i}$ corresponding to polygons
with no zero diagonals.
\end{theorem}

\begin{proof}
Let us first describe a map $\tilde{F}: M_{\br}^{S} \to V_{\br_1} \times \ldots \times V_{\br_k}$.
The diagonals $S$ define a partition of edges sets $E_i(\mathbf{e})$ for each $\mathbf{e} \in \widetilde{M}_{\br}^{S}$.  By the above proposition, $E_i(\mathbf{e})$ corresponds to a closed polygon.
So we may send a member of the equivalence class of $\mathbf{e}$ to the product of the equivalence classes
defining these closed polygons in $V_{\br_1} \times \ldots \times V_{\br_k}$, with the appropriate
$\br_i$.  This map is certainly onto, and by the definition of $V_{\br_i}$, it factors through
the relation $\sim{\tree}$. Hence, we get a 1-1 and onto continuous function
 $F: (V_{\br}^{\tree})^{S} \to V_{\br_1} \times \ldots \times V_{\br_k}$, which is a homeomorphism
by the compactness of $V_{\br_1} \times \ldots \times V_{\br_k}$.

We may restrict this map to the spaces $(V_{\br}^{\tree})^{[S]}$ which define the induced
decomposition of $V_{\br}^{\tree}$.  Recall that these are the polygons with exactly the $S$
diagonals zero.  We therefore obtain that $(V_{\br}^{\tree})^{[S]}$ is homeomorphic to $M_{\br_1}^o \times \ldots \times M_{\br_k}^o$,
\end{proof}

\begin{remark}
Since the fibers of $\pi : M_\br \to V^\tree_\br$ are sometimes
odd-dimensional, the quotient map $\pi$ cannot be algebraic
even when $\br$ is integral.
\end{remark}

\subsection{$\tree$-congruence of imploded spin-framed polygons}

In this section we introduce a generalization of the $\tree$-congruence
relation by lifting $\tree$-congruence from $\mathrm{Pol}_ n(\R^3)$ to $P_n(\mathrm{SU}(2))$.
In section 3 we constructed a map $F_n: P_n(\mathrm{SU}(2)) \rightarrow \mathrm{Pol}_ n(\R^3)/\mathrm{SO}(3,\R)$. 
Pulling back the decomposition of $\mathrm{Pol}_ n(\R^3)/\mathrm{SO}(3,\R)$ into $\tree$-
congruence classes
 by $F_n$ produces a decomposition of $P_n(\mathrm{SU}(2))$ by the spaces
we will denote $P_n(\mathrm{SU}(2))^{[S]} =F_n^{-1}(\mathrm{Pol}_ n(\R^3)^{[S]})$.

\begin{lemma}\label{P_nStrat}
Elements of $P_n(\mathrm{SU}(2))^{[S]}$ are the spin-framed polygons
in $P_n(\mathrm{SU}(2))$ such that the following equation holds,

$$\sum \lambda_iAd^*_{g_i}(\varpi_1) = 0$$
where the sum is over all $[g_i, \lambda_i\varpi]\in \mathbf{E}_j(\mathbf{e})$, the
edges in the $j$-th partition defined by $S$, for each $j$.
\end{lemma}

\begin{proof}
This follows from Theorem \ref{framedngons} and Remark \ref{SpinF}.
\end{proof}
By the previous lemma we have a decomposition of $\mathbf{e} \in P_n(\mathrm{SU}(2))^{[S]}$ into imploded spin-framed
polygons

$$\mathbf{E} = \mathbf{E}_1(\mathbf{e} ) \cup \ldots \cup \mathbf{E}_{|S|+1}(\mathbf{e} ).$$
We define an action of $\mathrm{SU}(2)^{D(S)}$ on $P_n(\mathrm{SU}(2))^{[S]}$ by
letting $\mathrm{SU}(2)^{D(C)}$ act diagonally on the framed edges associated
to the component $C$ of $\tree^S$.

\begin{definition}
The   $\mathrm{SU}(2)^{D(S)}$-orbits  on $P_n(\mathrm{SU}(2))^{[S]}$
fit together to give an equivalence relation $\sim_{\mathcal{T}}$,
which we again call $\tree$-congruence, on $P_n(\mathrm{SU}(2))$.
\end{definition}

\begin{definition}
$$V_n^{\tree} = P_n(\mathrm{SU}(2))/ \sim_{\mathcal{T}}$$
\end{definition}
Note that this defines a decomposition on $V_n^{\tree}$ into
subspaces $(V_n^{\tree})^{[S]} = P_n(\mathrm{SU}(2))^{[S]}/\mathrm{SU}(2)^{D(S)}$.
We have seen that  $P_n(\mathrm{SU}(2))$ carries an action of $T_{\mathrm{SU}(2)^n}$ given
by the following formula. Let $\mathbf{t} = (t_1, \cdots, t_n) \in T_{\mathrm{SU}(2)^n}$,

$$ \mathbf{t} \circ ([g_1, \lambda_1\varpi_1, \ldots, [g_n, \lambda_n \varpi_1)   =
([g_1t_1, \lambda_1\varpi_1], \ldots, [g_nt_n, \lambda_n\varpi_1).$$
In particular, note that the action of the diagonal $(-1)$ element in $T_{\mathrm{SU}(2)^n}$ is trivial, because
this corresponds to acting on the left by the nontrivial central $\mathrm{SU}(2)$ element.
Since $t_i$ fixes $\varpi_1$ for $1 \leq i \leq n$  this action does not change the image of a point under
$F_n$, hence it fixes each piece of the decomposition.  Also this action commutes
with the $\mathrm{SU}(2)^{D(S)}$-action on the piece $P_n(\mathrm{SU}(2))^{[S]}$, so
the action of $T_{\mathrm{SU}(2)^n}$ must descend to $V_n^{\tree}$.

\section{The toric varieties $P^{\tree}_{n}(\mathrm{SU}(2))$ and $Q^{\tree}_{n}(\mathrm{SU}(2))$}

In this section we will construct the affine toric variety $P^{\tree}_{n}(\mathrm{SU}(2))$ of imploded triangulated $\mathrm{SU}(2)$--framed $n$--gons (without imposing the condition that the side-lengths are $\br$) and its projective quotient
$Q^{\tree}_{n}(\mathrm{SU}(2))$.

We have tried to follow the notation of \cite{HurtubiseJeffrey} when possible. In \cite{HurtubiseJeffrey} the superscript $D$ on $P^D$ refers to a ``pair of pants'' decomposition $D$ of the surface. For us superscript $\tree$ on $P_n^{\tree}(\mathrm{SU}(2))$ refers to the triangulation of the $n$-gon. The connection is the following. Under the correspondence between moduli spaces of $n$-gons and character varieties briefly explained in Remark \ref{moduliofconnections}
(and explained in detail in \S 5 of \cite{KapovichMillson}) the standard triangulation corresponds to the following ``pair of pants'' decomposition of the $n$ times punctured two-sphere. Represent the sphere as the complex plane with a point at infinity. Take the punctures to be the points $1,2,\cdots,n$ on the real line. Draw small circles around the punctures. Now draw n-3 more circles with centers on the $x$-axis, so that the first circle contains the small circles around $1$ and $2$, the next circle contains the circle just drawn and the small circle around 3 and the last circle contains all the previous circles except the small circles around $n-1$ and $n$. The graph dual to the pair of pants decomposition is the tree $\tree$ that is dual to the triangulation.
Furthermore the decomposed tree $\tree^D$ is dual to the pair-of-pants
decomposition of the $n$-punctured sphere obtained by cutting the sphere
apart along the above $2n-3$ circles - we might say that $\tree^D$ is the
pair-of-pants decomposition of $\tree$.
Using this correspondence the reader should be able to relate what follows with \cite{HurtubiseJeffrey} for the case of the $n$-fold punctured sphere.

It will be important in what follows that we have earlier defined the quotient map
$$\pi^{D}: \tree^D \to \tree$$
that glues together the pairs of vertices that are the boundaries of
the open intervals removed from $\tree$.

\subsection{The space $E^{\tree}_n(\mathrm{SU}(2))$ of imploded framed edges}

We define $E^{\tree}_n(\mathrm{SU}(2))$ to be the product of $\mathcal{E}T^*(\mathrm{SU}(2))$
over the edges of $\tree^D$, hence an element $\mathbf{T} \in
E^{\tree}_n(\mathrm{SU}(2))$  is a map
from the $3(n-2)$ edges of $\tree^D$ into $\mathcal{E}T^*(\mathrm{SU}(2))$ or equivalently
an assignment of an element of $\mathcal{E}T^*(\mathrm{SU}(2))$ to each
each of the $3(n-2)$ edges of the triangles $\tau_i,1 \leq n-2$ in the triangulation
of $P$.  It will be important later to note that there is a forgetful
map that restricts $\mathbf{T}$ to the distinguished edges of $\tree^D$ to obtain
an element $\mathbf{E}$ of $E_n(\mathrm{SU}(2))$.

It will be convenient to represent
the resulting product $\mathcal{E}T^*\mathrm{SU}(2)^{3n-6} \cong (\C^2)^{3n-6}$ by a
$2$ by $3n-6$ matrix. To do this we will linearly order the $3n-6$ edges by
first ordering the $n-2$ triangles (tripods) and then ordering the $3$ edges for
each triangle (tripod). Now use the lexicographic ordering on pairs $(triangle,edge)$
($(tripod,edge)$). The point is that in the matrix $A^{\tree}$ below the $(\C^2)$'s belonging
to the same triangle (tripod) appear consecutively. Let $(\tau_i, j)$ denote the
$j$-th edge of the $i$-th tripod.

We will use the following notation for the entries in the matrix $A^{\tree}$ to indicate
that each successive group of $3$ columns of $A^{\tree}$ belongs to a triangle in the
triangulation. 
$$A^{\tree} =
\begin{pmatrix}
z_1(\tau_1) & z_2(\tau_1) & z_3(\tau_1) & z_1(\tau_2) & z_2(\tau_2) & z_3(\tau_2) & \cdots & z_1(\tau_{n-2}) & z_2(\tau_{n-2}) & z_3(\tau_{n-2}) \\
w_1(\tau_1) & w_2(\tau_1) & w_3(\tau_1) & w_1(\tau_2) & w_2(\tau_2) & w_3(\tau_2) & \cdots & w_1(\tau_{n-2}) & w_2(\tau_{n-2}) & w_3(\tau_{n-2})
\end{pmatrix}$$
In what follows we will abbreviate $\mathcal{E}T^*(\mathrm{SU}(2))^{3n-6}$
to $E^{\tree}_n(\mathrm{SU}(2))$.

\subsection{The space $X^{\tree}_n(\mathrm{SU}(2))$ of imploded framed triangles}

The action of the group $\mathrm{SU}(2)^{3n-6}$ on $\mathcal{E}T^*(\mathrm{SU}(2))^{3n-6}$ is then
represented by acting on the columns of $A^{\tree}$. We will represent elements $\bg$ of
$\mathrm{SU}(2)^{3n-6}$ as $3n-6$-tuples

$$\mathbf{g} = (f_1(\tau_1),f_2(\tau_1),f_3(\tau_1)|\cdots|f_1(\tau_{n-2}),,f_2(\tau_{n-2}),
f_3(\tau_{n-2})).$$
We let $\mathrm{SU}(2)^{n-2}$ denote the ``diagonal'' subgroup of $\mathrm{SU}(2)^{3n-6}$ defined
by the condition that for each triangle $T_i$ (tripod $\tau_i$) we have

$$f_1(\tau_i) = f_2(\tau_i) = f_3(\tau_i).$$
We will regard $\mathrm{SU}(2)^{n-2}$ as the space of mappings $\mathbf{f}$ from the tripods in $\tree^D$ to $\mathrm{SU}(2)$.

\begin{definition}
We then define the space $X^{\tree}_n(\mathrm{SU}(2))$ as the symplectic quotient
$$ X^{\tree}_n(\mathrm{SU}(2)) = \mathrm{SU}(2)^{n-2}\ql (\mathcal{E}T^*\mathrm{SU}(2))^{3n-6}.$$
\end{definition}

Thus $X^{\tree}_n(\mathrm{SU}(2))$ is obtained by taking the symplectic quotient of each
triple of edges belonging to one of the triangles $T_i, 1 \leq i \leq n-2$, by
the group $\mathrm{SU}(2)^{n-2}$. The resulting space $X_n(\mathrm{SU}(2))$ is clearly the product of $n-2$
copies of $(P_3(\mathrm{SU}(2)))^{n-2} \cong (\bigwedge^{2}(\C^3))^{n-2}$.\\

We will often denote an element of $X^{\tree}_n(\mathrm{SU}(2))$ by $\mathbf{T}=([T_1],[T_2],\cdots,[T_{n-2}])$
where $[T_i]$, the $i$-th triangle denotes the equivalence class of the matrix

$A_{\tau_i} =
\begin{pmatrix}
z_1(\tau_i) & z_2(\tau_i) & z_3(\tau_i)\\
w_1(\tau_i) & w_2(\tau_i) & w_3(\tau_i)
\end{pmatrix}$
such that the momentum image of $A_i$ under the momentum map for $\mathrm{SU}(2)$ is zero
(equivalently the rows are orthogonal with the same length).

\subsection{The affine toric variety $P_n^{\tree}(\mathrm{SU}(2))$}\label{mastertoricspace}

In equation (\ref{P3}) we have seen that
$$P_3(\mathrm{SU}(2)) \cong \mathrm{\bigwedge} \space^2(\C^3).$$
Since the space of imploded triangles is the product of
$n-2$ copies of $P_3(\mathrm{SU}(2))$ we see that
$X^{\tree}_n(\mathrm{SU}(2))$ is the affine space obtained by taking the product of $n-2$ copies
of $\bigwedge^2(\mathbb{C}^3)$. It has an action of a $3n-6$ torus $\mathbb{T}$ induced
by the right actions of the torus $\mathbb{T} = T _{\mathrm{SU}(2)}^{3n-6}$ on the $3n-6$ copies
of $\mathcal{E}T^*(\mathrm{SU}(2))$. In terms of our matrices $A^{\tree}$ this 
amounts to scaling the
columns of $A^{\tree}$ (right multiplication of $A^{\tree}$ by $\mathbb{T}$).
However note that the  entry of an element of $\mathbb{T}$ corresponding to
the edge $e$ of $\tree^D$
scales the column of $A^{\tree}$ corresponding to $e$ by its {\it inverse}.
We will use the notation $\underline{\mathbb{T}}$ to indicate the complexification of 
$\mathbb{T}$.  Similarly for any compact group $G$ that appears in this paper, 
the notation $\underline{G}$ indicates the complexification of $G$.

Let $\mathbb{T}_e$, $\T_d$, $\T_d^-$ be as in the introduction.
Now we glue the diagonals of $P$ back together by taking the
symplectic quotient by $\mathbb{T}_d^-$ at level $0$.

\begin{definition}
$$P_n^{\tree}(\mathrm{SU}(2)) = X_n^{\tree}(\mathrm{SU}(2)//\mathbb{T}_d^- = (\mathrm{\bigwedge} \space^2(\C^3))^{n-2}//\mathbb{T}_d^-.$$
\end{definition}

\subsection{The space $P_n^{\tree}(\mathrm{SO}(3,\R))$}\label{orthogonalanalogues}
The spaces $E_n^{\tree}(\mathrm{SU}(2)), X_n^{\tree}(\mathrm{SU}(2))$ and $P_n^{\tree}(\mathrm{SU}(2))$ all have analogues when $\mathrm{SU}(2))$ is replaced by
$\mathrm{SO}(3,\R)$ which are quotients
by a finite group of the corresponding spaces for $\mathrm{SU}(2)$. There
also analogues of the tori $\mathbb{T},\mathbb{T}_e$ and $\mathbb{T}_d^-$
for $\mathrm{SO}(3,\R)$ which we will denote $\mathbb{T}(\mathrm{SO}(3,\R)),\mathbb{T}_e(\mathrm{SO}(3,\R))$ and $\mathbb{T}_d^-(\mathrm{SO}(3,\R))$ respectively which are quotients of the
corresponding tori for $\mathrm{SU}(2)$.  We leave the details to the reader.

\subsection{$P_n^{\tree}(\mathrm{SU}(2))$ as a GIT quotient}\label{firsttorusquotient}
Since affine GIT quotients coincide with symplectic quotients, see  \cite{KempfNess} and  \cite{Schwarz},
Theorem 4.2,  we may also obtain $P_n^{\tree}(\mathrm{SU}(2))$ as the GIT quotient of $(\bigwedge^2(\C^3))^{n-2}$
by the complex torus $\underline{\T}^d$. For each triangle $T_k$ (tripod $\tau_k$) we have a corresponding
$\bigwedge^2(\C^3)$ with Pl\"ucker coordinates $Z_{ij}(\tau_k), 1 \leq i,j \leq 3$. The coordinate
ring of the affine variety $P_n^{\tree}(\mathrm{SU}(2))$ will be the ring of invariants
$\C[(Z_{ij}(\tau_k))]^{\underline{\T}^-_d}$. This ring of invariants will be spanned by the ring of
invariant {\em monomials} which we now determine. There is a technical problem
here. We need to know that we have chosen the correctly normalized
momentum map for the action $\mathbb{T}^-_d$. But by Theorem \ref{twistandshift}
the correct normalization is the one that is homogeneous linear in the
squares of the norms of the coordinates which is the one we have used here.

\subsection{The semigroup $\mathcal{P}^{\tree}_n$}

As a geometric quotient of affine space by a torus the space
$P_n^{\tree}(\mathrm{SU}(2))$ is an affine toric variety. We now analyze the
affine coordinate ring of $P_n^{\tree}(\mathrm{SU}(2))$.   In what follows we label the leaf of the tripod
$\tau_i$ incident to the edge $(\tau_i,k)$ by $k$ - we will do this only when the tripod of which
$k$ is a vertex is clearly indicated. We leave the proof of the following lemma to the reader.

\begin{lemma}
The monomial $f(Z)=
\prod_{i=1}^{n-2}Z_{12}(\tau_i)^{x_{12}(\tau_i)}Z_{13}(\tau_i)^{x_{13}(\tau_i)}
Z_{23}(\tau_i)^{x_{23}(\tau_i)}$ is $\underline{\T}_d^-$--invariant if and only if the
exponents $\mathbf{x} = (x_{jk}(\tau_i))$ satisfy the system of equations

$$x_{k,m}(\tau_i) + x_{k,l}(\tau_i) = x_{l,k}(\tau_j) + x_{l,m}(\tau_j), \quad \text{for } \ (\tau_i, k) \ \text{identified to } (\tau_j, l) \ \text{in $\tree$}.$$
\end{lemma}
Before stating a corollary of the lemma we need a definition.

\begin{definition}
Let $\mathcal{P}^{\tree}_n$ be the subset of $\mathbf{x} \in (\N^3)^{(n-2)}$ satisfying the equations
in the above lemma.  Clearly $\mathcal{P}^{\tree}_n$ is a semigroup under addition.
Let $\C[\mathcal{P}_n^{\tree}]$ denote the associated semigroup algebra.
\end{definition}

\begin{corollary}
The affine coordinate ring of
$P_n^{\tree}(\mathrm{SU}(2))$ is isomorphic to the semigroup ring $\C[\mathcal{P}_n^{\tree}]$.
\end{corollary}


Now we will relate the semigroup $\mathcal{P}^{\tree}_n$ and the monomials of the lemma to graphs on vertices of the decomposed tree $\tree^D$. Monomials in the Pl\"ucker coordinates $Z_{ij}(\tau)$ for $\tau$ fixed  correspond to graphs on the vertices $i,j,k$  of the tripod $\tau$ as follows. We associate to  the exponent $x_{ij}(\tau)$ of $Z_{ij}(\tau)$    the graph consisting of $x_{ij}(\tau)$ arcs joining the vertex  $i$ of $\tau$ to the vertex $j$. Each triple $\{x_{ij}(\tau)\}$ for $\tau$ fixed determines a graph on the leaves of the tripod $\tau$.
Thus each element $\mathbf{x} \in \mathcal{P}_n^{\tree}$ corresponds to a collection of $n-2$ graphs,
each on 3 vertices, one for each tripod in  $\tree^D$. Also each such element
$\mathbf{x}$  corresponds to a product of monomials attached to tripods. This leads to a bijective correspondence
between monomials and graphs.

\begin{figure}[htbp]
\centering
\includegraphics[scale = 0.4]{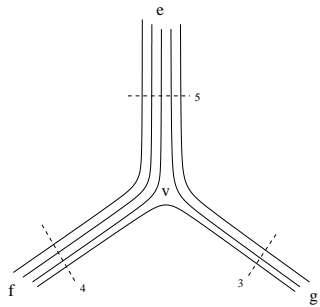}
\caption{A single tripod $\tau$ with vertex $v$, with $x_{ef}(\tau) = 3$, $x_{eg}(\tau)=2$, $x_{fg}(\tau)=1$.  Hence $x_{ef}(\tau)+x_{eg}(\tau)=5$, $x_{ef}(\tau)+x_{fg}(\tau)=4$, $x_{eg}(\tau)+x_{fg}(\tau)=3$. }
\label{fig:tripod}
\end{figure}
The condition

$$x_{km}(\tau_i) + x_{k\ell}(\tau_i) = x_{\ell,k}(\tau_j) + x_{\ell,m}(\tau_j)$$
specifies that the number of arcs associated to the identified  edges $(\tau_i, k)$ of
the tripod $\tau_i$, and $(\tau_j,\ell)$ of the tripod $\tau_j$ must agree.  Later, this will allow
us to glue these ``local'' graphs to obtain a ``global'' graph
on $n$ vertices.  In the next section we will associate a weighting of $\tree$ to the $\{x_{ij}(\tau_k)\}$.

\subsection{The projective toric variety $Q_n^{\tree}(\mathrm{SU}(2))$}

We define $Q_n^{\tree}(\mathrm{SU}(2))$, a projectivization of $P_n^{\tree}(\mathrm{SU}(2))$,
which will be isomorphic to $\mathrm{Gr}_2(\C^n)^\tree_0$. The coordinate ring of the affine
toric variety $P_n^{\tree}(\mathrm{SU}(2))$ is the semigroup algebra $\C[\mathcal{P}_n^{\tree}]$.
Hence to define the  $\C^*$ action required
for projectivization, it suffices to define a grading on $\mathcal{P}^{\tree}_n$.
First, for $\mathbf{x} \in P^{\tree}$ we define the associated weight of the
edge $(\tau,i)$  of the tripod $\tau$ in
$\tree^D$.
let $w_{(\tau,i)}(\mathbf{x} ) = x_{ij}(\tau) + x_{ik}(\tau)$.
We define the degree of an element $\mathbf{x} \in \mathcal{P}_n^{\tree}$  by

$$degree(\mathbf{x}) = \frac{1}{2} \sum_{(\tau,i) \in dist(\tree^D)} w_{(\tau, i)}(\mathbf{x})$$
where the sum is over all $(\tau, i)$ which represent leaf edges of $\tree$.

\begin{remark}
We will see below that the under the isomorphism from the semigroup $\mathcal{P}_n^{\tree}$
to the semigroups of admissible weightings $\mathcal{W}_n^{\tree}$  the number $w_{(\tau,i)}(\mathbf{x})$ will be the weight assigned to the edge $(\tau,i)$
of the tree $\tree$. Thus the degree of $\mathbf{x}$ is then half the sum of the weights of the leaf edges.
\end{remark}

\begin{definition}
Give $\mathcal{P}^{\tree}_n$ the grading defined above, then we define
$$Q_n^{\tree}(\mathrm{SU}(2)) = \mathrm{Projm}(\C[\mathcal{P}_n^{\tree}])$$
\end{definition}
Note that as defined $Q_n^{\tree}(\mathrm{SU}(2))$ is a GIT $\C^*$-quotient of $P_n^{\tree}(\mathrm{SU}(2))$.

\section{The toric varieties $Q^{\tree}_n(\mathrm{SU}(2))$ and $\mathrm{Gr}_2(\C^n)^{\tree}_0$ are isomorphic}

In this section we will prove the first statement of Theorem \ref{firstauxiliarytheorem} namely that the toric variety $Q_n^{\tree}(\mathrm{SU}(2))$ constructed in
the previous section is the isomorphic to $\mathrm{Gr}_2(\C^n)^{\tree}_0$ by proving
Proposition \label{threeisomorphicsemigroups} below.

Recall that in Section 4 it was shown that the two semigroups $\mathcal{S}^{\tree}_n$ and $\mathcal{W}^{\tree}_n$ are isomorphic.
$\mathcal{S}^{\tree}_n$ is the semigroup of Kempe Graphs on $n = |edge(\tree)|$ vertices, with the product $\ast_{\tree}$.  $\mathcal{W}^{\tree}_n$ is the semigroup of admissible weightings on $\tree$ under addition.  Recall that an admissible weighting is an integer weight satisfying the triangle inequalities about each internal vertex of $\tree$, along with the condition that the sum of the weights about each internal node be even. It was also shown in section 4 that the semigroup algebra of $\mathcal{S}_n^{\tree} \cong \mathcal{W}_n^{\tree}$ is isomorphic to the coordinate ring of $\mathrm{Gr}_2(\C^n)_0^{\tree}$.

\begin{proposition}\label{twoisomorphicsemigroups}
The graded semigroups $\mathcal{W}^{\tree}_n$ and $\mathcal{P}^{\tree}_n$
are isomorphic.
\end{proposition}

\begin{proof}
In all that follows we deal with tripods, so we give the unique leaf edge incident to the leaf labelled $i$ the label $i$, and vice-versa.
With this in mind, $X_{ij}$ can be thought of as giving the number of arcs in a graph on the
leaves of a tripod $Y$ between the $i$-th and $j$-th leaves.  The number $N_i$ is a natural number
assigned to the $i$-th edge of a tripod $Y$ obtained by counting the number of arcs which have unique
path in $Y$ containing the $i$-th edge.
The elements of both  $\mathcal{P}^{\tree}_n$ and $\mathcal{W}^{\tree}_n$ both associate triples of
integers to each tripod of ${\tree^D}$ with certain ``gluing conditions''. Let us consider a single tripod $Y \in \tree^D$.  A triple of numbers $N_1, N_2, N_3$ is an admissible weighting of $Y$ if and only if there are integers $X_{ij}$, $i, j \in \{1, 2, 3\}$ such that $X_{ij} + X_{ik} = N_i$.  To see this simply note that the equations
$$X_{ij} = \frac{N_i + N_j - N_k}{2}$$
have natural solutions if and only if $(N_1, N_2, N_3)$ is admissible.
Therefore we may define a map from $\mathcal{W}^{\tree}_n$ to $\mathcal{P}^{\tree}_n$ by solving for $x_{ij}(\tau)$, with an obvious inverse given by solving for the weighting on the edge $(\tau, i)$.  To see that these maps are well defined, note that the gluing condition
$$x_{km}(\tau_i) + x_{k\ell}(\tau_i) = x_{\ell,k}(\tau_j) + x_{\ell,m}(\tau_j), \quad \text{for } (\tau_i, k) \text{ identified to } (\tau_j, \ell) \text{$\in \tree$}.$$
is exactly equivalent to the weights on $(\tau_i, n)$ and $(\tau_j, l)$ being equal when these tuples represent the same diagonal in $\tree$.  Since both semigroup operations are defined by adding integers, and since both $\Phi$ and its inverse are linear functions over each trinode, these maps are semigroup isomorphisms.
Finally, note that the grading on $\mathcal{P}^{\tree}_n$ was chosen specifically to match the grading on $\mathcal{S}^{\tree}_n$, we leave
direct verification of this to the reader.
\end{proof}

\subsubsection{An explicit description of the ring isomorphism}
 The semigroup $\mathcal{S}^{\tree}_n$ is generated as a graded semigroup by the elements corresponding to
graphs with exactly one edge.  The element corresponding to the graph with one edge, between the $i$ and $j$ vertices corresponds to the Pl\"ucker coordinate $Z_{ij}$.  Recall that the unique path in the tree $\tree$ joining
the leafs $i$ and $j$ has been denoted $\gamma(ij)$. The path $\gamma(ij)$
gives rise to a sequence of edges of the forest $\tree^D$ (where we pass from
one tripod to the next by passing from an edge $(\tau_i, k)$ to an equivalent
edge $(\tau_j,\ell)$. We let $Z_{\gamma(ij)}$ be the corresponding
product of Pl\"ucker coordinates $Z_{st}(\tau_k)$.  Thus
$$Z_{\gamma(ij)} = \prod Z_{st}(\tau)$$
where the $(\tau, s)$ and $(\tau, t)$ are the edges in the unique path
defined by the path $\gamma(i,j)$ corresponding to $Z_{ij}$ in $\tree$.
Note that $Z_{\gamma(ij)}$ is a $\underline{\T}_d^-$-invariant monomial
in the homogeneous coordinate ring of $(\bigwedge^2(\C^3))^{n-2}$ -
moreover it is a generator of the ring of invariants.
The isomorphism of toric rings $\Phi$ from the homogeneous
coordinate ring of $\mathrm{Gr}_2(\C^n)^{\tree}_0$ to $P^{\tree}_n(\mathrm{SU}(2)$ is then given on generators
by
$$\Phi(Z_{ij}) = Z_{\gamma(ij)}.$$
\begin{remark}
It is important to see that the degree assigned to $Z_{\gamma(ij)} = \prod Z_{st}(\tau)$ by the isomorphism $\Phi$ (namely {\it one}) is
different from that given by counting the $Z_{st}(\tau)$ in the product formula
for $Z_{\gamma(i,j)}$. The latter
count is in fact the Speyer-Sturmfels weight $w^{\tree}_{i,j}$ of the Pl\"ucker coordinate $Z_{ij}$.
\end{remark}

\begin{figure}[htbp]
\centering
\includegraphics[scale = 0.3]{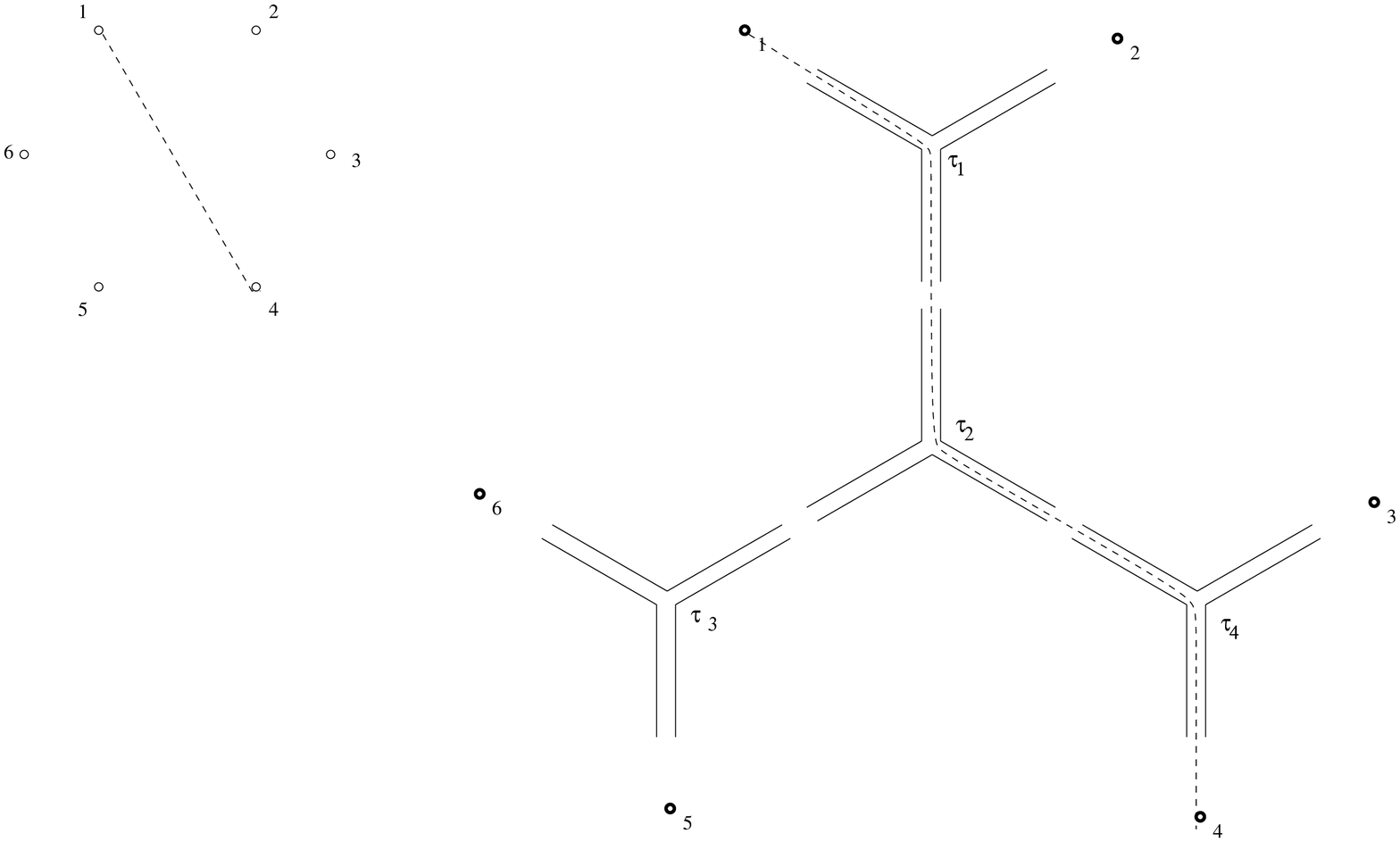}
\caption{This illustrates $\Phi(Z_{14}) = Z_{\gamma(14)} = Z_{13}(\tau_1)Z_{12}(\tau_2)Z_{13}(\tau_4)$ in 
the case $n=6$ with the symmetric tree.}\label{fig:tripod2}
\end{figure}

\subsubsection{The grading circle action on $P_n^{\tree}(\mathrm{SU}(2))$}

We verify the second statement of Theorem \ref{firstauxiliarytheorem}.
We have seen that the action of $\lambda \in \C^*$ that gives the grading
scales each $Z_{\gamma(i,j)}$ by $\lambda$.
Clearly this action is induced by the action on the matrix $A^{\tree}$
that scales each row corresponding to a leaf edge by $\sqrt{\lambda}$,
in other words by the action of the element $\mathbf{t}_e((\sqrt{\lambda})^{-1})$ as claimed in
the second statement of Theorem \ref{firstauxiliarytheorem}.
We conclude by explaining why the actions of $\mathbf{t}_e(\sqrt{\lambda})$
and $\mathbf{t}_e(-\sqrt{\lambda})$ coincide. It suffices to prove that
$\mathbf{t}_e(-1)$ acts trivially. But since the operation
that scales all rows of $A^{\tree}$ by $-1$ is induced by
an element of $\mathrm{SU}(2)^{n-2}$ it acts trivially on $X_n^{\tree}(\mathrm{SU}(2))$
and hence on $P_n^{\tree}(\mathrm{SU}(2))$. But also the operation of scaling all
the rows of $A^{\tree}$ that belong to nonleaf edges of $\tree^D$
is induced by an element of $\underline{\T}_d^{-}$ hence  this
operation too is trivial on $P_n^{\tree}(\mathrm{SU}(2))$. But $\mathbf{t}_e(-1)$
is the composition of the two operations just proved to be trivial.

\section{The spaces $W^{\tree}_n$ and $Q_n^{\tree}(\mathrm{SU}(2))$ are homeomorphic.}\label{homeomorphismVtoP}

We now give the proof of Theorem \ref{firstmaintheorem}.
We first note that we have identified an ordered subset of edges of $\tree^D$
with the leaf edges of $\tree$ (equivalently the edges of $P$). This
identification gives an isomorphism   $\rho:T_{\mathrm{SU}(2)^n} \to \mathbb{T}_e$.

\subsection{A homeomorphism of affine varieties}\label{secondtorusquotient}
We will  prove the following.

\begin{theorem}\label{noncompact}
The spaces $V^{\tree}_n$ and $P^{\tree}_n(\mathrm{SU}(2))$ are equivariantly homeomorphic with respect to $\rho$.
\end{theorem}
We will use the symbol
$\Psi^{\tree}_n$ to denote the above equivariant homeomorphism
and $\Phi^{\tree}_n$ to denote its inverse.
By the results of subsection \ref{firsttorusquotient} it suffices to
construct a $\rho$-equivariant homeomorphism (and its inverse) from $V^{\tree}_n$
to the above symplectic quotient which we will continue to denote
$P^{\tree}_n(\mathrm{SU}(2))$.  
We first construct the  $\rho$-equivariant map 
$\Phi^{\tree}_n : P_n^{\tree}(\mathrm{SU}(2)) \to V^{\tree}_n$.
There is a simple idea behind this map. We have an inclusion of
the leaf edges of $\tree$ into the edges of $\tree^D$.
Now an element of $E_n^{\tree}(\mathrm{SU}(2))$ is a map $\mathbf{T}$ from the edges
of $\tree^D$ into $\C^2$. The map $\Phi^{\tree}_n$ is induced by the map
$\widetilde{\Phi}^{\tree}_n:E_n^{\tree}(SU(2)) \to E_n(\mathrm{SU}(2))$  that
restricts $\mathbf{T}$ to the leaf edges of $\tree$.  However we need
to verify that the image of an element of momentum level zero
for $\mathrm{SU}(2)^{n-2} \times \mathbb{T}_d^-$ has $\mathrm{SU}(2)$--momentum level zero,
and that the induced map of zero momentum levels descends  to the required quotients.
For the rest of this discussion, let $\textbf{T} = (F_1(\tau_1), \dots, F_{n-2}(\tau_{n-2}))$ be an 
element of $E_n^{\mathcal{T}}(\mathrm{SU}(2))$.  Each $F_i$ is an
imploded spin framing of the edges of the triangle $\tau_i$.
This means that
$$F_i = [(g_1(i), \lambda_1(i)\varpi_1), (g_2(i), \lambda_2(i)\varpi_1), (g_3(i), \lambda_3(i)\varpi_1)]$$
such that
$$\lambda_1(i)g_1(i)\varpi_1 + \lambda_2(i)g_2(i)\varpi_1 + \lambda_3(i)g_3(i)\varpi_1 = 0$$
We will henceforth denote $T_i = F_i(\tau_i)$, thus the symbol
$T_i$ stands for a triangle together with an imploded spin-frame on its
edges.

\subsubsection{$\varrho$ flips and normalized framings}\label{normalizedframings}

Let $\varrho = \begin{pmatrix} 0 & 1 \\ -1 & 0\end{pmatrix}$.  Note that if
$t \in \mathrm{SU}(2)$ fixes $\varpi_1$ by conjugation, then $t$ is diagonal and
$\varrho t \varrho^{-1} = t^{-1}$. Furthermore, $\varrho \varpi_1 \varrho^{-1} = - \varpi_1$.
Let $[[T_1],[T_2],\ldots,[T_{n-2}]$ be an element of $P^{\tree}_n(\mathrm{SU}(2))$.

\begin{lemma}\label{rhoflip}
Suppose that the $k$-th edge of $\tau_i$ is identified with the $\ell$-th edge
of $\tau_j$ in $\tree$ and that the edge $(\tau_j, \ell)$ comes after
the edge $(\tau_i, k)$ in the above ordering of edges of $\tree^D$. Then using the left actions of $\mathrm{SU}(2)$ we may arrange
that

$$g_{(\tau_j,\ell)} = g_{(\tau_i,k)}\varrho.$$
\end{lemma}

\begin{proof}
Suppose we have

$$[T_i] = [(a_1, \ell_1 \varpi_1),(a_2, \ell_2\varpi_1), (h, d \varpi_1)],$$
If the diagonal defined by $(\tau_i, \ell)$ and $(\tau_j, k)$ has length $0$
there is nothing to prove since any two frames are equivalent. Suppose then
that this diagonal is not $0$. Then by applying $g =h \varrho (h')^{-1}$ to $T_{j}$,
we get

$$g \cdot T_{j} = [(h \varrho, d \varpi_1),(g a_3,\ell_3 \varpi_1),(g a_4, \ell_4 \varpi_1)],$$
which is equivalent to $T_{j}$.  Hence the consecutive diagonal frames (the third
frame of $T_{i}$ and the first frame of $g \cdot T_{i+1}$) are now related
by right multiplication by $\varrho$.
\end{proof}
We shall henceforth choose representatives in $E_n^{\tree}(\mathrm{SU}(2))$ for
elements of $P^{\tree}_n(\mathrm{SU}(2))$ so that nonzero frames associated to equivalent edges
satisfy this ``$\varrho$ flip condition''.  If  the frame associated to
one edge of a pair of equivalent edges is zero then we require  that
the frame associated to the other edge of the pair is also zero.  We say 
such elements of $P_n^{\tree}(\mathrm{SU}(2))$ are {\it normalized}.
Note that the definition of normalized depends on the ordering
of the triangles $T_i$.

It is important to note that if an element $A=([T_1],[T_2],\cdots,[T_{n-2}])$ is normalized
then so is $t\cdot A$ for any $t \in \mathbb{T}_d^-$. This is because $h\varrho t^{-1} = ht\varrho$. We leave the details
to the reader. Thus we may speak of a $\T_d^-$-equivalence class as being $normalized$.
We will define $\Phi_n^{\tree}$ on such normalized elements.

\begin{lemma}\label{norm}
For any $\textbf{T} \in  E_n^{\mathcal{T}}(\mathrm{SU}(2))$ there exist $\mathbf{f} \in \mathrm{SU}(2)^{n-2}$ and  a normalized $\textbf{T}' \in  E_n^{\mathcal{T}}(\mathrm{SU}(2))$ such that
$$\mathbf{f} \mathbf{T} = \mathbf{T}'.$$
\end{lemma}

\begin{proof}
Let $\tree' \subset \tree$ be connected, and let $Y \subset \tree$
be a tripod which shares exactly one edge, $d$ with $\tree'$.  Let $(\tau, i)$
and $(\tau', j)$ be the edges of $\tree^D$ which map to $d$.
Suppose $\textbf{T} \in E_n^{\tree}(\mathrm{SU}(2))$ is normalized
at all diagonals corresponding to internal edges of $\tree'$.
Let $\mathbf{f} \in \mathrm{SU}(2)^{n-2}$ be the element
which is $g_{(\tau, i)}\varrho g_{(\tau', j)}^{-1}$ on
the $\tau'$-th factor, and the identity elsewhere.
Then $\textbf{T}' = \mathbf{f}\textbf{T}$ is normalized
at all diagonals corresponding to internal edges of $\tree' \cup Y$.
By lemma \ref{trivalentinduction} we can always find a sequence
of connected trees $\tree_i \subset \tree_{i+1}$
with $Y_i = \tree_{i+1}^D \setminus \tree_i^D$ a tripod
such that each internal edge of $\tree$ appears
as an edge shared by the images of $Y_i$ and $\tree_i$
in $\tree$ for some $i$.  The lemma now follows by induction.
\end{proof}

\subsubsection{Definition of the maps $\Phi_n^{\tree}$ and $\Psi_n^{\tree}$}

Given a normalized element $\textbf{T} \in E^{\tree}_n(\mathrm{SU}(2))$
define $\widetilde{\Phi}_n^{\tree}(\textbf{T})$ to be the element in $E_n(\mathrm{SU}(2))$ given by projecting on the components
$(g_{(\tau_i, j)}, \lambda_{(\tau_i, j)}\varpi_1) \in \textbf{T}$ such that $(\tau_i, j) \in \tree^D$
maps to a leaf edge in $\tree$ under $\pi_{\tree}$.

\begin{lemma}
If $\textbf{T} \in \mu_{\mathrm{SU}(2)^{n-2}}^{-1}(0) \cap \mu_{\T_d^-}^{-1}(0)$
and is normalized then $\widetilde{\Phi}_n^{\tree}(\textbf{T}) \in
\widetilde{P}_n(\mathrm{SU}(2))$, that is the polygon in $\R^3$ associated to $\widetilde{\Phi}_n^{\tree}(\textbf{T})$ closes up.
\end{lemma}
\begin{proof}
First  observe   that $\lambda g\varpi_1 = - \lambda g\varrho\varpi_1$, Now
because each triangle closes up the sum of $g \lambda \varpi_1$ over all edges of $\tree^D$ is zero. But by the observation just above the sum over pairs of equivalent edges of $\tree^D$
cancel. Hence the sum over leaf edges of $\tree$
is zero as required.
\end{proof}

\begin{lemma}\label{WDPhi}
$\widetilde{\Phi}_n^{\tree} $ induces a well-defined map on the quotient
$$\widehat{\Phi}_n^{\tree}: \mathrm{SU}(2)^{n-2} \ql (\mu_{\mathrm{SU}(2)^{n-2}}^{-1}(0) \cap \mu_{\T_d^-}^{-1}(0)) \to V_n^{\tree}$$
and a well-defined map
$$\Phi_n^{\tree}:P_n^{\mathcal{T}}(\mathrm{SU}(2)) \to V_n^{\tree}.$$
\end{lemma}

\begin{proof}
To define $\widehat{\Phi}_n^{\tree}$ we need only show that for any pair of  normalized $\textbf{T}$ and $\textbf{T}'$ in $\mu_{\mathrm{SU}(2)^{n-2}}^{-1}(0) \cap \mu_{\T_d^-}^{-1}(0)$ and $\mathbf{f} = (f_1,\cdots, f_{n-2}) \in G$  such that $\mathbf{f} \textbf{T}=  \textbf{T}'$
we have
$$\widetilde{\Phi}^{\tree}_n(\textbf{T}) =  \widetilde{\Phi}_n^{\tree}(\textbf{T}')$$
in $V_n^{\tree}$.

The assumption $\mathbf{f} \textbf{T}=  \textbf{T}'$ is equivalent to

$$(f_ig_{(\tau_i, j)}, \lambda_{(\tau_i, j)}\varpi_1) = (g'_{(\tau_i, j)}, \lambda'_{(\tau_i, j)}\varpi_1)$$
for each $j \in \{1, 2, 3\}$ and $1 \leq i \leq n-2$.  Hence $\lambda_{(\tau_i, j)} = \lambda'_{(\tau_i, j)}$, and in particular a (pair of) diagonals of $\textbf{T}$ vanishes if and only if the same diagonals vanish for $\textbf{T}'$
that is

$$S(\mathbf{T}) = S(\mathbf{T}').$$
Hence the forests $\tree^{S(\mathbf{T})}$ and $\tree^{S(\mathbf{T}')}$ coincide.
We must show that the images of $\widetilde{\Phi}_n^{\tree}(\mathbf{T})$ and $\widetilde{\Phi}_n^{\tree}(\mathbf{T}')$ in $V_n^{\tree}$ coincide, which amounts to proving that the $\mathbf{f}$ is constant on any connected component in the
forest $\tree^{S(\mathbf{T})}$ (recall that we think of $\mathbf{f}$ as a function from
the trivalent vertices of $\tree$ to $\mathrm{SU}(2)$).
Clearly it suffices to prove that if $(\tau_i, j)$ and $(\tau_h, k)$
are equivalent  then

$$f_h = f_i.$$
Suppose the imploded spin framings on the two edges are
respectively
$(g_{(\tau_i, j)}, \lambda_{(\tau_i, j)}\varpi_1)$ and
$(g_{(\tau_h, k)}, \lambda_{(\tau_h, k)}\varpi_1)$.
We have equations

$$f_ig_{(\tau_i,j)} = g_{(\tau_i,j)}'$$
$$f_hg_{(\tau_h,k)} = g_{(\tau_h,k)}'$$
and because we have assumed both $\textbf{T}$ and $\textbf{T}'$ are normalized we have

$$g_{(\tau_h,k)} = g_{(\tau_i,j)}\varrho,$$
$$g_{(\tau_h,k)}' = g_{(\tau_i,j)}'\varrho.$$
We can rearrange these equations to get

$$f_ig_{(\tau_h,k)}\varrho^{-1} = g_{(\tau_h,k)}'\varrho^{-1}.$$
We cancel $\varrho^{-1}$ to obtain $f_ig_{(\tau_h,k)} = g_{(\tau_h,k)}'$.
But from above $f_hg_{(\tau_h,k)} = g_{(\tau_h,k)}'$
This proves the first statement of the lemma.
It is clear that $\widehat{\Phi}_n^{\tree}$ is constant on $\T_d^-$ orbits, so it descends to give $\Phi_n^{\tree}$.
\end{proof}
This shows that $\Phi_n^{\tree}$ is well-defined and continuous
because of the continuity of $\widetilde{\Phi}_n^{\tree}.$ 

\begin{lemma}
$\Phi_n^{\tree} $ is proper (hence closed).
\end{lemma}

\begin{proof}
Suppose $K$ is a compact subset of $P_n(\mathrm{SU}(2))$. Then the lengths of
the columns of any matrix $A$ (so the edge-lengths of the corresponding $n$-gon) representing an element of $K$ are uniformly
bounded by a constant $C$. We may reinterpret the above uniform bound as a bound
on all edge lengths  of all triangles in the given triangulation of $P$ that are
also edges of $P$.
But we may assume all our matrices $A$ are
in the zero level sets of  the momentum maps for $\mathrm{SU}(2)^{n-2}$ and
$\mathbb{T}_d^-$.  It follows by Lemma \ref{trivalentinduction} and the triangle inequalities that the lengths of all columns of any matrix
$A^{\tree}$ representing any element in $P_n^{\tree}(\mathrm{SU}(2))$ in the inverse image
of $K$ are bounded by $N(\tree)C$ where $N(\tree)$ is a positive integer
depending only on the tree $\tree$. Then $(\Phi_n^{\tree})^{-1}(K)$
is contained in the image of a subset of  $E_n^{\tree}(\mathrm{SU}(2))$ homeomorphic to the product of $3(n-2)$ copies of the ball
of radius $N(\tree)C$ in $\C^2$
\end{proof}
We now construct the  map $\Psi_n^{\tree} : V_n^{\tree} \to P_n^{\mathcal{T}}(\mathrm{SU}(2))$ inverse to $\Phi_n^{\tree}$.
We first define $\widetilde{\Psi}_n^{\tree}$ on $E_n(\mathrm{SU}(2))$ then verify that the resulting map descends to $V_n^{\tree}$.
We will need to be able to add a diagonal frame on a triangle with two already framed edges.

\begin{lemma}\label{framing}
Suppose that two sides of an imploded spin-framed triangle $(g_1, \lambda_1 \varpi_1)$, $(g_2, \lambda_2 \varpi_1)$ are given.  Then we can find $g_3 \in \mathrm{SU}(2)$ and $\lambda_3 \in \R$ so that
$(g_1, \lambda_1 \varpi_1)$, $(g_2, \lambda_2 \varpi_1)$ $(g_3, \lambda_3 \varpi_1)$ is an imploded spin-framed triangle.  Precisely we may solve
\begin{equation}\label{closingequation}
\lambda_1g_1\varpi_1 + \lambda_2g_2\varpi_1 + \lambda_3g_3\varpi_1 = 0.
\end{equation}
Moreover $\lambda_3$ is uniquely determined by $\lambda_1$ and $\lambda_2$ and if
$\lambda_3 \neq 0$ any two choices of $g_3$ for given $g_1$ and $g_2$ are related by right multiplication of $T_{\mathrm{SU}(2)}$.
\end{lemma}

\begin{proof}
If both $\lambda_1$ and $\lambda_2$ are zero then we take $\lambda_3$ to be zero and $g_3 = 1$. Suppose then that exactly one of $\lambda_1$ and $\lambda_2$ is nonzero.  Without loss of generality, assume $\lambda_1$ is nonzero. Put $\lambda_3 = \lambda_1$ and choose $g_3 = g_1\varrho$.  Hence we have
$$\lambda_1g_1\varpi_1 = -\lambda_3g_3\varpi_1.$$
Note that the map $S^3 \to S^2$ given by $g \to g\varpi_1$ is the Hopf fibration and consequently we may solve the equation $g\varpi_1 = -g_1\varpi_1$ locally in a neighborhood of $g_1$, in particular we can solve this equation in a neighborhood of $(g_1, 1)$ of $\mathrm{SU}(2) \times \mathrm{SU}(2)$.
The number $\lambda_3 \geq 0$ is uniquely determined  but $g_3$ is defined only up to right multiplication by an element of $\T$.
Now assume that both $\lambda_1$ and $\lambda_2$ are nonzero.  First assume that the sum $\lambda_1g_1\varpi_1 +\lambda_2g_2\varpi_1$ is zero, then we are forced to take $\lambda_3 = 0$ and we may choose any framing data we wish, so we choose $g_3 = 1$. Assume then that $\lambda_1g_1\varpi_1 + \lambda_2g_2\varpi_1$ is nonzero. We choose $g_3 \in \mathrm{SU}(2)$ and $\lambda_3$ so that $$\lambda_1g_1\varpi_1 + \lambda_2g_2\varpi_1 = -\lambda_3g_3\varpi_1.$$
Again, $\lambda_3 \neq 0$ is uniquely determined, and $g_3$ is defined up to right multiplication by an element of $\T$. We define the required framed triangle $T$ by
$$T = (g_1, \lambda_1 \varpi_1), (g_2, \lambda_2 \varpi_1), (g_3, \lambda_3 \varpi_1)$$
\end{proof}
We make use of Lemma \ref{framing} to extend any framing of
the edges of $P$ to an element of $E_n^{\tree}(\mathrm{SU}(2))$. The resulting object will be a normalized element
$\mathbf{T} \in \mu_{\mathrm{SU}(2)^{n-2}}^{-1}(0) \cap \mu_{\T_d^-}^{-1}(0) \subset  E_n^{\mathcal{T}}(\mathrm{SU}(2))$.

\begin{lemma}\label{extend}
Suppose we are given an imploded spin framing $\mathbf{E}$ of the edges of the model convex $n$-gon $P$ which is of momentum level zero for the action
of $\mathrm{SU}(2)$. Then we may extend the framing by choosing
imploded spin-frames on the diagonals and an enumeration of the triangles
of $P$ so that the resulting element $\mathbf{T} \in E_n^{\tree}(\mathrm{SU}(2))$
is
\begin{enumerate}
\item normalized
\item of momentum level zero for $\mathrm{SU}(2)^{n-2}$
\item of momentum level zero for $\mathbb{T}_d ^-$.
\end{enumerate}

Moreover any two such extensions are equivalent under the action of
the torus $\mathbb{T}_d^-$.

\end{lemma}

\begin{proof}
We prove the existence of the framing by induction on $n$. We take as base
case $n=3$, here there is nothing to prove.

Let $P$ be a model convex $n$-gon  with an imploded spin framing on
its edges. Define $P_{i_1} =P$.
Choose a triangle $T$ in the triangulation of $P$ which shares two edges with $P$.
Hence two sides of $T$ are framed. Let $e$ be the remaining side (a diagonal
of $P$). Suppose $e$ has length $\lambda$. Define $T_{i_1} = T$.
Apply Lemma \ref{framing} to frame the third side $e$ such that the resulting
framing is of momentum zero level for $\mathrm{SU}(2)$. Split $T$ off from $P$
to obtain an $n-1$--gon $P_{i_2}$. Suppose the framing on $e$
is $[g, \lambda \varpi_1]$.  Give the edge $e'$ of $P_{i_2}$ which is not yet
framed the framing $[g \varrho, \lambda \varpi_1]$. We obtain the existence
part of the lemma by induction.

Now we prove uniqueness. Suppose that $\mathbf{T}'$ is another extension of $\mathbf{E}$, and
that $\mathbf{T}'$ assigns the frame $[h, \lambda \varpi_1]$ to $e$ with $h \neq g$. 
Then again by Lemma \ref{framing} either $\lambda = 0$ or there exists an element $t \in T_{\mathrm{SU}(2)}$ such that
$h = gt$. In case $\lambda \neq 0$ the frame on $e'$ is necessarily 
$[gt \varrho, \lambda \varpi_1] = [ g \varrho t^{-1}, \lambda \varpi_1]$. Thus the new frames
on the pair of equivalent edges of $\tree^D$ are
$[gt ,\lambda \varpi_1]$ followed by $[g\varrho t^{-1},\lambda \varpi_1]$.
Note that the framing of $P_{i_2}$ obtained by restriction of $\mathbf{T}'$ 
satisfies the three properties in the statement above, and also it is an extension of the framing of its boundary
induced by $\mathbf{T}'$. Hence for the case $\lambda \neq 0$
the induction step of the uniqueness part of the lemma is completed.
In case $\lambda = 0$ then we may represent both frames by
$[I ,\lambda \varpi_1]$. But since $\lambda = 0$ we have (by definition of implosion) for any $t \in
T_{\mathrm{SU}(2)}$
$$[I ,\lambda \varpi_1] = [t ,\lambda \varpi_1] = [t^{-1} ,\lambda \varpi_1].$$
This completes the induction step in the uniqueness part of the lemma.
\end{proof}

By the above lemma we obtain a well-defined map

$$\widetilde{\Psi}_n^{\tree}: E_n(\mathrm{SU}(2)) \to (\mu_{\mathrm{SU}(2)^{n-2}}^{-1}(0) \cap \mu_{\mathbf{T}_d^-}^{-1}(0))\q \T_d^-.$$
We prove this map descends to $V_n^{\tree}$. We let
$\widehat{\Psi}_n^{\tree}$ denote the map obtained from $\widetilde{\Psi}_n^{\tree}$
by postcomposing it with projection to $P_n^{\tree}(\mathrm{SU}(2))$.
Let us note the following obvious refinement of Lemma \ref{extend}.
Suppose $P_1$ is a subpolygon of $P$ which meets $P$ along
diagonals of length zero in the realization of
$P$ given by the framing of the edges. Then given any extension of
the framing of $P$ to all diagonals of $P_1$ we may find an extension of
the framing of $P$ to all diagonals of $P$ agreeing with the given one on $P_1$.

\begin{lemma}
The map $\widehat{\Psi}_n^{\tree}$ descends to
$V_n^{\tree}$.
\end{lemma}

\begin{proof}
We first verify that $\widehat{\Psi}_n^{\tree}$ descends to $P_n(\mathrm{SU}(2))$.
In fact we will show it is equivariant under $\mathrm{SU}(2)$ where the action
of $\mathrm{SU}(2)$ on $E_n^{\tree}(\mathrm{SU}(2))$ is by the diagonal in $\mathrm{SU}(2)^{n-2}$.
Let $f \in \mathrm{SU}(2)$. Let $\mathbf{E}$ be a framing of the
edges of $P$ and $\mathbf{T}$ be the extension of Lemma \ref{extend}.
Apply Lemma \ref{extend} to extend $f \mathbf{F}$ to an
element $\mathbf{T}'$ of $E^{\tree}_n(\mathrm{SU}(2))$
so that this extension also satisfies the three properties in
the statement of Lemma \ref{extend}.  We wish to prove that the resulting element
of  $E_n^{\tree}(\mathrm{SU}(2))$ is in the
$\mathbb{T}_d^-$-orbit of $f\mathbf{T}$. But since any two extensions satisfying the three properties of Lemma \ref{extend} are equivalent under $\mathbb{T}_d^-$
it suffices to prove that the image of {\it any one} of the above extensions is   
equal to $f \mathbf{T}$. But there is an
extension that is obviously equivalent to $f \mathbf{T}$ namely
$f \mathbf{T}$ itself.  

It remains to check that $\widehat{\Psi}_n^{\tree}$ descends to $V_n^{\tree}$.
Let $\mathbf{E}$ and $\mathbf{T}$ be as above. Let $\mathbf{E}'$ be another framing
of the edges of $P$ in the same $\tree$-congruence class as
$\mathbf{E}$. We have to prove that
$$\widehat{\Psi}_n^{\tree}(\mathbf{E}) = \widehat{\Psi}_n^{\tree}(\mathbf{E}').$$
By transitivity of the $\tree$-congruence relation it suffices to prove this
for $\mathbf{E}'$.

Let $C_1$ be a component of
$\tree^{S(\mathbf{T})}$. We may assume that $C_1$ contains leaf edges.  
Apply an element $f \in \mathrm{SU}(2)$ to the
frames of the leaf edges of $\tree_1$ to obtain a new framing $\mathbf{E}'$
of the edges of $P$. We now prove
$$\widehat{\Psi}_n^{\tree}(\mathbf{E}) = \widehat{\Psi}_n^{\tree}(\mathbf{E}')$$
for this choice of $\mathbf{E}'$.
Apply Lemma \ref{extend} to  $\mathbf{E}'$ to obtain an element $\mathbf{T}'$
of  $E_n^{\tree}(\mathrm{SU}(2))$ such that the three properties
in Lemma \ref{extend} are satisfied.  We must prove that $\mathbf{T}$
and $\mathbf{T}'$  have the same image in $P_n^{\tree}(\mathrm{SU}(2))$.
As before in the paragraph above it suffices to find one extension
$\mathbf{T}'$ of  $\mathbf{E}'$ which has the same image as $\mathbf{T}$.
We now construct such an extension $\mathbf{T}''$
The subtree $C_1$ corresponds to a subpolygon $P'$ of $P$ such that
the edges of $P'$ that correspond to diagonals of $P$  have zero length.  Define
$\mathbf{T}''$ by defining the frame on an edge $e$ of $\tree^D$ that
is in $C_1$ to be $f$ applied to the value of $\mathbf{T}$ on $e$
and if $e$ is not in  $C_1$ then define the value of $\mathbf{T}''$
on that edge to be the value of $\mathbf{T}$ on that edge.
We leave to the reader that $\mathbf{T}''$ satisfies the three properties
of Lemma \ref{extend}. Now we construct an element $\tilde{f}  \in \mathrm{SU}(2)^{n-2}$
such that
$$\mathbf{T}'' = \tilde{f}\mathbf{T}.$$
Recall that an element of $\mathrm{SU}(2)^{n-2}$ may be regarded as a function
from the trivalent nodes of $\tree$ to $\mathrm{SU}(2)$.
We define $\tilde{f}$ to be the function that is $I$ if the trivalent node
is not in $C_1$, and is $f$ if the trinode is in $C_1$. Since an edge $e$ of $\tree$ is in $C_1$
if and only if the trivalent node incident to $e$ we do indeed have
$$\mathbf{T}'' = \tilde{f}\mathbf{T}.$$
\end{proof}

Because $\Psi_n^{\tree}$ is the inverse of a closed map we have

\begin{lemma}
$\Psi_n^{\tree}$ is continuous.
\end{lemma}

Finally we have

\begin{lemma}
 $\Phi^{\tree}_n$ and $\Psi^{\tree}_n$ are inverses to one another
\end{lemma}
\begin{proof}
It is obvious that $\Phi^{\tree}_n \circ \Psi^{\tree}_n = I_{P_n(\mathrm{SU}(2))}$.
This is because neither of the two map changes the framing of the edges of $P$.
Now let $\mathbf{T}$ be a normalized element of  $\mathrm{SU}(2)^{n-2} \times \mathbb{T}_d^-$-momentum level zero in $E_n^{\tree}(\mathrm{SU}(2))$. Let $\mathbf{F}$ be the restriction
of $\mathbf{T}$ to the edges of $P$. Thus $\mathbf{T}$ is an extension
of $\mathbf{F}$ satisfying the three properties of Lemma \ref{extend}. But
$\Psi^{\tree}_n \circ \Phi^{\tree}_n(\mathbf{T})$ also has restriction
to $P$ given by $\mathbf{F}$ and also satisfies the three properties of
Lemma \ref{extend}. Hence $\mathbf{T}$ and $\Psi^{\tree}_n \circ \Phi^{\tree}_n(\mathbf{T})$ are in the same $\mathbb{T}_d^-$-orbit and
hence their images in $P_n^{\tree}(\mathrm{SU}(2))$ coincide.
\end{proof}
It is clear from  the definition of $\Phi^{\tree}_n$ that this map intertwines the actions of
$T_{\mathrm{SU}(2)^n}$ and $\T_e$.
This proves Theorem \ref{noncompact}.

\subsection{Pulling back Hamiltonian functions from $P_n^{\tree}(\mathrm{SU}(2))$ to $V_n^{\tree}$}

In the previous subsection we constructed a homeomorphism
$\Psi_n^{\tree}:V_n^{\tree} \to P_n^{\tree}(\mathrm{SU}(2))$ with inverse
$\Phi_n^{\tree}$. The torus $\mathbb{T}$ acts on $P_n^{\tree}(\mathrm{SU}(2))$
such that the Hamiltonian for the circle factor corresponding to the edge $(\tau_i,j)$ of $\tree^D$ is given by $f_{\tau_i,j}(A) = (1/2)|z_j(\tau_i)|^2 + |w_j(\tau_i)|^2$.  In spin-framed coordinates this is $f_{\tau_i,j}(A) = \lambda_{(\tau_i, j)}$.
We will compute the pullback of these functions to $V_n^{\tree}$ via $\Psi_n^{\tree}$.

\begin{definition}
Let $v_{(\tau, i)}^{\tree}: E_n^{\tree}(\mathrm{SU}(2)) \to \mathfrak{su}(2)^*$
by the composition of the projection on the $(\tau, i)$-th factor
of $E_n^{\tree}(\mathrm{SU}(2))$ with the function $h$ from subsection \ref{frommatricestopolygons}.
Let $v_e: E_n(\mathrm{SU}(2)) \to \mathfrak{su}(2)^*$ be the composition of the projection
on the $e$-th factor of $E_n(\mathrm{SU}(2))$ with the function $h$ from subsection \ref{frommatricestopolygons}.
\end{definition}

For a distinguished edge $(\tau, i)$ of $\tree^D$ let $e(\tau, i)$
be associated edge of the model convex planar $n$--gon.  Then
we have the following identity

$$(\tilde{\Phi}_n^{\tree})^*(v_{e(\tau, i)}) = v_{(\tau, i)}^{\tree}.$$

\subsubsection{Pulling back distinguished edge Hamiltonian functions}

In what follows
we will need to compute the pull-back of $f_{(\tau_i,j)}$ to $V_n^{\tree}$
for those distinguished edges  $(\tau_i,j)$.  Let $f$ be the map from subsection
\ref{frommatricestopolygons}, then we have

$$f_{\tau_i, j}(\mathbf{E}) = 2 \| f \circ v_{(\tau_i, j)}^{\tree}(\mathbf{T}) \|$$
for any normalized $\mathbf{T} \in \mu_{\mathrm{SU}(2)^{n-2} \times \T_d^-}^{-1}(0)$ which maps to $\mathbf{E}$.

\begin{proposition}\label{pullback}
For $A \in V_n^{\tree}$ we have
$$(\Psi^{\tree}_n)^* f_{\tau_i,j}(A) = 2 \|f \circ v_{e(\tau_i, j)}(\tilde{A})\|.$$
where $\tilde{A} \in E_n(\mathrm{SU}(2))$ maps to $A$.
\end{proposition}

\begin{proof}
The proposition follows from

$$(\Psi^{\tree}_n)^*(f_{\tau_i, j})(A) = f_{\tau_i, j}(\Psi^{\tree}_n(A)) = 
2 \| f \circ (\Psi^{\tree}_n)^*(v_{(\tau_i, j)}^{\tree})(\tilde{A}) \| = 
2 \| f \circ v_{e(\tau_i, j)}(\tilde{A}) \|.$$
\end{proof}

\begin{remark}
Note that $\| f \circ v_{e(\tau_i, j)}(\tilde{A}) \|$ is the length
of the $e(\tau_i, j)$ edge of $F_n(A)$.
\end{remark}

\subsubsection{Pullback of the internal edge Hamiltonian functions}

We also need to compute the $\Psi_n^{\tree}$ pullbacks
of the Hamiltonian functions $f_{(\tau_j, i)}$ for $(\tau_j, i)$
an internal edge of $\tree$, this will be important when we
wish to indentify the Hamiltonian flows of these pullbacks
with the bending flows. Note that $(\tau_j, i)$ corresponds
to a diagonal $d(\tau_j, i)$ in a model $n$-gon $P$.

\begin{lemma}
Let $\tree' \subset \tree$ be a connected subtree
such that every leaf of $\tree'$ except one, say $(\tau, k)$, is a
leaf of $\tree$. Then for any element $\mathbf{E} \in P_n^{\tree}(\mathrm{SU}(2))$
we can compute

$$f_{\tau, k}(\mathbf{E}) = \| \sum_{(\tau_i, j) \in dist(\tree')} v_{(\tau_i, j)}^{\tree}(\mathbf{T}) \|$$

For $\mathbf{T}$ a normalized element of $\mu_{\mathrm{SU}(2)^{n-2}\times\T_d^-}^{-1}(0)$ which
maps to $\mathbf{E}$.
\end{lemma}

\begin{proof}
First note that we compute $f_{\tau, k}$ by
taking half the length of the $(\tau, k)$ edge for
any normalized element $\mathbf{T} \in \mu_{\mathrm{SU}(2)^{n-2}\times\T_d^-}^{-1}(0)$
which maps to $\mathbf{E}$, this is equal to $\lambda_{(\tau, k)}$.
By the normalized condition and the closing condition imposed
on $\mathbf{T}$, the leaf edge of $\tree'$ form a closed polygon,
which implies that $\lambda_{(\tau, k)} =  \| \sum_{(\tau_i, j) \in dist(\tree')} \lambda_{(\tau_i, j)}(\mathbf{T}) Ad^*_{g_{(\tau_i,j)}(\mathbf{T})}(\varpi_1) \|$.
\end{proof}

The previous proposition allows us to conclude the following
theorem.

\begin{theorem}\label{diagonalpullback}
Let $d(\tau_j, i)(A)$ be the associated diagonal in
$F_n(A)$.

$$(\Psi_n^{\tree})^*(f_{(\tau_j, i)})(A) =  2\| d(\tau_j, i)(A)\|$$

\end{theorem}

\begin{proof}
We have the following equation

$$(\Psi_n^{\tree})^*(f_{(\tau_j, i)})(A) = \| \sum_{(\tau_i, j) \in dist(\tree')} v_{(\tau_i, j)}^{\tree}(\tilde{\Psi}_n^{\tree}(\tilde{A})) \|.$$
The right hand side of this equation is equal to

$$2\| \sum_{(\tau_i, j) \in dist(\tree')} f \circ v_{e(\tau_i, j)}(\tilde{A}) \|.$$
This last expression is in turn equal to $2\|d(\tau_i, j)(A)\|$.
\end{proof}

This, along with Proposition \ref{pullback}, proves Theorem \ref{stratification}.

\subsection{The homeomorphism of projective varieties}\label{thirdtorusquotient}
In this subsection we will descend the homeomorphisms of affine varieties $\Phi_n^{\tree}$
and $\Psi_n^{\tree}$ to their projective quotients completing the proof
of Theorem \ref{firstauxiliarytheorem} and hence the proof of
Theorem \ref{firstmaintheorem}.
Recall that we defined  elements $\mathbf{t}(\lambda) \in T_{\mathrm{SU}(2)^n}$
and $\mathbf{t}_e(\lambda) \in \mathbb{T}_e$ for $\lambda \in S^1$. We note
$$\rho(\mathbf{t}(\lambda)) = \mathbf{t}_e(\lambda).$$
Since $\Psi^{\tree}_n$
is equivariant we have
\begin{equation}\label{degreeequivariance}
\Psi^{\tree}_n \circ \mathbf{t}(\lambda)  = \mathbf{t}_e(\lambda) \circ
\Psi^{\tree}_n.
\end{equation}

\begin{proposition}
The map $\Psi^{\tree}_n$ induces a homeomorphism from $W_n^{\tree}$
to the symplectic quotient of $P_n^{\tree}(\mathrm{SU}(2))$ by the grading
circle action.
\end{proposition}

\begin{proof}
It suffices to prove that  the pull-back by $\Psi_n^{\tree}$ of the grading circle action on $P_n^{\tree}(\mathrm{SU}(2))$
is the grading circle action on $V_n^{\tree}$.
This follows
by replacing $\lambda$
by $(\sqrt{\lambda})^{-1}$ in equation (\ref{degreeequivariance}).
\end{proof}

\subsection{The symplectic and GIT quotients coincide}\label{quotientscoincide}
We complete the proof of Theorems \ref{firstauxiliarytheorem}
and \ref{firstmaintheorem} by proving

\begin{proposition}
The symplectic quotient of $P_n^{\tree}(\mathrm{SU}(2)$ by the grading circle
action coincides with the GIT quotient of $P_n^{\tree}(\mathrm{SU}(2))$ by the
grading $\C^{\ast}$-action linearized by acting by the identity on the
fiber of the trivial complex line bundle over $P_n^{\tree}(\mathrm{SU}(2))$.
\end{proposition}

We are forced to give an indirect argument because $P_n^{\tree}(\mathrm{SU}(2)$
is not smooth and there do not seem to be theorems asserting
the isomorphism of symplectic and GIT quotients of nonsmooth spaces.

We first note that the above GIT quotient is canonically isomorphic  to the GIT quotient of $(\bigwedge^2(\C^3))^{n-2}$ by the product group
$\C^{\ast} \times \underline{\mathbb{T}}_d^-$ where the first factor acts by
the grading action linearized by the identity action on the fiber $\C$
of the trivial complex line bundle over $(\bigwedge^2(\C^3))^{n-2}$.
This follows because we can take the ring of invariants for a product
group acting on a ring $R$ by first taking the invariants of one factor and then taking the invariants of the resulting ring by the second factor.
The corresponding quotient by stages for symplectic quotients is
also true but slightly harder. It is proved in Theorem 4.1 of \cite{SjamaarLerman}.

Thus it remains to prove that the symplectic quotient of $(\bigwedge^2(\C^3))^{n-2}$ by the product of the grading $S^1$-action
and $\mathbb{T}_d^-$ is isomorphic to
the GIT quotient of $(\bigwedge^2(\C^3))^{n-2}$ by the product of
the grading $\C^{\ast}$ action and the complexified torus 
$\underline{\mathbb{T}}_d^-$ where the first factor
acts by the grading action and the product is linearized by acting by
the identity map applied to the projection on the first factor.
This follows immediately from Theorem \ref{twistandshift}
once we establish that the momentum map for the product
of the grading $S^1$-action
and $\mathbb{T}_d^-$ is proper.
We now prove that the momentum map for the product of the grading circle
action and $\mathbb{T}_d^-$ is proper. In fact we prove a different result
that turns out to be equivalent to the one we need here because we will need this different result below.

\begin{proposition}\label{propernesstheorem}
The momentum mapping $\mu : (\bigwedge^2 \C^3)^{n-2} \to \R^{2n-3}$
for the action of $\mathbb{T}_e \times \mathbb{T}_d^-$ is proper.
\end{proposition}

\begin{proof}
Earlier we identified $\bigwedge^2 \C^3$ with the space of
framed triangles $P_3(\mathrm{SU}(2))$. Let
$(\ldots, [g_{(\tau_i, j)}, \lambda_{(\tau_i, j)}\varpi_1], \ldots)$
be an element of $P_3(\mathrm{SU}(2))^{n-2}$.
Under this identification we have that
$\mu^{-1}_{\T_e}(\ldots, [g_{(\tau_i, j)}, \lambda_{(\tau_i, j)}\varpi_1], \ldots)$
is the vector of elements $\lambda_{(\tau_i, j)}$ where $(\tau_i, j)$
is a distinguished edge of $\tree^D$.  Similarly
$\mu_{\T_d^-}(\ldots, [g_{(\tau_i, j)}, \lambda_{(\tau_i, j)}\varpi_1], \ldots)$
is the vector of elements $\lambda_{(\tau_i, j)} - \lambda_{(\tau_k, \ell)}$
where $(\tau_i, j)$ and $(\tau_k, \ell)$ represent the same internal edge of $\tree$.
In order to show that $\mu_{\T_e \times \T_d^-}$ is a proper map, we must show
that if all $\lambda_{(\tau_i, j)}$ are bounded for $(\tau_i, j)$ distinguished,
and all differences $\lambda_{(\tau_i, j)} - \lambda_{(\tau_k, \ell)}$ are bounded
for all $(\tau_i, j)$ and $(\tau_k, \ell)$ which represent the same internal edge of $\tree$,
then all $\lambda_{(\tau_i, j)}$ are bounded.

Let $([g_1, \lambda_1\varpi_1], [g_2, \lambda_2\varpi_1], [g_3, \lambda_3\varpi_1])$
be an element of $P_3(\mathrm{SU}(2))$.  Since
$$\sum_{i = 1}^3 \lambda_i Ad_{g_i}^*(\varpi_1) = 0,$$
the $\lambda_i$ are side-lengths of a closed triangle, hence if two of the three are bounded, so is the third.
Furthermore, if $(\tau_i, j)$ and $(\tau_k, \ell)$ represent the same internal edge of $\tree$,
and $\lambda_{(\tau_i, j)}$ and the difference $\lambda_{(\tau_i, j)} - \lambda_{(\tau_k, \ell)}$ are bounded,
then $\lambda_{(\tau_k, \ell)}$ is bounded.  The proposition now follows from
trivalent induction, see lemma \ref{trivalentinduction}.
\end{proof}

\begin{corollary}
The momentum mapping $\mu : (\bigwedge^2 \C^3)^{n-2} \to \R^{2n-3}$
for the action of $S^1\times \mathbb{T}_d^-$ is proper.
Here $S^1$ acts by the grading circle action.
\end{corollary}
\begin{proof}
The Hamiltonian for the grading circle action is the half the sum of the
Hamiltonians for the circle $\mathbb{T}_e$ factors of $\mathbb{T}_e$.
But note that each of the summands is bounded if and only if the sum
is (since they have the same sign).
\end{proof}

There is one more technical point. We claim that the GIT
quotient coincides with the symplectic quotient at momentum level
$(1,0)$. We will prove first that the level for the
$\mathbb{T}_d^{-1}$ is zero. But this follows
because  the torus  $\mathbb{T}_d^{-1}$ acts linearly on
$\bigwedge^{n-2}(\C^3)$ and trivially on the fiber whereas the grading
circle action is twisted by the identity action on the fiber.

\section{The space $V_{\br}^{\tree}$ is homeomorphic to the
toric variety $(M_{\br})^{\tree}_0$}
In this section we will prove Theorem \ref{secondmaintheorem}
by proving Theorem \ref{secondauxiliarytheorem}.

\subsection{The toric varieties $(M_{\br})^{\tree}_0$ and $P^{\tree}_{\br}(\mathrm{SU}(2))$ are isomorphic}
In this subsection we prove the first statement of Theorem \ref{secondauxiliarytheorem}.
By definition the toric variety $P^{\tree}_{\br}(\mathrm{SU}(2))$ is obtained
as
the GIT quotient of $P^{\tree}_{n}(\mathrm{SU}(2))$ by $\underline{\T}_e$
using the linearization given by the character
$$\chi_{\br}(\mathbf{t}(\mathbf{\lambda})) = \lambda_1^{r_1} \cdots \lambda_n^{r_n}.$$
Here $\mathbf{t}(\mathbf{\lambda})$ denotes the element of $\underline{\T}_e$
corresponding to $\mathbf{\lambda} = (\lambda_1,\cdots,\lambda_n) \in (\C^{\ast})^n$.
We recall that the graded ring
of $(M_{\br})^{\tree}_0$ is the semigroup ring $\C[\mathcal{S}^{\tree}_{\br}]$ where
$\mathcal{S}^{\tree}_{\br}$ is the graded subsemigroup of $\mathcal{S}^{\tree}_n$ defined by taking
graphs with valence $k\br$ for positive integers $k$.
Define the subsemigroup $\mathcal{P}^{\tree}_{\br}$ of $\mathcal{P}^{\tree}_n$
to be the inverse image of  $\mathcal{W}_{\br}^{\tree}$ under the isomorphism
from $\mathcal{P}^{\tree}_n$ to $\mathcal{W}^{\tree}_n$ under the
isomorphism of Proposition \ref{twoisomorphicsemigroups}.

\begin{lemma}
\
The $\underline{\T}_d^-$-invariant monomial
$$f(Z)=
\prod_{i=1}^{n-2}Z_{12}(\tau_i)^{x_{12}(\tau_i)}Z_{13}(\tau_i)^{x_{13}(\tau_i)}
Z_{23}(\tau_i)^{x_{23}(\tau_i)}$$
is $\underline{\T}_e$--invariant for the twist
$\chi_{\br}^p$ if and only if the
exponents $\{x_{jk}(\tau_i)\}$ satisfy the system of equations
$$x_{k,m}(\tau) + x_{k,\ell}(\tau_i) = p r_{(\tau_i,k)}, \quad \text{for all leaf edges} \ (\tau_i, k) \ \text{where} \ k,\ell,m = 1,2,3.$$
By Proposition \ref{twoisomorphicsemigroups} we may rewrite the
above equation as a condition on the leaf weights
$$w_{(\tau_i,k)}^{\tree} = p  r_{(\tau_i,k)}, \ (\tau_i, k) \in \mathcal{L}.$$
\end{lemma}

\begin{proof}
The lemma follows from the formula for $\mathbf{t}(\mathbf{\lambda})$
acting on $f$, namely
$$\mathbf{t}(\mathbf{\lambda}) \circ f(Z) = \prod_{(\tau_i, k) \in \mathcal{L}}
\lambda_{(\tau_i,k)}^{[p r_{(\tau_i,k)} - (x_{k,m}(\tau_i) + x_{k,\ell}(\tau_i))]} f(Z).$$
Here $\lambda_{(\tau_i,k)}$ is the coordinate of $\underline{\T}_e$ corresponding to the
leaf edge $(\tau_i,k)$ and $k,\ell,m = 1,2,3$.
\end{proof}

\begin{remark}
The  set of $\underline{\T}_e$-invariant monomials in $Z_{ij}(\tau)$ {\it for the twist by $\chi_{\br}^p$}
is the subset of invariants {\it of degree $p$} (the $p$-th graded piece of the
associate semigroup of lattice points).
\end{remark}

As an immediate consequence we have
\begin{corollary}\label{rthsemigroup}
The toric variety $P_{\br}^{\tree}(\mathrm{SU}(2))$ is the projective toric variety
associated to the graded semigroup $\mathcal{P}^{\tree}_{\br}$.
\end{corollary}

Now by Proposition \ref{gradedsubsemigroupiso} we have
an isomorphism of graded semigroups  
$\Omega_{\br}: \mathcal{S}_{\br}^{\tree} \to \mathcal{W}_{\br}^{\tree}$.  
Since by definition we have an isomorphism
of graded semigroups $\mathcal{P}^{\tree}_{\br} \cong \mathcal{W}_{\br}^{\tree}$
we obtain the required isomorphism of graded semigroups
$\mathcal{S}^{\tree}_{\br} \cong \mathcal{P}^{\tree}_{\br}$
and we have proved the desired isomorphism of projective toric varieties.

\subsection{$V_{\br}^{\tree}$ is homeomorphic to $P^{\tree}_{\br}(\mathrm{SU}(2))$}
In this section we will prove the second statement of Theorem \ref{secondauxiliarytheorem}.

\begin{proposition}
$V_{\br}^{\tree}$ is homeomorphic to the symplectic quotient
of $P^{\tree}_n(\mathrm{SU}(2))$ by $\mathbb{T}_e$ at level $\br$.
\end{proposition}

\begin{proof}
Recall that in Proposition \ref{pullback} we proved the following formula.
Let  $\mathbf{E} \in V_n^{\tree}$ then

$$(\Psi^{\tree}_n)^* f_{\tau_i,j}(A) =  2\|f \circ v_{e(\tau_i, j)}(\tilde{A})\| = \| v_{e(\tau_i, j)}(\tilde{A}) \|.$$

Where $e(\tau_i, j)$ is the leaf edge of $\tree$
corresponding to the leaf edge $(\tau_i,j)$ of the decomposed tree $\tree^D$.
Thus $\Psi_n^{\tree}$ induces a homeomorphism between the $\tree$-congruence classes
of imploded spin-framed $n$-gons with side-lengths $\br$ and the $\br$-th level set for
the momentum map of $\mathbb{T}_e$. But also $\Psi_n^{\tree}$ is $\rho$-equivariant,
hence the above bijection descends to a homeomorphism from the $T_{\mathrm{SU}(2)^n}$-quotient of the imploded spin-framed $n$-gons
to the  $\mathbb{T}_e$ symplectic quotient at level $\br$.
\end{proof}

It remains to prove that the symplectic quotient of $P_n^{\tree}(\mathrm{SU}(2))$
at level $\br$ coincides with the GIT quotient for the action on the
trivial bundle using the twist by $\chi_{\br}$.

\subsubsection{The symplectic quotient coincides with the GIT quotient}
\label{fourthtorusquotient}
Let $L$ be the trivial complex line bundle over $\bigwedge^{n-2}(\C^3)$.
Then $L$ descends to the trivial complex  line bundle $\overline{L}$
over $P^{\tree}_n(\mathrm{SU}(2))$. The torus
$\underline{\T}_e$ acts on $\overline{L}$ by twisting by
the character $\chi_{\br}$.
We  need to prove that the GIT quotient of
$P^{\tree}_n(\mathrm{SU}(2))$ by $\underline{\T}_e$ with linearization the above
action on $\overline{L}$ is homeomorphic to the symplectic quotient by the maximal compact subgroup of $\underline{\T}_e$ at level $\br$.
The argument is almost  the same as  that of subsection \ref{quotientscoincide}.
In particular we use reduction in stages to reduce to the
corresponding problem for the quotients of $\bigwedge^{n-2}(\C^3)$ by
the product torus $\underline{\T}_e \times \underline{\T}_d^{-}$.
Here we twist the action by the character that is $\chi_{\br}$ on $\underline{\T}_e$ and trivial on
$\underline{\T}_d^{-}$.
Since the momentum map for the action by $\underline{\T}_e \times \underline{\T}_d^{-}$ 
is proper by Proposition \ref{propernesstheorem}
the equality of quotients follows by Theorem \ref{twistandshift}.
The momentum level for the product group is then $(\br,0)$.

\section{The residual action of $\T / (\T_e \times \T_d^-)$ and bending flows}\label{bendingflowsection}

Now we will relate the action of $\T / (\T_e \times \T_d^-)$ on $P_n^{\tree}(\mathrm{SU}(2))$ and $P_\br^\tree(\mathrm{SU}(2))$ to bending flows.

\subsection{A complement to $\mathbb{T}_d^-$ in $\mathbb{T}_d$}
In this subsection we will define a complement $\mathbb{T}_d^+$
to $\mathbb{T}_d^-$ in $\mathbb{T}_d$. It will be more convenient
to work with this complement rather than the quotient
$\mathbb{T}_d/ \mathbb{T}_d^-$.

We recall that an element of $\mathbb{T}$ corresponds to a function
$f$ from the edges of $\tree^D$ to the circle. The subgroup $\mathbb{T}_d$
corresponds to the functions with value the identity on the distinguished
edges of $\tree^D$. The nondistinguished edges of $\tree^d$ occur in
equivalent pairs. The subtorus $\mathbb{T}_d^-$ of $\mathbb{T}_d$
consists of those $f$ whose values on equivalent pairs of
nondistinguished edges are inverse to each other.

In order to construct the complement $\mathbb{T}_d^+$ we need to choose
an edge from each distinguished pair. To do this in a systematic way
we use the ordering of the trinodes of $\tree^D$ induced by
the trivalent induction construction, see Lemma \ref{trivalentinduction}. Suppose the edges
$(\tau, i)$ and $(\tau', j)$ are equivalent. The two edges occur
in different trinodes. We relabel the edges by ${\epsilon}^+$ and
${\epsilon}^-$ by labeling the edge that comes in the first trinode
by (the superscript) plus and the edge that comes in a later trinode
by minus. Thus every nondistinguished edge is either a plus edge
or a minus edge. We now define the subtorus $\mathbb{T}_d^+$ as the subtorus
of $\mathbb{T}_d$ consisting of those functions that take value the identity on 
all nondistinguished minus edges. It is clear that $\mathbb{T}_d^+$ is the required complement.
It is also clear that $\mathbb{T}_d^+$ is a complement to
$\mathbb{T}_e \times \mathbb{T}_d^-$ in $\mathbb{T}$. We point out
here that the choice of where to place the identity
in the definition of this complement is irrelevant with respect
to the Hamiltonian functions - see below.

\subsection{The action of $\T_d^+$ coincides with bending} \hfill

We now prove the following theorem.
\begin{theorem}\label{noncompact2}
\hfill
\begin{enumerate}
\item The homeomorphism $\Psi^{\tree}_n$ intertwines the bending flows
on $V_n^{\tree}$ with the action of $\T_d^+$ on $P_n^{\tree}(\mathrm{SU}(2))$. \item The homeomorphism $\Psi^{\tree}_{\br}$ intertwines the bending flows
on $V_{\br}^{\tree}(\mathrm{SU}(2))$ with the action of $\T_d^+$ on $P_{\br}^{\tree}(\mathrm{SU}(2))$.
\end{enumerate}
\end{theorem}
\begin{proof}

Let $\mathbf{T} \in P_n^{\tree}(\mathrm{SU}(2))$ and $\mathbf{t} \in \T_d^+$.
Let $d$ be a diagonal of the triangulated $n$-gon $P$ and suppose
the triangles $T_i$ and $T_j$ share the diagonal $d$. Suppose
$d$ divides $P$ into two polygons $P'$ and $P''$ with
$T \in P'$ and $T' \in P''$. Let $\tau$ and $\tau'$
be the trinodes associated to $T$ and $T'$ respectively,
and let $[g_{(\tau, i)}, \lambda_{(\tau, i)}\varpi_1]$ and
$[g_{(\tau', j)}, \lambda_{(\tau', j)}\varpi_1]$ be spin-framed
representatives of the edges of $\mathbf{T}$ corresponding to
the diagonal $d$.  Suppose $\epsilon(d)^- = (\tau', j)$.

Since the representatives are normalized we have
$$g_{(\tau, i)} = g_{(\tau', j)} \varrho.$$
Let $t$ be the $\epsilon(d)^+$-th component of $\mathbf{t}$ (the
$\epsilon(d)^-$-th component is $1$). Under the action
of $\mathbf{t}$ the edge $[g_{(\tau, i)}, \lambda_{(\tau, i)}\varpi_1]$
becomes $[g_{(\tau, i)}t,\lambda_{(\tau, i)}\varpi_1]$ and all other
components are unchanged.  Thus the resulting
element of $E_n^{\tree}(\mathrm{SU}(2))$ is no longer normalized.

To normalize this element we multiply all the frames on the edges and diagonals
of $P''$ by $g_{(\tau', j)} t g_{(\tau', j)}^{-1}$. Note that the image of $h=g_{(\tau', j)} t g_{(\tau', j)}^{-1} \in \mathrm{SU}(2)$.
is a rotation about the oriented line in $\R^3$ corresponding to
$Ad^{\ast}h(\varpi_1) \in \mathfrak{su}(2)^{\ast} \cong \R^3$ which is  the diagonal corresponding to $d$ of
the Euclidean $n$-gon $\mathbf{e}$ underlying the imploded
framed $n$-gon $\mathbf{T}$. Thus the Euclidean $n$-gon underlying
the imploded framed $n$-gon $\mathbf{T}$ is bent along the
diagonal $d$. Furthermore the frames  of all the edges
and diagonals of $P''$  are transformed by applying the element of the one-parameter group 
$g_{(\tau', j)} t g_{(\tau', j)}^{-1}$ in $\mathrm{SU}(2)$ which covers
rotation along the diagonal. Applying $\Phi_n^{\tree}$ we forget the
frames along the diagonals but the frames on the edges transform
in the same way as before.  This amounts to bending the framed 
$n$-gon $\Phi_n^{\tree}(\mathbf{T})$ along the diagonal $d$.

The first statement implies the second statement because the
bending flows on $V_{\br}^{\tree}$ are descended from
those on $V_{n}^{\tree}$, and the action of
$\mathbb{T}_d^+$ on $P_{\br}^{\tree}(\mathrm{SU}(2))$ is descended from the action
on $P_{n}^{\tree}(\mathrm{SU}(2))$
\end{proof}

\subsection{The Hamiltonians for the residual action}
It remains to prove that the action of $\mathbb{T}_d^+$ is Hamiltonian
with the given Hamiltonians.


We claim that $\T_d^+$ preserves the orbit-type stratification.
Since $\T = (\T_e \times \T_d^-) \times \T_d^+$ is abelian,
the isotropy subgroup $(\T_e \times \T_d^-)_x \subset \T_e \times \T_d^-$ of
$x \in P_\br^\tree(\mathrm{SU}(2))$ is equal to the isotropy
group $(\T_e \times \T_d^-)_{t \cdot x}$ of $t \cdot x$, for any $t \in \T$; in particular this is true for
$t \in \T_d^+$.  Although the space $P_\br^\tree(\mathrm{SU}(2))$ is possibly
singular, we can work in a given symplectic stratum.
Thus it makes sense to say that $\T_d^+$ acts in a Hamiltonian
fashion (on each individual stratum), and we may identify
the torus action on the whole space
by the Hamiltonians of the $S^1$ factors.
These Hamiltonians are smooth in the sense
of \cite{SjamaarLerman}, since these functions are obtained by
restricting $\T_e \times \T_d^-$ invariant continuous functions on
$(\bigwedge^2 \C^3)^{n-2}$ to individual strata.

Each $\C^2 \subset E_n^\tree(\mathrm{SU}(2))$ is indexed by
an edge of $\tree^D$.
Given an internal edge $\epsilon$,
the factor $(S^1)_{\epsilon}$ of $\T_d^+$
has Hamiltonian function
$\|(z,w)_{\epsilon^-}\|^2/2 = \|(z,w)_{\epsilon^+}\|^2/2$
on any given orbit-type stratum.
The $(z,w)_{\epsilon^-}$ corresponds to the diagonal
$d_{\epsilon}$ of the associated $n$-gon in
$\R^3$.
By Lemma \ref{diagonalpullback} and Lemma \ref{SpinLength}
we have $\|d_{\epsilon}\| = \frac{1}{4} \|(z,w)_{\epsilon^-}\|^2$,
and so $\|d_{\epsilon}\|$ is
\emph{one-half} that of
the Hamiltonian for the $\epsilon$-th factor $(S^1)_\epsilon$ of
$\T_d^+$.
In what follows we will need the following elementary lemma.
\begin{lemma}
Suppose $A: S^1 \times X \to X$ is a Hamiltonian action of
the circle on a stratified symplectic space $X$. Suppose the action
is generated by the Hamiltonian potential $f$ and that the element
$-1 \in S^1$ acts trivially. Let $\overline{A}$ be the induced action
of the quotient circle $\overline{S^1} = S^1/\{\pm 1\}$.
Then the Hamiltonian potential for the $\overline{A}$ action
is $\frac{f}{2}$.
\end{lemma}

\begin{proof}
We have
$$\overline{A}(\exp{\sqrt{-1}t},x) = A(\exp{\sqrt{-1}t/2},x).$$
\end{proof}

The bending flow torus $$T_{bend} = \prod_{\text{$d$ diagonal}} (S^1)_d,$$
acting on $P_\br^\tree(\mathrm{SU}(2))$, is such that
the $d(\epsilon)$-th factor has Hamiltonian $\|d_\epsilon\|$.
Therefore, the $\T_d^+$ action factors through the action of $T_{bend}$ via
a two to one cover on each component of $\T_d^+$.  Thus we may conclude
that $T_{bend}$ and $\T_d^+(\mathrm{SO}(3,\R))$ coincide.

\begin{theorem}
Let $\T_d^+(\mathrm{SO}(3,\R)) = \T_d^+ / (\Z/2\Z)^{n-3}$ be the image of $\T_d^+$ in $\mathrm{SO}(3,\R)^{n-3}$ under the
surjection $\mathrm{SU}(2) \to \mathrm{SO}(3,\R)$.  The action of $\T_d^+(\mathrm{SO}(3,\R))$ on $P_\br^\tree(\mathrm{SU}(2))$
coincides with the bending flows. Precisely the action is Hamiltonian
and the circle factor
corresponding to $d(\epsilon)$ has Hamiltonian potential $\|d_{\epsilon}\|$.
\end{theorem}

\begin{proof}
Let $\pi : \T_d^+ \to \T_d^+(\mathrm{SO}(3,\R))$ be the quotient map.
Let $\epsilon$ be any internal edge of $\tree$. The action of
$(S^1)_\epsilon \subset \T_d^+$ on $P_\br^\tree(\mathrm{SU}(2))$
factors through the action of $\pi((S^1)_\epsilon) \in \T_d^+(\mathrm{SO}(3,\R))$,
through the map $t \mapsto t^2$.  Consequently by the lemma
above the Hamiltonian function
of $(S^1)_\epsilon \subset \T_d^+$ is twice that of the Hamiltonian
of $\pi((S^1)_\epsilon) \subset \T_d^+(\mathrm{SO}(3,\R))$.
Hence the Hamiltonian for $\pi((S^1)_\epsilon) \subset \T_d^+(\mathrm{SO}(3,\R))$ is
$\|d_{\epsilon}\|$, as it is for $(S^1)_{d_{\epsilon}} \subset T_{bend}$.
Since the Hamiltonian functions
of $T_d^+(\mathrm{SO}(3,\R))$ and $T_{bend}$ coincide, their actions must coincide on
$P_\br^\tree(\mathrm{SU}(2))$.
\end{proof}

\section{Appendix : symplectic and GIT quotients of affine space}\
The results in this appendix were obtained with the aid of W. Goldman.

\subsection{Fiber twists and normalizing the momentum map}
\label{fibertwists}
The goal of this appendix is to prove Theorem \ref{twistandshift} below -
we match the level for the symplectic quotient with the
twist used to define the linearization in forming the GIT quotient.
In four places in the paper,
namely subsections \ref{firsttorusquotient},\  
\ref{secondtorusquotient},\ 
\ref{thirdtorusquotient} and subsubsection \ref{fourthtorusquotient} we applied the results of \cite{Sjamaar} to deduce
that for a torus acting on affine space linearized by acting
on the trivial line bundle by a character $\chi$ the GIT
quotient is homeomorphic to the symplectic quotient at level
the derivative of $\chi$ at the identity provided the momentum map was proper.
However there is a technical problem 
about the normalization  of the momentum
map chosen for the
action of the torus 
(there is an indeterminancy of an  additive
constant vector) 
The correct normalization of
the momentum map  must depend on  the action of the torus on the total space of the line
bundle. It is given in formula (2.3), page 116 of \cite{Sjamaar} which
we now state for the convenience of the reader. 

Let  $p:E \to M$ be a Hermitian line bundle $L$. Let $G$ be a compact group
acting on $E$ by automorphisms of the Hermitian structure. 
Let $\xi \in \mathfrak{g}, e \in E$ and $m=p(e)$.   In what follows
$\xi_E$ denotes the vector field on $E$ induced by $\xi$,
$\xi_M^{hor}$ denotes the horizontal lift of the vector
field $\xi_M$ induced on $M$ and $\nu_E$ denotes the canonical vertical
vector field (induced by the $\mathrm{U}(1)$ action). The connection and curvature
forms take values in the Lie algebra $\mathfrak{u}(1) = \sqrt{-1} \R$
of $\mathrm{U}(1)$. We can now state the formula from \cite{Sjamaar}:
\begin{equation} \label{normalization}
\xi_E(e)  = \xi_M^{hor}(e)  + 2\pi <\mu(m),\xi>  \nu_E.
\end{equation}

\begin{definition}
We will say a momentum map satisfying equation (\ref{normalization})
is {\em normalized} relative to the linearization (action of $G$
on the total space of the bundle).
\end{definition}

\begin{remark}
We can check the conventions involved in equation (\ref{normalization})
by applying the connection form $\theta$ to both sides to obtain
$$\theta(\xi_E(e)) = 2\pi \sqrt{-1}<\mu(m),\xi>.$$ 
Applying $d$ to each side and Cartan's formula we find that 
$<\mu(m),\xi>$ is a Hamiltonian potential for $\xi_M $ if and only if
the symplectic form $\omega$ and the connection form $\theta$ are related
by
$$\omega = - \frac{1}{2\pi \sqrt{-1}} d \theta.$$
\end{remark}

Suppose now we twist the action of the torus $G$ on the total
space of the line bundle by scaling each fiber by a fixed character
$\chi$. This changes the invariant sections and hence changes the
GIT quotient. Note that the differential of $\chi$ at the identity
of $G$ is an element $\dot{\chi}$ of $\mathfrak{g}^*$. Thus we could
change the momentum map $\mu$  by adding $\dot{\chi}$ and obtain a
new momentum map. The following lemma is an immediate consequence
of equation (\ref{normalization}).

\begin{lemma}\label{twistshift}
Suppose we twist the action of $G$ by a character $\chi$. Then the
normalized momentum map for the new action is obtained by adding
$\dot{\chi}$.
\end{lemma}

We now restrict to the case of a torus $T$ acting linearly on a
symplectic vector space $V,\omega$. We assume we have chosen a
$T$-invariant complex structure $J$ on $V$ so that $\omega$ is 
of type $(1,1)$ for $J$ (this means $J$ is an isometry of $\omega$) and 
the symmetric form $B$ given by $B(v,v) = \omega(v,Jv)$
is positive definite.
We let $W$ be the subspace of $V \otimes \C$ of type $(1,0)$ vectors,
that is $W=\{ v -\sqrt{-1} Jv: v \in V\}$. We define a 
positive-definite Hermitian form $H$ on $W$ by 
$$H(v_1 - \sqrt{-1}Jv_1, v_2 - \sqrt{-1}Jv_2)= B(v_1,v_2) - \sqrt{-1}
\omega(v_1,v_2).$$
We will abbreviate $\sqrt{H(v,v)}$ to $\|v\|$ in what follows. We define
a symplectic form $A$ on $W$ by $A(w_1,w_2) = - \Im H(w_1,w_2)$.
We note that the map $w \to \Re w$ is a symplectomorphism
from $W,A$ to $V,\omega$.

Suppose that we have chosen an $H$-orthonormal basis for $W$ so we
have identified
$$W \cong \C^n.$$
We let $T_0 \cong (S^1)^n$ be the  compact torus
consisting of the diagonal matrices with unit length elements on
the diagonal and $\underline{T}_0$ be the complexification of $T_0$.  We let $\mathfrak{t}_0^*
\cong \R^n$ be the dual of the Lie algebra of $T_0$.  We define
$\mu_0:W \to \mathfrak{t}_0 ^*$ by
$$\mu_0((z_1,\cdots,z_n)) = (-\frac{|z_1|^2}{2},\cdots,-\frac{|z_n|^2}{2})$$
Then $\mu_0$ is a momentum map for the Hamiltonian action of $T_0$
on $W$.
We note that  all possible momentum maps
are obtained from $\mu_0$ by adding a vector $\mathbf{c} = (c_1,\cdots,c_n)
\in \mathfrak{t}_0^{\ast}$.  Thus $\mu_0$ is the unique momentum map
vanishing at the origin of $W$.  

Let $E= W \times \C$ be total space of the trivial line bundle $L$ over $W$. We first describe the Hermitian holomorphic structure on $L$.
We give $L$ a holomorphic  structure by requiring that the nowhere vanishing section $s_0$ of $E$ defined by
$$s_0(w) = (w,1)$$
is holomorphic. 
Thus if $U$ is an open subset of $W$ and $s$ is a local section over $U$
then $s$ is holomorphic if and only if the function $f$ on $U$ defined by
$$s = f s_0|U$$
is holomorphic.

We  define a Hermitian  structure on $L$ by defining
$$||s_0||(w) = \exp{(-\frac{\pi}{2}\|w\|^2)}.$$ 
Hence the section $\sigma_0$ given by
$$\sigma_0(w) = \exp{(\frac{\pi}{2}\|w\|^2)} s_0(w)$$
has unit length  at every point.
We note that
$$ - \frac{1}{2\pi \sqrt{-1}}\  \overline{\partial}\partial \ \log ||s_0||^2 = \sum _{i=1} ^{n} dx_i \wedge dy_i$$
in agreement with \cite{Sjamaar}, pg. 115.
Let $E_0$ denote the principal circle bundle of unit length vectors in 
$E$.

Our goal is to describe the lift of the action of $T_0$ to $E$
so that the normalized momentum map corresponding to this lifted
action is $\mu_0$. We note there is a distinguished lift given by
$$t \circ(w,z) = (tw,z), t \in T_0, w \in W, z \in \C.$$
By definition this lift leaves invariant the holomorphic section $s_0$
but it also leaves invariant the unit length section $\sigma_0$
because the function $\exp{(-\frac{\pi}{2}\|w\|^2)}$ is invariant
under $\mathrm{U}(n)$ and hence under $T_0$.  Hence $T_0$ leaves fixed the Hermitian structure on $E$
and  $\underline{T}_0$ leaves the holomorphic section $s_0$ fixed.
We will call such lifts {\it untwisted} and we will say the linearization
consisting of the trivial bundle together with the previous action
is untwisted. We will now prove
\begin{proposition}\label{untwistedstandard}
The normalized momentum map of $T_0$ corresponding to the untwisted linearization
is $\mu_0$.
\end{proposition}

The proposition will be a consequence of the next lemma and corollary.

We will compute $\mu_0$ in the trivialization of $E_0$ given by $\sigma_0$.
We note that since $\sigma_0$ is invariant under $T_0$ the untwisted lift of the compact torus above (relative
to the trivialization by $s_0$) remains untwisted in the trivialization
by $\sigma_0$.  We let $\psi$ denote the coordinate in the fiber circle
of $E_0$ so $\nu_{E_0} = \partial / \partial \psi$. Hence if $z_i = x_i + \sqrt{-1}y_i, 1 \leq i \leq n$, then
$x_1,y_1,\cdots,x_n,y_n, \psi$ are coordinates in $E_0$.

\begin{lemma}\label{connectionform}
The canonical connection on $E$ is given by
$$\nabla \sigma_0 = - \pi \sqrt{-1} (\sum_{i=1}^n (x_i dy_i - y_i dx_i) \otimes \sigma_0.$$
Equivalently in the above coordinates the connection form $\theta$ of the canonical connection is given by
$$\theta = \sqrt{-1} d\psi - \pi \sqrt{-1} (\sum_{i=1}^n (x_i dy_i - y_i dx_i).$$
\end{lemma}
\begin{proof}
The reader will verify that $\nabla$ satisfies
$$ \nabla_{\partial / \partial \overline{z}_j} s_0 =0, 1 \leq j \leq n, $$
and has curvature $- 2\pi \sqrt{-1}\omega$.  Hence $\nabla$ is the unique Hermitian connection with
curvature $d \theta = - 2\pi \sqrt{-1}\omega$.
\end{proof}

We then have the  following corollary.

\begin{corollary}
The horizontal lift of the vector field $x_i \partial/\partial y_i -
y_i \partial/\partial x_i$ is $x_i \partial/\partial y_i - y_i
\partial/\partial x_i + \pi (x_i^2 + y_i^2)\partial / \partial \psi$.
\end{corollary}

\begin{proof}
By the formula for $\theta$ we see that the horizontal lift of
$\partial/\partial x_i$ is $\partial/\partial x_i - \pi y_i \partial / \partial \psi
$ and the horizontal lift of $\partial/\partial y_i$ is
$\partial/\partial y_i +  \pi x_i \partial / \partial \psi$. Since the operation
of taking horizontal lifts is linear over the functions the corollary
follows.
\end{proof}

Proposition \ref{untwistedstandard} follows from the corollary and  equation (\ref{normalization}).

Now suppose $T$ is a compact torus with complexification $\underline{T}$ and $T$ acts on $W$ through a
representation 
$\rho:T \to T_0$. Assume further that we linearize the action 
of $\underline{T}$ on $W$ by the untwisted linearization.
It is  clear from equation (\ref{normalization}) that we obtain the {\it normalized} momentum
map $\mu_T$ for $T$ by restricting the {\it normalized} momentum map for $T_0$. More precisely
let $\rho^{\ast} : \mathfrak{t}_0^{\ast} \to \mathfrak{t}^{\ast}$
be the induced map on dual spaces.
Then we have
$$\mu_T = \rho^{\ast} \circ \mu_0.$$
Thus the normalized momentum map corresponding to the untwisted linearization
of the linear action of $T$ is  homogeneous linear  in the squares of the $|z_i|$'s.
We now obtain the desired result in this appendix by
applying Lemma \ref{twistshift} and Theorem 2.18 (page 122) of \cite{Sjamaar}.
 
\begin{theorem}\label{twistandshift}
Let $\underline{T}$ be a complex torus with maximal compact  subtorus $T$ and let $\chi$ be a character of $\mathbb{T}$.
Then the  GIT quotient for a linear action of  $\underline{T}$ on a complex vector space $W$ with  linearization given by $W \times \C$ and the action  
$$t \circ (w,z) = (tw, \chi(t) z)$$
is homeomorphic to the symplectic quotient by 
$T$ obtained using the momentum
map $(\mu_T)_0 + \dot{\chi}$ if the momentum map for the action of $T$
is proper. Here $(\mu_T)_0$ is the momentum map that vanishes at the origin
of $W$.
\end{theorem}

\bigskip

Benjamin Howard:
Department of Mathematics,
University of Michigan,
Ann Arbor, MI 48109, USA,
howardbj@umich.edu

\smallskip

John Millson:
Department of Mathematics,
University of Maryland,
College Park, MD 20742, USA,
jjm@math.umd.edu

\smallskip

Christopher Manon: Department of Mathematics, University of
Maryland, College Park, MD 20742, USA, manonc@math.umd.edu

\end{document}